\documentclass[a4paper,10pt,leqno]{amsart}
        \title{Conformal nets I: coordinate-free nets}     
               \author{Arthur Bartels}
      \address{WWU M\"unster\\
               Mathematisches Institut\\
               Einsteinstr.~62,
               48149 M\"unster, Germany}
        \email{bartelsa@math.uni-muenster.de}
      \urladdr{http://www.math.uni-muenster.de/u/bartelsa}
       \author{Christopher L. Douglas} 
      \address{Mathematical Institute\\ 24Ð29 St Giles'\\ Oxford\\ OX1 3LB, United Kingdom}
        \email{cdouglas@maths.ox.ac.uk}
      \urladdr{http://people.maths.ox.ac.uk/cdouglas}
      \author{Andr{\'e} Henriques}
      \address{Mathematisch Instituut\\
               Universiteit Utrecht, Postbus 80.010\\
               3508 TA Utrecht, The Netherlands}
        \email[Corresponding author]{a.g.henriques@uu.nl}
      \urladdr{http://www.staff.science.uu.nl/\!\raisebox{-1mm}{~}\!henri105}

\usepackage{hyperref}
\usepackage{enumerate,amssymb}
\usepackage[arrow,curve,matrix,tips,2cell]{xy}
  \SelectTips{eu}{10} \UseTips
  \UseAllTwocells
\usepackage{tikz}                \usetikzlibrary{calc}               \usetikzlibrary{matrix}		\usetikzlibrary{patterns}
\usepackage{pdfcolmk}
\usepackage{calc}
\usepackage{multirow}
\usepackage{stmaryrd}

       \newcommand{\cala}{\mathcal{A}}
       \newcommand{\calb}{\mathcal{B}}
  \newcommand{\IC}{\mathbb{C}}     \newcommand{\calc}{\mathcal{C}}
  \newcommand{\ID}{\mathbb{D}}     
       \newcommand{\cale}{\mathcal{E}}

       \newcommand{\cali}{\mathcal{I}}

       \newcommand{\call}{\mathcal{L}}
  \newcommand{\IM}{\mathbb{M}}     
  \newcommand{\IN}{\mathbb{N}}     
       \newcommand{\calo}{\mathcal{O}}

  \newcommand{\IR}{\mathbb{R}}     
  \newcommand{\IS}{\mathbb{S}}

  \newcommand{\IZ}{\mathbb{Z}}     

  \newcommand{\bfA}{{\mathbf A}}      
  \newcommand{\bfB}{{\mathbf B}}

  \newcommand{\bfJ}{{\mathbf J}}

  \definecolor{AHcolor}{rgb}{0.5,0.0,0.5}   % Textcolor for AB
  \definecolor{CDcolor}{rgb}{0.7,0.0,0.3}   % Textcolor for CD
  \definecolor{ABcolor}{rgb}{0.2,0.8,0.2}   % Textcolor for AH
  \newcommand{\tikzmath}[2][]
     {\vcenter{\hbox{\begin{tikzpicture}[#1]#2
                     \end{tikzpicture}}}
     }

  \newcommand{\displscale}{.05}

  \definecolor{spacecolor}{gray}{.7}
  \definecolor{antispacecolor}{gray}{.45}

  \theoremstyle{plain}
  \newtheorem{theorem}{Theorem}[section]
  
  \newtheorem{lemma}[theorem]{Lemma}
  \newtheorem{corollary}[theorem]{Corollary}
  \newtheorem{proposition}[theorem]{Proposition}

  \newtheorem*{theorem*}{Theorem}

  \theoremstyle{definition}
  \newtheorem{definition}[theorem]{Definition}
  
  \newtheorem{warning}[theorem]{Warning}
  
  \newtheorem{convention}[theorem]{Convention}
  \newtheorem*{construction}{Construction}

  \theoremstyle{remark}
  
  \newtheorem{remark}[theorem]{Remark}
  \newtheorem{example}[theorem]{Example}
  \newtheorem*{question}{Question}
  \newtheorem*{note}{Note}

  \makeatletter\let\c@equation=\c@theorem\makeatother

  \DeclareRobustCommand{\SkipTocEntry}[5]{}

  \DeclareMathOperator{\Ad}{Ad}
  
  \DeclareMathOperator{\Aut}{Aut}
  \DeclareMathOperator{\colim}{colim}
  \DeclareMathOperator{\Conf}{Conf}
  
  \DeclareMathOperator{\Diff}{Diff}

  \DeclareMathOperator{\id}{id}
  
  \DeclareMathOperator{\Inn}{Inn}

  \DeclareMathOperator{\PSL}{PSL}
  \DeclareMathOperator{\PU}{PU}
  \DeclareMathOperator{\PN}{PN}
  \DeclareMathOperator{\Rep}{Sect}
  \DeclareMathOperator{\U}{U}
  
  \DeclareMathOperator{\supp}{supp}
  
  \DeclareMathOperator{\SU}{SU}
  \DeclareMathOperator{\tr}{tr}

  \newcommand{\INT}{{\mathsf{INT}}}
  \newcommand{\VN}{{\mathsf{VN}}}

  \newcommand{\alg}{{\mathit{alg}}}
  
  \newcommand{\op}{{\mathit{op}}}

  \newcommand{\x}{{\times}}
  \newcommand{\ox}{{\otimes}}
  
  \newcommand{\dd}{{\partial}}

  \newcommand{\nid}{\noindent}

\begin{document}

\begin{abstract}
We describe a coordinate-free perspective on conformal nets, as functors from intervals to von Neumann algebras. We discuss an operation of fusion of intervals and observe that a conformal net takes a fused interval to the fiber product of von Neumann algebras. Though coordinate-free nets do not a priori have vacuum sectors, we show that there is a vacuum sector canonically associated to any circle equipped with a conformal structure.  This is the first in a series of papers constructing a 3-category of conformal nets, defects, sectors, and intertwiners.
\end{abstract}

\maketitle

%====================================================================

% ------------------------------------- 
\tableofcontents

\section*{Introduction}
In 
their work on algebraic quantum field theory, Haag and Kastler studied nets of operator algebras.
These are covariant functors from the category of open subsets of space-time to that of $C^*$-algebras or von Neumann algebras~\cite{Doplicher-Haag-Roberts(fields-observablesI),Doplicher-Haag-Roberts(fields-observablesII),Haag(Local-quantum-physics),Haag-Kastler(algebraic-approach)}.

For two-dimensional conformal field theory, one takes space-time to be two-dimensional Minkowski space $\IM^2$, or a compactification thereof,
and requires the net to be covariant with respect to the group of conformal diffeomorphisms of~$\IM^2$.
Since that group contains $\Diff(\IR) \times \Diff(\IR)$ as a subgroup of finite index,
the natural next step is to consider nets of von Neumann algebras on the real line or the circle,
which are covariant with respect to diffeomorphisms of the line or the circle.
The latter correspond to chiral conformal field theories, and are called conformal nets.
They have been studied intensively---see for example the papers~\cite{Brunetti-Guido-Longo(1993modular+duality-in-CQFT),Buchholz-Mack-Todorov(1988current-alg),Buchholz-Schulz-Mirbach(Haag-duality-in-conformal-quantum-field-theory),Gabbiani-Froehlich(OperatorAlg-CFT), Wassermann(ICM)}.
Classification results and many more references can be found in Kawahigashi--Longo~\cite{Kawahigashi-Longo(2004classification)}.

Our interest in conformal nets was prompted by the following question of Stephan Stolz and Peter Teichner,
which arose in connection with their ongoing program to construct elliptic cohomology using local quantum field theories \cite{Stolz-Teichner(2004what-is), Stolz-Teichner(SUSY-field-theories-and-cohomology)}.
Recall that von Neumann algebras form the objects of a 2-category,
where the morphisms are bimodules, and the 2-morphisms are maps of bimodules.
Given an $n$-category $\calc$ equipped with a unit object $1\in \calc$, the $(n-1)$-category $\call:=\mathrm{Hom}_\calc(1,1)$ is called the loops on $\calc$.
The $n$-category $\calc$ is then said to deloop $\call$.

\begin{question} (Stolz-Teichner, $2004$)
Does there exist an interesting 3-category that deloops the $2$-category of von Neumann algebras?
\end{question}

\noindent 
Here, by an ``interesting'' 3-category, they meant a 3-category other than the obvious one-object 3-category defined by $\mathit{Ob}(\calc)=\{1\}$ and $\mathrm{Hom}_\calc(1,1)=\{$von Neumann algebras$\}$.
Actually, given that von Neumann algebras form a symmetric monoidal 2-category, one should ask for a 
symmetric monoidal 3-category that deloops von Neumann algebras.
Some axiomatizations of the notion symmetric monoidal 3-category were presented in \cite{Douglas-Henriques(Internal-bicategories)}.
One of them is the notion of an internal bicategory in the 2-category of symmetric monoidal categories.

The present paper is the first of a series~\cite{BDH(all-together)},
the goal of which is to provide a positive answer to the above question of Stolz and Teichner.
Namely, we will show that conformal nets form a symmetric monoidal 3-category (an internal bicategory in the 
2-category of symmetric monoidal categories) that deloops the symmetric monoidal 2-category of von Neumann algebras.
This first paper of our series contains our definition of conformal nets.
We also treat the notion of sectors of conformal nets, otherwise known as representations.

Our definition of conformal nets 
is different from the standard definition in two important ways.
The first difference is that we 
do not include a positive-energy assumption.
This is important for the following reason.
We want certain objects in our $3$-category to be dualizable.
The natural candidate for the dual of a net $\cala$ is its complex conjugate $\bar \cala$,
which is almost never positive-energy.
We expect our conformal nets include the usual ones (Virasoro, loop groups, free boson, orbifolds, and cosets, among others), but
we also expect them to include new examples which do not have positive-energy, 
such as the spatial slices of nets on $\IM^2$ studied in Kawahigashi--Longo--M\"uger~\cite{Kawahigashi-Longo-Mueger(2001multi-interval)}.

The second difference is that we expand the category of 
intervals on which a net is defined, and expand the category of 
von Neumann algebras in which they take values.
Traditionally, conformal nets assign to a subinterval of the circle
$S^1$ a von Neumann subalgebra of $\bfB(H_0)$ 
for a fixed ``vacuum" Hilbert space $H_0$.    
In our definition, a ``coordinate-free" conformal net 
assigns abstract von Neumann algebras to abstract intervals.
For local quantum field theory such a coordinate-free point of view
has been introduced in Brunetti--Fredenhagen--Verch~\cite{Brunetti-Fredenhagen-Verch(new-paradigm)}. 
The coordinate-free definition 
has the advantage that there is a priori no Hilbert space 
as part of the structure of a net.
(In the hierarchy of our $3$-category, Hilbert spaces appear later
as $2$-morphisms; the vacuum Hilbert space of a net $\cala$ will be the 
identity $2$-morphism of the identity $1$-morphism of the net.  In particular, we prove that the vacuum Hilbert space with its action of the diffeomorphism group of the circle can be reconstructed from the von Neumann algebras associated to intervals; our notion of conformal nets is therefore closely related to existing notions in the literature~\cite{Kawahigashi-Longo-Mueger(2001multi-interval)}.)
Moreover, we can glue abstract intervals to one another, but not arbitrary subintervals of the
circle: the gluing of two subintervals is typically no longer a subinterval.
This will be of crucial importance for the composition of $1$-morphisms, in the third paper of our series.

\subsection*{Outline}
\addtocontents{toc}{\SkipTocEntry}
In the first section, we introduce our definition of conformal nets and study some of their basic properties.
We discuss the category $\Rep(\cala)$ of sectors of a conformal net $\cala$, and define the monoidal structure on it.
We also introduce a coordinate-free version of the category of sectors, $\Rep_S(\cala)$, that depends on the choice of a circle~$S$.

In Section~\ref{sec:covariance}, we study the vacuum sector $H_0(S,\cala)$ of a conformal net $\cala$, which is the unit object in $\Rep_S(\cala)$.
In Theorem~\ref{thm: Vaccum Sector}, we show that if the circle $S$ is equipped with a conformal structure, then the vacuum sector is well defined up to unique unitary isomorphism.
The vacuum sector is covariant for conformal maps of circles, and projectively covariant for diffeomorphisms.

Section~\ref{sec:The finiteness condition} concerns finiteness properties of the category of sectors. 
We start by discussing the notion of finite index for conformal nets 
in terms of a certain minimal index of the vacuum sector.
We describe the characterization of finite index nets in terms
of the category of sectors: the conformal net $\cala$ has 
finite index if and only if the category of sectors $\Rep(\cala)$ is fusion. 
This provides alternative proofs of results of
Kawahigashi--Longo--M\"uger~\cite{Kawahigashi-Longo-Mueger(2001multi-interval)}
and Longo-Xu~\cite{Longo-Xu(dichotomy)}.

Finally, in Section~\ref{sec:pos-energy-nets}, we relate our definition of conformal nets to other definitions in the literature.
We start by presenting a circle-based definition that is in principle equivalent to our usual coordinate-free definition.
We then discuss the more classical definition of conformal nets, which includes the positive-energy condition and is not equivalent to our notion.
We call these classical nets ``positive-energy nets", and check that positive-energy nets yield examples of conformal nets in our sense.
We review the construction of positive-energy nets from loop groups, and collect the necessary results from the literature to show that they yield coordinate-free nets.

The appendix contains a brief  summary of definitions and results 
about von Neumann algebras, Connes fusion, dualizability, 
statistical dimension, and Haagerup's $u$-topology. 
With the exception of the last subsection, these topics are discussed in more detail in our paper~\cite{BDH(Dualizability+Index-of-subfactors)}.

\subsection*{Acknowledgements} 
\addtocontents{toc}{\SkipTocEntry}
We would like to thank Stephan Stolz and Peter Teichner for their continual support during this project and for having formulated the question that led to it, and to thank Michael M\"uger for his invaluable guidance regarding conformal nets. We thank Sebastiano Carpi for pointing out a missing argument and
the reference~\cite{Longo-Xu(dichotomy)}, and Marcel Bischoff for pointers to the literature.
The last author also would like to thank Michael Hopkins for his suggestion to read Wassermann's articles.  The first author was supported by the Sonderforschungsbereich 878, and the second author was partially supported by a Miller Research Fellowship.

%=====================================================================

\section{Conformal nets} \label{sec:nets}

\subsection{Definition of conformal nets}\label{sec: Definition of conformal nets}

  All $1$-manifolds are compact, smooth, and oriented.
  The standard circle $S^1:=\{z\in \IC:|z|=1\}$ is the set of 
  complex numbers of modulus one, equipped with the 
  counter-clockwise orientation.
  By a \emph{circle}, we shall mean a smooth manifold $S$ that is 
  diffeomorphic to the standard circle $S^1$.
  Similarly, by an \emph{interval}, we shall mean a smooth 1-manifold that 
  is diffeomorphic to the standard interval $[0,1]$. 
  Equivalently, an interval is a compact non-empty $1$-manifold with 
  boundary that is connected and simply connected.
  For a 1-manifold $I$, we denote by $\bar I$ the same manifold equipped 
  with the opposite orientation, by $\Diff(I)$ the group of diffeomorphisms of $I$, and by 
  $\Diff_+(I)$ the subgroup of orientation-preserving diffeomorphisms.
  Let $\INT$ be the category whose objects are intervals and whose 
  morphisms are embeddings, not necessarily orientation-preserving.
  We also let $\VN$ be the category whose objects are von Neumann algebras with separable preduals, and 
  whose morphisms are $\IC$-linear homomorphisms, 
  and $\IC$-linear anti-homomorphisms\footnote{An anti-homomorphism is a  
  unital map satisfying $f(ab) = f(b)f(a)$.}.

The following notion of conformal nets differs from the standard definition in the literature: our nets are ``coordinate-free'', in the sense that they take values on all abstract intervals, not only on subintervals of the standard circle; our nets also need not satisfy the usual positive-energy condition.

  A \emph{net} is a covariant functor $\cala \colon \INT \to \VN$ taking 
  orientation-preserving embeddings to injective homomorphisms and 
  orientation-reversing embeddings to injective antihomomorphisms.
  A net is said to be \emph{continuous} if for any intervals $I$ and $J$, 
  the natural map $\mathrm{Hom}_{\INT}(I,J)\to \mathrm{Hom}_{\VN}(\cala(I),\cala(J))$ is continuous 
  for the $\mathcal C^\infty$ topology on $\mathrm{Hom}_{\INT}(I,J)$ and 
  Haagerup's $u$-topology on 
  $\mathrm{Hom}_{\VN}(\cala(I),\cala(J))$, reviewed in the appendix.
  Given a subinterval $I \subseteq K$, we will 
  often not distinguish between $\cala(I)$ and its image in $\cala(K)$.

\begin{definition}
\label{def:conformal-net}
A \emph{conformal net} is a continuous net $\cala$ subject to the following conditions.
Here, $I$ and $J$ are subintervals of an interval $K$: 
\begin{enumerate}
\item \emph{Locality:} If $I,J\subset K$ have disjoint interiors, then $\cala(I)$ and $\cala(J)$ are commuting subalgebras of $\cala(K)$.
\item \emph{Strong additivity:} If $K = I \cup J$, then $\cala(K)$ is generated as a von Neumann algbera by its two subalgebras: $\cala(K) = \cala(I) \vee \cala(J)$.
\item \emph{Split property:} If $I,J\subset K$ are disjoint, then the map from the algebraic tensor product $\cala(I) \ox_{\alg} \cala(J) \to \cala(K)$ extends to a map from the spatial tensor product
$\cala(I) \, \bar{\ox} \, \cala(J) \to \cala(K)$.
\item \emph{Inner covariance:} 
  If $\varphi\in\Diff_+(I)$ restricts to the identity in a neighborhood of 
  $\partial I$, then $\cala(\varphi)$ is an inner automorphism 
  of $\cala(I)$. 
  (A unitary $u \in \cala(I)$ with $\Ad(u) = \cala(\varphi)$
  is said to \emph{implement} $\varphi$.)
\item \label{def:conformal-net:vacuum} 
\emph{Vacuum sector:} 
Suppose that $J \subsetneq I$ contains the boundary point
$p \in \dd I$, and let $\bar{J}$ denote $J$ with the reversed 
orientation; $\cala(J)$ acts on $L^2(\cala(I))$ 
via the left action of $\cala(I)$, and 
$\cala(\bar{J}) \cong \cala(J)^\op$ acts on $L^2(\cala(I))$ 
via the right action of $\cala(I)$.
In that case, we require that the action of
$\cala(J) \ox_{\alg} \cala( \bar{J} )$ on $L^2(\cala(I))$
extends to an action of $\cala(J \cup_p \bar{J})$:
\begin{equation}\label{eq: Vaccum sector axiom for nets}
\qquad\tikzmath{
\matrix [matrix of math nodes,column sep=1cm,row sep=5mm]
{ 
|(a)| \cala(J) \ox_{\alg} \cala( \bar{J} ) \pgfmatrixnextcell |(b)| \bfB(L^2\cala(I))\\ 
|(c)| \cala(J \cup_p \bar{J}) \\ 
}; 
\draw[->] (a) -- (b);
\draw[->] (a) -- (c);
\draw[->,dashed] (c) -- (b);
}
\end{equation}
Here, $J \cup_p \bar{J}$ is equipped with any smooth structure extending the given smooth structures on $J$ and $\bar J$, and for which the orientation-reversing involution that exchanges $J$ and $\bar{J}$ is smooth.
\end{enumerate}
\end{definition}

\noindent
Note that $\cala(\bar J)$ is canonically isomorphic to $\cala(J)^\op$ via the antihomomorphism $\cala(\mathrm{Id}_J):\cala(J)\to \cala(\bar J)$.
That fact was used above in the vacuum axiom in order to identify $\cala(\bar{J})$ with $\cala(J)^\op$.
Also, the proper way of visualizing the interval $J \cup_p \bar{J}$ is as submanifold of the circle $S:=I\cup_{\partial I}\bar I$:
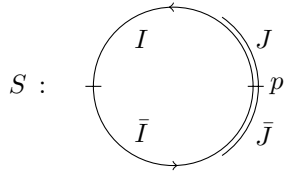
\[
S\,:\,\,\,\tikzmath[scale=.07]
{ \useasboundingbox (-20,-18) rectangle (20,18); \draw[->] (-.7,15) -- (-.8,15);\draw[->] (.7,-15) -- (.8,-15); \draw (0,0) circle (15) (-16.5,0) -- (-13.5,0) (14,0) -- (17,0)
(16,0) arc (0:55:16) (16,0) arc (0:-55:16) (-6,9) node {$I$} (-6,-8.5) node {$\bar I$} (19.5,0) node {$p$} (17,9) node {$J$} (17,-9) node {$\bar{J}$};} %tikzmath
\]
Here, the circle $S$ is equipped with a smooth structure such that the three embeddings $I\hookrightarrow S$, $\bar I\hookrightarrow S$, $J \cup_p \bar{J}\hookrightarrow S$ are smooth,
and the involution $S\to S$ that exchanges $I$ with $\bar I$ is smooth.

\begin{example}\label{ex: trivial conf net}
The trivial conformal net $\underline \IC$ is given by $\underline \IC(I)=\IC$ for any interval $I$,
and $\underline \IC(\iota)=\mathrm{Id}_\IC$ for any embedding of intervals $\iota$.
\end{example}

\noindent Some more substantial examples of conformal nets are discussed in Section \ref{subsec:loop-SU(N)-net}.

Given conformal nets $\cala$ and $\calb$,
their direct sum is given by $ (\cala \oplus \calb) (I) := \cala(I) \oplus \calb(I)$,
and their tensor product is $ (\cala \otimes \calb) (I) := \cala(I) \,\bar\otimes\, \calb(I)$.
Most of the axioms for $\cala \oplus \calb$ and $\cala \otimes \calb$ are straightforward. We just check the vacuum axiom for $\cala \otimes \calb$:
the Hilbert spaces $L^2(\cala(I)\,\bar\otimes\,\calb(I))$ and $L^2(\cala(I))\otimes L^2(\calb(I))$ are isomorphic as $\cala(I) \,\bar\otimes\, \calb(I)$-bimodules,
so the action of $\cala(J)\otimes_\alg\calb(J)\otimes_\alg \cala(\bar J)\otimes_\alg\calb(\bar J)$ extends to $\cala(J\cup_p\bar J) \,\bar\otimes\, \calb(J\cup_p\bar J)$.
This provides the desired extension of the action of $(\cala\otimes\calb)(J)\otimes_\alg (\cala\otimes\calb)(\bar J)$ on $L^2((\cala\otimes\calb)(I))$.

We record the following easy result for future use:

\begin{lemma} \label{lem:irrelevance-of-points}
Let $I_n\subset I$ be an increasing sequence of intervals whose union
is the interior of~$I$.
Then $\bigvee_n \cala(I_n) = \cala(I)$.
\end{lemma}

\begin{proof}
Let $\varphi_n:I \to I$, $\varphi_n(I)=I_n$ be a sequence of
embeddings that tends to $\id_I$.
Then every element $a\in \cala(I)$ can be written as
$\lim_n\cala(\varphi_n)(a)\in\bigvee_n\cala(I_n)$.
\end{proof}

\begin{remark}
The continuity condition in the definition of conformal nets is equivalent to the following slightly weaker condition.
It is enough that for any interval $I$, the natural map $\Diff_+(I)\to \mathrm{Aut}(\cala(I))$ be continuous for the $\mathcal C^\infty$ topology on $\Diff_+(I)$ and 
the $u$-topology on $\mathrm{Aut}(\cala(I))$, see Lemma \ref{lem: equiv def of continuity axiom}.
\end{remark}

\begin{remark}
It is possible that the condition that the algebras $\cala(I)$ have separable preduals follows from Definition~\ref{def:conformal-net},
more specifically, from the split property axiom---compare with~\cite[Proposition~1.6]{Doplicher-Longo(Standard-and-split-inclusions-of-von-Neumann-algebras)}.
\end{remark}

\begin{definition}
  \label{def:irreducilbe-nets}
   A conformal net $\cala$ is called \emph{irreducible} 
   if every algebra $\cala(I)$ is a factor. 
   A direct sum of finitely many irreducible conformal nets is 
   called \emph{semisimple}.
\end{definition}

\subsection{Sectors of conformal nets}
\label{subsec:sectors-for-nets}

In the 3-category that our series of papers \cite{BDH(all-together)} constructs, $\cala$-sectors correspond to 2-morphisms
\[
\tikzmath{
\node (a) at (0,0) {$\cala$};\node (b) at (2.1,0) {$\cala$};
\draw[->] (a) to[bend right = 30]node[below, xshift=.7]{$\scriptstyle 1_\cala$} (b);
\draw[->] (a) to[bend left = 30]node[above, xshift=.7]{$\scriptstyle 1_\cala$} (b);
\node at (1.07,0) {$\Downarrow$};
\node at (2.4,-.1) {.};
}
\]
Here, we shall discuss $\cala$-sectors without any 
reference to the 3-category.

\begin{definition}
\label{def:sector-of-a-conf-net}
Let $S$ be a circle and let $\cala$ be a conformal net.
An $S$-sector of $\cala$ (also called an $\cala$-sector on $S$) consists of a Hilbert space $H$ and a collection of homomorphisms
\[
\rho_I \colon \cala(I) \to \bfB(H),\quad I\subset S
\]
subject to the compatibility condition $\rho_I|_{\cala(J)}=\rho_J$ whenever $J \subset I$.
The category of $S$-sectors of $\cala$ is denoted $\Rep_S(\cala)$.

For the standard circle $S^1:=\{z\in\IC\,:\,|z|=1\}$, an $S^1$-sector of $\cala$ is simply called a sector of $\cala$, or an $\cala$-sector.
The category of $\cala$-sectors is denoted  $\Rep(\cala)$.
\end{definition}

Given an interval $I$, let $\Diff_0(I)$ denote the group of diffeomorphisms that restrict to the identity near the boundary of $I$.

\begin{lemma}\label{lem: open cover of circle => sector}
Let $S$ be a circle, and let $I_i\subset S$ be intervals whose interiors cover~$S$.
Let $\rho_{i} \colon \cala(I_i) \to \bfB(H)$ be actions 
subject to the following two conditions:
\begin{enumerate}
 \item $\rho_{i}|_{\cala(I_i\cap I_j)}=\rho_{j}|_{\cala(I_i\cap I_j)}$,
 \item if $J\subset I_i$ and $K\subset I_j$ are disjoint, then $\rho_{i}(\cala(J))$ commutes with $\rho_{j}(\cala(K))$.
\end{enumerate}
These actions extend uniquely to the structure of an $S$-sector on $H$.
\end{lemma}

\begin{proof}
For every interval $J\subset S$, we will construct an action
\[
\hat\rho_J:\cala(J)\to\bfB(H),
\]
uniquely determined by the requirement that
$\hat\rho_J|_{\cala(J \cap I_\ell)}=\rho_\ell|_{\cala(J \cap I_\ell)}$ for every $I_\ell$ in our cover.

Fix an element $I_0$ of our cover, with corresponding action $\rho_0$ of $\cala(I_0)$ on $H$.
Pick a diffeomorphism $\varphi=\varphi_n\circ\ldots\circ \varphi_1$ with $\varphi_s\in \Diff_0(I_{i_s})$, for $I_{i_s}$ in our cover, such that $\varphi(J)\subset I_0$.
Let $u_s\in \cala(I_{i_s})$ be unitaries implementing $\varphi_s$.
Identifying the elements $u_s$ with their images in $\bfB(H)$, we set
\[
\qquad\hat\rho_J(a):=u^*_1\ldots u^*_n\rho_0\big(\cala(\varphi)(a)\big)u_n\ldots u_1.
\]
For every sufficiently small interval $K\subset J \cap I_\ell$, we will show that
\begin{equation}\label{eq:diff trick for nets - bis}
\hat\rho_J|_{\cala(K)}=\rho_\ell|_{\cala(K)}.
\end{equation}
Here, `sufficiently small' means that the intervals $K_s:=\varphi_s(\ldots(\varphi_1(K)))$ should be contained in some $I_{k_s}$ in our cover,
and that for every $s,s'\le n$, either $K_s\subset I_{i_{s'}}$ or $K_s\cap \supp(\varphi_{s'})=\emptyset$.

For every $s\le n$, we claim that
\begin{equation}\label{eq:diff trick for nets - ter}
\quad\qquad u^*_1\ldots u^*_s\rho_{k_s}\big(\cala(\varphi_s\circ\ldots\circ \varphi_1)(a)\big)u_s\ldots u_1 = \rho_\ell(a) \qquad \forall a\in \cala(K).
\end{equation}
Note that \eqref{eq:diff trick for nets - bis} is the special case $s=n$ of this equation.
We prove \eqref{eq:diff trick for nets - ter} by induction on $s$.
The base case ($s=0$, $k_0=\ell$) is trivial.
The induction step reduces to the equation
\[
\rho_{k_s}\big(\cala(\varphi_{s})(b)\big) = u_{s} \rho_{k_{s-1}}(b)u^*_s,
\]
with $b=\cala(\varphi_{s-1}\circ\ldots\circ \varphi_1)(a)$.
Recall that $b\in\cala(K_{s-1})$, $u_s\in \cala(I_{i_s})$, and that, by assumption, either $K_{s-1}\subset I_{i_s}$ or $K_{s-1}\cap \supp(\varphi_s)=\emptyset$.
In the first case, we have
\[
u_s \rho_{k_{s-1}}(b)u^*_s  = u_s \rho_{i_s}(b)u^*_s = \rho_{i_s}(u_s b u^*_s) = \rho_{i_s}\big(\cala(\varphi_s)(b)\big) = \rho_{k_s}\big(\cala(\varphi_s)(b)\big).
\]
In the second case, by condition (ii), the elements $\rho_{k_{s-1}}(b)$ and $u_s$ commute in $\bfB(H)$.
It follows that
\[
u_s \rho_{k_{s-1}}(b)u^*_s = \rho_{k_{s-1}}(b) = \rho_{k_s}(b) = \rho_{k_s}\big(\cala(\varphi_s)(b)\big).
\]
This finishes the proof of \eqref{eq:diff trick for nets - ter} and hence of \eqref{eq:diff trick for nets - bis}.

Finally, by the strong additivity axiom, it follows from \eqref{eq:diff trick for nets - bis} that
$\hat\rho_J(a)=\rho_\ell(a)$ for every $a\in \cala(J \cap I_\ell)$.
\end{proof}

Let $S$ be a circle, let $j\in\Diff_-(S)$ be an orientation-reversing involution fixing the boundary $\dd I$ of some interval $I\subset S$, 
and let $I':=j(I)$.
The Hilbert space $H:=L^2(\cala(I))$ is equipped with:
\begin{itemize}
\item for each $J\subset I$, an action
\[
\rho_J:\cala(J)\hookrightarrow \cala(I)\xrightarrow{\,\,\,\text{left action of $\cala(I)$ on $H$}\,\,\,} \bfB(H)
\]
of the algebra $\cala(J)$.
\item for each $J\subset I'$, an action
\[
\qquad\rho_J:\cala(J)\hookrightarrow \cala(I')\xrightarrow{\cala(j)}\cala(I)^\op \xrightarrow{\,\,\text{right action of $\cala(I)$ on $H$}\,\,\,} \bfB(H)
\]
of $\cala(J)$.
\item If $J\subset S$ satisfies $j(J)=J$, then by the vacuum sector axiom, 
the homomorphism
\[
\rho_{J\cap I}\otimes \rho_{J\cap I'}:\cala(J\cap I)\otimes_\alg \cala(J\cap I')\to \bfB(H)
\]
extends to an action $\rho_J$ of $\cala(J)$ on $H$.
\end{itemize}
Applying Lemma \ref{lem: open cover of circle => sector}, we see that 
$L^2(\cala(I))$ comes naturally equipped with the structure of an $S$-sector of $\cala$.
We call it the {\em vacuum sector of $\cala$ associated to $S$, $I$, and $j$}.

We record the following subintervals of the standard circle for future usage:
\[
\begin{split}
S^1_\top:=\{z\in S^1\,|\, \Im\mathrm{m}(z)\ge 0\},
\qquad S^1_\dashv:=\{z\in S^1\,|\, \Re\mathrm{e}(z)\ge 0\},\\
S^1_\bot:=\{z\in S^1\,|\, \Im\mathrm{m}(z)\le 0\},
\qquad S^1_\vdash:=\{z\in S^1\,|\, \Re\mathrm{e}(z)\le 0\}.
\end{split}
\]

\begin{definition}\label{def: L2(A)}
We let $H_0(\cala)$ denote the vacuum sector of $\cala$ associated to the standard circle $S^1$, its upper half $S^1_\top$, and the involution $z\mapsto \bar z$.
It is defined by
\[
H_0(\cala):=L^2(\cala(S^1_\top)),
\]
and has left actions of $\cala(I)$ for every $I\subset S^1$.
\end{definition}

Given two circles $S_1$, $S_2$ and a diffeomorphism $\varphi:S_1\to S_2$, there is a corresponding functor
\begin{equation} \label{eq: functor phi^*}
\begin{split}
\Rep_{S_2}(\cala)\,&\longrightarrow\,\Rep_{S_1}(\cala)\\
H\,\,\,\,\,\,&\,\mapsto\,\,\,\,\, \varphi^*H.
\end{split}
\end{equation}
For an orientation-preserving diffeomorphism $\varphi$, that functor sends $(H,\{\rho_J\}_{J\subset S_2})$ to 
the $S_1$-sector with underlying Hilbert space $H$, and actions $\rho_{\varphi(J)}\circ \cala(\varphi|_J)$.
If the diffeomorphism $\varphi$ is orientation-reversing, then $\varphi^*H$ is the complex conjugate Hilbert space $\overline{H}$, equipped with the actions
\[
\cala(J)\xrightarrow{\cala(\varphi|_J)}\cala(\varphi(J))^\op
\xrightarrow{\,\rho_{\varphi(J)}\,} \bfB(H)^\op
\xrightarrow{\,\,\,*\,\,\,}\, \overline{\bfB(H)}\,=\,\bfB(\overline{H})
\]
for $J\subset S_1$.

\begin{proposition}\label{prop: phi* == psi*}
Let $S_1$ and $S_2$ be two circles, and let $\varphi,\psi\in\Diff_+(S_1,S_2)$ be diffeomorphisms.
Then the functors
\[
\varphi^*,\,\psi^*\,\,:\,\,\Rep_{S_2}(\cala)\to\Rep_{S_1}(\cala)
\] 
are non-canonically unitarily naturally equivalent (in other words, there exists a unitary natural equivalence $\varphi^*\cong \psi^*$, but there is no canonical way of choosing such a natural equivalence).
\end{proposition}
\begin{proof}
Write $\psi\circ\varphi^{-1} = \varphi_1\circ\ldots\circ\varphi_n$ 
as a product of diffeomorphisms $\varphi_i\in\Diff_0(I_i)$ with support in intervals $I_i\subset S_2$,
and let $u_i\in \cala(I_i)$ be unitaries implementing them (one may arrange this with $n=2$).
For any sector $(H,\{\rho_J\}_{J\subset S_2})$, conjugation by $\rho_{I_1}(u_1)\ldots \rho_{I_n}(u_n)$ provides a unitary isomorphism of $S_2$-sectors from $H$ to $(\varphi^{-1})^* \psi^* H$, which can be interpreted as a unitary isomorphism of $S_1$-sectors from $\varphi^* H$ to $\psi^* H$.
The collection of all those isomorphisms is a unitary natural transformation from $\varphi^*$ to $\psi^*$.
\end{proof}

\begin{corollary}\label{phi* == psi*}
Let $S$ be a circle, and $\varphi\in\Diff_+(S)$ a diffeomorphism.
Then for any $S$-sector $H$ we have $\varphi^* H\cong H$.
\end{corollary}

\begin{proof}
Take $\psi=\mathrm{Id}_S$ in the above proposition.
\end{proof}

\begin{corollary}\label{cor: non canonical vacuum}
Let $S$ be a circle, $I_1,I_2\subset S$ intervals, and let $j_1,j_2\in\Diff_-(S)$ be involutions that fix $\partial I_1$ and $\partial I_2$, respectively.
Let $H_1$ be the vacuum sector of $\cala$ associated to $S$, $I_1$, $j_1$,
and let $H_2$ be the vacuum sector of $\cala$ associated to 
$S$, $I_2$, $j_2$.
Then $H_1$ and $H_2$ are isomorphic as $S$-sectors.
\end{corollary}

\begin{proof}
Let $\varphi\in\Diff_+(S)$ be a diffeomorphism that sends $I_1$ to $I_2$ and that intertwines the actions of $j_1$ and $j_2$.
By the definition of vacuum sector, the diffeomorphism $\varphi$ induces an isomorphism $\varphi^* H_2\cong H_1$ and by Corollary~\ref{phi* == psi*}, we also have $\varphi^* H_2 \cong H_2$.
\end{proof}

Given a sector $H\in\Rep(\cala)$ on the standard circle,
and given another circle $S$, we denote by $H(S)\in \Rep_S(\cala)$ the $S$-sector $\varphi^*H$, where $\varphi\in\Diff_+(S,S^1)$ is \emph{some} diffeomorphism.
Note that by Proposition \ref{prop: phi* == psi*}, the $S$-sector $H(S)$ is well defined up to (non-canonical) unitary isomorphism.

In the case of the vacuum sector, the above construction specializes to:

\begin{definition}\label{def: non canonical vacuum}
Given a circle $S$, we let $$H_0(S,\cala)\in\Rep_S(\cala)$$ stand for any one of the Hilbert spaces $H_0(\cala)(S)$ considered in Corollary \ref{cor: non canonical vacuum} or equivalently any one of the Hilbert spaces considered 
 in the paragraph following Corollary \ref{cor: non canonical vacuum},
and call it \emph{the vacuum sector of $\cala$ associated to $S$}.
We sometimes abbreviate $H_0(S,\cala)$ by $H_0(S)$.
That Hilbert space is well defined up to non-canonical unitary isomorphism of $S$-sectors.
\end{definition}

\nid We will see later, in Section~\ref{sec:covariance}, that $H_0(S,\cala)$ can be
determined canonically if we fix a conformal structure on $S$.

For $S$ a circle and $I\subset S$ an interval, let us denote by $I'\subset S$ the closure of its complement in $S$.

\begin{proposition}[Haag duality]\label{prop: [Haag duality]}
Let $\cala$ be a conformal net, and $S$ be a circle.
Then for any $I\subset S$, the algebra $\cala(I')$ is the commutant of $\cala(I)$ on $H_0(S,\cala)$.

Given intervals $J\subset K$ such that $J^c$, the closure of $K\setminus J$, is itself an interval, the commutant of $\cala(J)$ in $\cala(K)$ is $\cala(J^c)$.
\end{proposition}

\begin{proof}
Let $j\in\Diff_-(S)$ be an involution that fixes $\partial I$, and let $H=L^2(\cala(I))$ be the vacuum sector associated to $S$, $I$, and $j$.
The equation $\cala(I')=\cala(I)'$ is obvious on $H$, and follows for any vacuum $H_0(S,\cala)$ since the two are isomorphic.

For the second statement, pick an embedding $K\to S$.
Using strong additivity, we then have 
$\cala(J)'\cap \cala(K)=\cala(J)'\cap \cala(K')'=(\cala(J)\vee\cala(K'))'=\cala(J\cup K')'=\cala(J'\cap K)=\cala(J^c)$.
\end{proof}

\subsection{Fusion along intervals}\label{sec: Fusion along intervals}
In this section, we introduce the operation $\circledast$ of fusion of von Neumann algebras, and
show that it corresponds to the geometric operation of gluing two intervals and then discarding the part along which they were glued:
\begin{equation}\label{eq: picture of fusion of intervals}
\tikzmath[scale=.4] {\useasboundingbox (-2.1,-1.1) rectangle (.2,1.1); 
\draw (-2,-.7) [rounded corners=7pt]-- (-2,-.1) -- (-1.5,1) [rounded corners=4pt]-- (-.6,0) [rounded corners=1.5pt]-- (0,0) -- (0,-1);} %tikzmath
\tikzmath[scale=.4] {\useasboundingbox (-.2,-1.1) rectangle (2.1,1.1);
\draw (0,-1) [rounded corners=1.5pt]-- (0,0) [rounded corners=3pt]-- (.5,0) -- (1,-.3) -- (1.7,.5) -- (2.2,.2);} %tikzmath
\qquad\quad
\tikzmath[scale=.032] {\node at (8.7,-.65) {$\rightsquigarrow$}; \node at (3,7) {$\scriptstyle\circledast$};
\draw[line width = .5, line cap=round, line join=round] (-8,0) -- (-7,.9) -- (-5,-.9) -- (-3,.9) -- (-1,-.9) -- (1,.9) -- (3,-.9) -- (4,0);} %tikzmath
\qquad\quad \tikzmath[scale=.4] {\useasboundingbox (-2.1,-1.1) rectangle (2.1,1.1);\draw (-2,-.7) [rounded corners=7pt]-- (-2,-.1) -- 
(-1.5,1) [rounded corners=4pt]-- (-.6,0) -- (0,0); \draw (0,0) [rounded corners=3pt]-- (.5,0) -- (1,-.3) -- (1.7,.5) -- (2.2,.2);} %tikzmath
\end{equation}
Later in this section, we will use this operation to define the operation of fusion of sectors.
This operation will also be crucial in the third paper of our series \cite{BDH(all-together)}, in order to define the composition of two defects.

\subsubsection*{Fusion of von Neumann algebras}
\begin{definition}\label{def: Fusion of vN alg}
Let $A\leftarrow C^\op$, $C\to B$ be two homomorphisms between von Neumann algerbas, and let ${}_AH$ and ${}_BK$ be faithful modules.
Viewing $H$ as a right $C$-module, we may form the Connes fusion 
$H\boxtimes_C K$ (recalled in the Appendix).
One then defines 
\begin{equation}\label{eq:def fusion of algebras}
A \circledast_C B := (A\cap {C^\op\hspace{.2mm}}'\hspace{.2mm})\vee(C'\cap B)\,\subset\, \bfB(H\boxtimes_C K),
\end{equation}
where the commutants of $C^\op$ and $C$ are taken in $H$ and $K$, respectively. 
This algebra is independent of the choices of $H$ and $K$ (see Proposition~\ref{A oast_C B independent of H and K}).
\end{definition}

\begin{warning}\label{warn: poor formal properties} The operation $\circledast$ has rather poor formal properties.
For example, given homomorphisms $A_1\leftarrow B_1^\op$, $B_1\to A_2$, $A_2\leftarrow B_2^\op$, $B_2\to A_3$, such that the images of $B_1$ and $B_2^\op$ commute in $A_2$,
one might ask for an associator isomorphism $(A_1 \circledast_{B_1} A_2) \circledast_{B_2} A_3\cong A_1 \circledast_{B_1} (A_2 \circledast_{B_2} A_3)$.
Such an isomorphism does not exist in general: there are algebras for which one of the following two inclusions
\[
\tikzmath{
\node at (0,0) {$(A_1\cap {B_1^\op\hspace{.2mm}}'\hspace{.2mm})\vee(B_1'\cap A_1 \cap {B_2^\op\hspace{.2mm}}'\hspace{.2mm})\vee(B_2'\cap A_3)$};
\node at (5.55,.4) {$(A_1 \circledast_{B_1} A_2) \circledast_{B_2} A_3$};
\node at (5.65,-.4) {$A_1 \circledast_{B_1} (A_2 \circledast_{B_2} A_3)$};
\node[rotate=20] at (3.7,.23) {$\hookrightarrow$};
\node[rotate=-20] at (3.7,-.25) {$\hookrightarrow$};}
\]
is an isomorphism, but the other isn't.
We present a counterexample (but without further justifications as this would take us too far afield):
take a conformal net $\cala$ with $\mu(\cala)>1$ (Definition \ref{def:index-for-nets}), and let $A_1=\cala([0,2])^\op$, $B_1=\cala([1,2])$, $A_2= \cala([-2,2])\cap\cala([-1,0])'$, $B_2=A_3=\cala([0,1])^\op$.
\end{warning}

\begin{proposition}\label{A oast_C B independent of H and K}
The algebra $A \circledast_C B$ is independent, up to canonical isomorphism, of the choice of faithful modules ${}_AH$ and ${}_BK$.
\end{proposition}

\begin{proof}
Let $H_1$ and $H_2$ be faithful $A$-modules, and let $K_1$ and $K_2$ be faithful $B$-modules.
Upon choosing isomorphisms $H_1\otimes \ell^2\cong H_2\otimes \ell^2$ and $K_1\otimes \ell^2\cong K_2\otimes \ell^2$,
we get the following commutative diagram of algebra homomorphisms:\smallskip
\[
\tikzmath{
\node (a) at (-2.7,0) {$(A\cap {C^\op\hspace{.2mm}}'\hspace{.2mm})\otimes_\alg(C'\cap B)$};
\node (b) at (1.1,.6) {$\bfB(H_1\boxtimes_C K_1)$};
\node (b') at (5,.6) {$\bfB\big((H_1\otimes \ell^2)\boxtimes_C (K_1\otimes \ell^2)\big)$};
\node (c) at (1.1,-.6) {$\bfB(H_2\boxtimes_C K_2)$};
\node (c') at (5.05,-.6) {$\bfB\big((H_2\otimes \ell^2)\boxtimes_C (K_2\otimes \ell^2)\big).$};
\draw [->] ($(a.east)+(0,.2)$) -- ($(b.west)-(0,.1)$);
\draw [->] ($(a.east)-(0,.2)$) -- ($(c.west)+(0,.1)$);
\draw [->] ($(b.east) + (.05,.1)$) arc (90:270:.05) -- (b');
\draw [->] ($(c.east) + (.05,.1)$) arc (90:270:.05) -- (c');
\path (c'.north-|b') --node[sloped]{$\cong$} (b');
}
\smallskip\]
The completions of
$(A\cap {C^\op\hspace{.2mm}}'\hspace{.2mm})\otimes_\alg(C'\cap B)$ in $\bfB(H_1\boxtimes_C K_1)$ and $\bfB(H_2\boxtimes_C K_2)$
therefore agree since they might as well be taken in $\bfB\big((H_1\otimes \ell^2)\boxtimes_C (K_1\otimes \ell^2)\big)$
and $\bfB\big((H_2\otimes \ell^2)\boxtimes_C (K_2\otimes \ell^2)\big)$, respectively.
\end{proof}

If the modules $H$ and $K$ are not faithful, then there is still an action, albeit non-faithful, of $A \circledast_C B$ on $H\boxtimes_C K$:

\begin{lemma}\label{lem: extends to A (*)_C B}
Let $A\leftarrow C^\op$, $C\to B$ be homomorphisms, and let ${}_AH$ and ${}_BK$ be any modules.
Then the natural map
$(A\cap {C^\op\hspace{.2mm}}'\hspace{.2mm})\otimes_\alg(C'\cap B)\to \bfB(H\boxtimes_C K)$
extends to an action of $A \circledast_C B$ on $H\boxtimes_C K$.\hfill $\square$
\end{lemma}

The operation $\circledast$ is compatible with spatial tensor product in the sense that
given algebras $A_1$, $B_1$, $C_1$, $A_2$, $B_2$, $C_2$ and homomorphisms $A_1\leftarrow C_1^\op$, $C_1\to B_1$, $A_2\leftarrow C_2^\op$, $C_2\to B_2$,
there is a canonical isomorphism
\begin{equation}\label{eq: compatibility between circledast and bartimes}
(A_1\,\bar\otimes\,A_2) \circledast_{C_1\,\bar\otimes\,C_2} (B_1\,\bar\otimes\,B_2)\cong
(A_1 \circledast_{C_1} B_1)\,\bar\otimes\,(A_2 \circledast_{C_2} B_2).
\end{equation}

\begin{remark}\label{rm: fusion is generated by the relative commutants}
In~\cite{Timmermann(Invitation-2-quantum-groups)}, 
a similar operation $A \ast_C B$ is defined, under the name \emph{fiber product} of von Neumann algebras.
It is given by
\[
A \ast_C B := (A'\otimes_\alg B')',
\]
where the commutants $A'$ and $B'$ are taken in $\bfB(H)$ and $\bfB(K)$ respectively, while the last one is taken in
$\bfB(H\boxtimes_C K)$. 
Unlike $\circledast$, the operation $\ast$ is associative.
There is always an inclusion $A \circledast_C B\hookrightarrow A \ast_C B$ and under favorable circumstances, it can happen that those two algebras agree.
This will always be true in the cases that we consider (see the section on associativity of composition in the third paper of our series \cite{BDH(all-together)} for a precise statement).  We work with $A \circledast_C B$ as opposed to $A \ast_C B$ for technical convenience:
it is easier to check that the former commutes with other von Neumann algebras.
\end{remark}

\subsubsection*{The algebra of a fused interval}
We now make precise the heuristic of picture 
\eqref{eq: picture of fusion of intervals}.
Consider three intervals $I$, $I_l$, and $I_r$ equipped with two maps $i_l:I\to I_l$ and $i_r:I\rightarrow I_r$.
The maps $i_l$ and $i_r$ are orientation-reversing and orientation-preserving, respectively.
Moreover, we require that the closure $J_l$ of $I_l\setminus i_l(I)$ and the closure $J_r$ of $I_r\setminus i_r(I)$ be (non-empty) intervals,
and that $I_l\cup_I I_r$ be a ``Y-graph''---see \eqref{eq: Y-graph}.
We can then define the \emph{fused} interval $I_l\circledast_I I_r:=J_l\cup J_r\subset I_l\cup_I I_r$:
\begin{equation}\label{eq: Y-graph}
\tikzmath[scale=.4] {\useasboundingbox (-3,-2) rectangle (3,1); 
\draw (-2,-.7) [rounded corners=7pt]-- (-2,-.1) -- (-1.5,1) [rounded corners=4pt]-- (-.6,0) [rounded corners=1.5pt]-- (0,0) -- (0,-1);
\draw[->] (-.7,.139) -- (-.69,.13);
\node at (-.9,-1) {$I_l$};
\pgftransformxshift{10} %shift
\draw(0,0) -- (0,-.45);\draw[<-](0,-.44)--(0,-1) node[below]{$I$};
\pgftransformxshift{10} %shift
\draw (0,-1) [rounded corners=1.5pt]-- (0,0) [rounded corners=3pt]-- (.5,0) -- (1,-.3) -- (1.7,.5) -- (2.2,.2);
\draw[->] (.69,-.12) -- (.7,-.125);
\node at (.9,-1) {$I_r$};
} %tikzmath
\,\,\,\,\qquad I_l\cup_I I_r:\,\, \tikzmath[scale=.4] {\useasboundingbox (-2.1,-1.5) rectangle (2.1,1.1);\draw (0,-.05)--(0,-1)(-2,-.7) [rounded corners=7pt]-- (-2,-.1) -- 
(-1.5,1) [rounded corners=4pt]-- (-.6,0) [rounded corners=.5pt]-- (-.05,0) -- (0,-.05) [sharp corners]-- (0,-1); \draw[rounded corners=.5pt] (0,-.05)  -- (.05,0) [rounded corners=3pt]-- (.5,0) -- (1,-.3) -- (1.7,.5) -- (2.2,.2);} %tikzmath
\,\,\,\,\qquad I_l\circledast_I I_r:\,\, \tikzmath[scale=.4] {\useasboundingbox (-2.1,-1.5) rectangle (2.1,1.1);\draw[->] (-2,-.7) [rounded corners=7pt]-- (-2,-.1) -- 
(-1.5,1) [rounded corners=4pt]-- (-.6,0) -- (0.2,0); \draw (.19,0) [rounded corners=3pt]-- (.5,0) -- (1,-.3) -- (1.7,.5) -- (2.2,.2);} %tikzmath
\end{equation}
Note that the fused interval $I_l\circledast_I I_r$ only inherits a canonical $C^1$ structure.
Indeed, we have a natural identification of tangent spaces
\[
\def\hh{2}
\qquad\qquad\qquad\tikzmath{
\node (a) at (-\hh,0) {$T_pJ_l$};\node (b) at (0,0) {$T_pI$};\node (c) at (\hh,0) {$T_pJ_r$};
\draw [<-] (a) -- node[above]{$\scriptstyle -T_pi_l$}node[below, yshift=1]{$\scriptstyle \cong$} (b);	
\draw [->] (b) -- node[above]{$\scriptstyle T_pi_r$}node[below, yshift=1]{$\scriptstyle \cong$} (c);
}\qquad\qquad  \tikzmath{\node[scale=.9]{$p:=J_l\cap J_r$};}
\]
but no way to compatibly identify the higher germs of $J_l$ and $J_r$ at $p$.
Pick involutions $\alpha\in\Diff_-(I_l)$ and $\beta\in\Diff_-(I_r)$ that fix $p$, and such that the map
\[
\alpha|_{J_l}\cup\hspace{.3mm} \beta|_{I} : I_l=J_l\cup I \to I\cup J_r = I_r
\]
is smooth.
We equip $I_l\circledast_I I_r$ with the smooth structure pulled back via $\alpha|_{J_l}\cup \mathrm{Id}_{J_r}: I_l\circledast_I I_r\to I_r$ or, equivalently,
the one pulled back via $\mathrm{Id}_{J_l}\cup\beta|_{J_r} : I_l\circledast_I I_r\to I_l$.

The smooth structure on $I_l\circledast_I I_r$ depends on the involutions $\alpha$ and $\beta$.
The distinguishing feature of smooth structures arising this way is that there exists an action of the symmetric group $\mathfrak S_3$ on $I_l\cup_I I_r$
such that all the induced maps between $I_l$, $I_r$ and $I_l\circledast_I I_r$ are smooth.
The next proposition shows that the algebra $\cala(I_l\circledast_I I_r)$ associated to the fused interval does not depend on the above choices, up to canonical isomorphism:

\begin{proposition}
\label{prop: nets-and-fiber-product}
Let $\cala$ be a conformal net, and let $I$, $I_l$, $I_r$, $J_l$, $J_r$ be as above. 
Then there is a canonical isomorphism $\cala(I_l\circledast_I I_r)\cong \cala(I_l)\circledast_{\cala(I)} \cala(I_r)$, compatible with the inclusions of $\cala(J_l)$ and $\cala(J_r)$.
\end{proposition}

\begin{proof}
Let $S:=I\cup_{\partial I} \bar I$ be the circle obtained by gluing two copies of $I$ along their common boundary,
and let $j_0:S\to S$ be the orientation-reversing involution that exchanges $I$ and $\bar I$.
Equip $S$ with any smooth structure compatible with those on $I$ and $\bar I$, and for which the map $j_0$ is smooth.

Recall the involutions $\alpha\in\Diff_-(I_l)$ and $\beta\in\Diff_-(I_r)$ described above.
Identify $I$ with its image in $I_l$ and $I_r$ under the inclusions $i_l$ and $i_r$ respectively, and
let $p:=J_l\cap J_r$ be the trivalent vertex of the Y-graph $I_l\cup_I I_r$.
Pick orientation-preserving embeddings $f_l:J_l\to I$, $g_r:J_r\to \bar I$ that send $p$ to itself, and satisfy the following two conditions: ({\it i}\hspace{.4mm}) the maps
$f:I_l\to S$ and $g:I_r\to S$ given by
\[
\begin{split}
f\,:\,\, I_l&=J_l\cup \bar I \xrightarrow{\,\,f_l\cup \mathrm{Id}_{\bar I}\,\,} I\cup \bar I =S\\
g\,:\,\, I_l&=I\cup J_r \xrightarrow{\,\,\mathrm{Id}_{I}\cup g_r\,\,} I\cup \bar I =S
\end{split}
\]
are injective and smooth, and ({\it ii}\hspace{.4mm}) the equations $f \circ \alpha = j_0\circ f$ and $g \circ \beta = j_0\circ g$ are satisfied in a neighborhood of $p$.
Note that the map $f_l\cup g_r:I_l\circledast_I I_r\to S$ is then also smooth.

Let $H_0:=L^2(\cala(I))$ be the vacuum sector associated to $S$, $I$, and $j_0$.
We have two faithful actions
\[
\cala(I_l)\to \bfB(H_0)\quad \text{and}\quad \cala(I_r)\to \bfB(H_0)
\]
associated to $f$ and $g$.
We can therefore compute $\cala(I_l) \circledast_{\cala(I)} \cala(I_r)$ inside $\bfB\big(H_0\boxtimes_{\cala(I)} H_0\big) = \bfB\big(H_0\big)$.
By Haag duality, the relative commutant of $\cala(I)^\op$ in $\cala(I_l)$ is $\cala(J_l)$, and 
the relative commutant of $\cala(I)$ in $\cala(I_r)$ is $\cala(J_r)$.
We therefore have
\[
\cala(I_l) \circledast_{\cala(I)} \cala(I_r)
= \cala(J_l) \vee \cala(J_r)\,\,\subset\,\, \bfB\big(H_0\big),
\]
which is equal to $\cala(J_l\cup J_r)=\cala(I_l\circledast_I I_r)$ by the strong additivity axiom.
\end{proof}

\begin{corollary} \label{cor:cala(K)-acts}
Let $\cala$ be a conformal net, let $H_l$ be an $\cala(I_l)$-module, and let $H_r$ be an $\cala(I_r)$-module.
Then the two actions $\cala(J_l)$ and  $\cala(J_r)$ extend to an action of $\cala(I_l\circledast_I I_r)$ on $H_l \boxtimes_{\cala(I)} H_r$.
\end{corollary}

\subsubsection*{Fusion of sectors}
Consider now a theta-graph $\Theta$ with trivalent vertices $p$ and $q$, and let $S_1, S_2, S_3\subset \Theta$ be its three circle subgraphs,
with orientations as drawn below:
\begin{equation}\label{eq: Theta-graph}
\Theta:\, \tikzmath[scale=.055]{ \draw (60:14) arc (60:300:14); \draw (14,0) +(120:14) arc (120:-120:14); \draw (0,0) +(60:14) -- +(300:14); \node at (4,10) {$\scriptstyle p$}; \node at (4,-10) {$\scriptstyle q$}; }%tikzmath
\,\,\,\,\,\,\qquad \tikzmath[scale=.033]{ \useasboundingbox (-14,-30) rectangle (28,24); \draw[line width=.7] (60:14) arc (60:300:14); \draw[densely dotted] (14,0) +(120:14) arc (120:-120:14);
\draw[line width=.7] (0,0) +(60:14) -- +(300:14); \node at (8,-24) {$S_1$}; \draw[->] (90:14) ++ (-.5,0) -- +(-.1,0); }%tikzmath
\,\,,\quad \tikzmath[scale=.033]{ \useasboundingbox (-14,-30) rectangle (28,24); \draw[densely dotted] (60:14) arc (60:300:14); \draw[line width=.7] (14,0) +(120:14) arc (120:-120:14);
\draw[line width=.7] (0,0) +(60:14) -- +(300:14); \node at (8,-24) {$S_2$}; \draw[->] (14,14) ++ (-.5,0) -- +(-.1,0); }%tikzmath
\,\,,\quad \tikzmath[scale=.033]{ \useasboundingbox (-14,-30) rectangle (28,24); \draw[line width=.7] (60:14) arc (60:300:14); \draw[line width=.7] (14,0) +(120:14) arc (120:-120:14);
\draw[densely dotted] (0,0) +(60:14) -- +(300:14); \node at (8,-24) {$S_3$}; \draw[->] (14,14) ++ (-.5,0) -- +(-.1,0); }%tikzmath
\,\,.\end{equation}
Equip $S_1$, $S_2$, $S_3$ with smooth structures for which there exists an action of the symmetric group $\mathfrak S_3$ on $\Theta$
that fixes $p$ and $q$, permutes the three circles, and such that $\pi|_{S_a}$ is smooth for every $\pi\in \mathfrak S_3$ and $a\in \{1,2,3\}$.
Let
\[
I:=S_1\cap S_2,\quad K:=S_1\cap S_3,\quad L:=S_2\cap S_3.
\]
We equip $K$ with the orientation inherited from $S_1$, and give $I$ and $L$ the ones coming from $S_2$.

\begin{definition}
Given $S_1$, $S_2$, $S_3$, $I$, $K$, $L$ as above, the operation \medskip
\begin{equation}\label{eq: fusion of sectors}
\boxtimes_I\,\,:\,\,\Rep_{S_1}(\cala)\,
\times\,\Rep_{S_2}(\cala)\,\, \rightarrow \,\,\Rep_{S_3}(\cala) 
\medskip
\end{equation}
of \emph{fusion of sectors} is given by
$(H_1, H_2)\mapsto H_1\boxtimes_{\cala(I)} H_2$.
That space inherits an $\cala(K)$ action from $H_1$, and an $\cala(L)$ action from $H_2$ (both are left actions).
If $J\subset S_3$ is an interval not contained in and not containing $K$ or $L$, then $J=((J\cap S_1)\cup \bar I)\circledast_I ((J\cap S_2)\cup I)$,
and so $\cala(J)$ acts on $H_1\boxtimes_{\cala(I)} H_2$ by Corollary~\ref{cor:cala(K)-acts}.
The Hilbert space $H_1\boxtimes_{\cala(I)} H_2$ is then an $S_3$-sector by Lemma \ref{lem: open cover of circle => sector}.
\end{definition}

\subsubsection*{Associativity of fusion}
The operation \ref{eq: fusion of sectors} satisfies a certain version of associativity.
Given a graph that looks as follows
\tikz[scale=.3]{ \useasboundingbox (-2.2,-.6) rectangle (2.25,1.2); \draw (1.2,0)+(120:1) to[bend right=25] ($(-1.2,0)+(60:1)$); \draw  (-1.2,0)+(-60:1) to[bend right=25] ($(1.2,0)+(-120:1)$);
\draw (1.2,0) +(120:1) arc (120:-120:1); \draw (-1.2,0) +(60:1) arc (60:300:1); \draw (.2,0) +(60:1) -- +(300:1); \draw (-.2,0) +(120:1) -- +(240:1);},
with circles subgraphs $S_1,\ldots,S_6$ as indicated below \smallskip
\[
\tikzmath[scale=.38]{
\draw[densely dotted] (1.2,0)+(120:1) to[bend right=25] ($(-1.2,0)+(60:1)$);\draw[densely dotted]  (-1.2,0)+(-60:1) to[bend right=25] ($(1.2,0)+(-120:1)$);
\draw[densely dotted] (1.2,0) +(120:1) arc (120:-120:1);\draw[line width=.7] (-1.2,0) +(60:1) arc (60:300:1); \draw[densely dotted] (.2,0) +(60:1) -- +(300:1);
\draw[line width=.7] (-.2,0) +(120:1) -- +(240:1); \node[scale=.9] at (0,-1.7) {$S_1$}; \draw[->] (-1.25,1) -- +(-.01,0); }\,\,\, %tikzmath
\tikzmath[scale=.38]{
\draw[line width=.7] (1.2,0)+(120:1) to[bend right=25] ($(-1.2,0)+(60:1)$); \draw[line width=.7]  (-1.2,0)+(-60:1) to[bend right=25] ($(1.2,0)+(-120:1)$);
\draw[densely dotted] (1.2,0) +(120:1) arc (120:-120:1); \draw[densely dotted] (-1.2,0) +(60:1) arc (60:300:1); \draw[line width=.7] (.2,0) +(60:1) -- +(300:1);
\draw[line width=.7] (-.2,0) +(120:1) -- +(240:1); \node[scale=.9] at (0,-1.7) {$S_2$}; \draw[->] (-.06,1.04) -- +(-.01,0); }\,\,\, %tikzmath
\tikzmath[scale=.38]{
\draw[densely dotted] (1.2,0)+(120:1) to[bend right=25] ($(-1.2,0)+(60:1)$); \draw[densely dotted]  (-1.2,0)+(-60:1) to[bend right=25] ($(1.2,0)+(-120:1)$);
\draw[line width=.7] (1.2,0) +(120:1) arc (120:-120:1); \draw[densely dotted] (-1.2,0) +(60:1) arc (60:300:1); \draw[line width=.7] (.2,0) +(60:1) -- +(300:1);
\draw[densely dotted] (-.2,0) +(120:1) -- +(240:1); \node[scale=.9] at (0,-1.7) {$S_3$}; \draw[->] (1.13,1) -- +(-.01,0); }\,\,\, %tikzmath
\tikzmath[scale=.38]{
\draw[line width=.7] (1.2,0)+(120:1) to[bend right=25] ($(-1.2,0)+(60:1)$); \draw[line width=.7]  (-1.2,0)+(-60:1) to[bend right=25] ($(1.2,0)+(-120:1)$);
\draw[densely dotted] (1.2,0) +(120:1) arc (120:-120:1); \draw[line width=.7] (-1.2,0) +(60:1) arc (60:300:1); \draw[line width=.7] (.2,0) +(60:1) -- +(300:1);
\draw[densely dotted] (-.2,0) +(120:1) -- +(240:1); \node[scale=.9] at (0,-1.7) {$S_4$}; \draw[->] (-.06,1.04) -- +(-.01,0); }\,\,\, %tikzmath
\tikzmath[scale=.38]{
\draw[line width=.7] (1.2,0)+(120:1) to[bend right=25] ($(-1.2,0)+(60:1)$); \draw[line width=.7]  (-1.2,0)+(-60:1) to[bend right=25] ($(1.2,0)+(-120:1)$);
\draw[line width=.7] (1.2,0) +(120:1) arc (120:-120:1); \draw[densely dotted] (-1.2,0) +(60:1) arc (60:300:1); \draw[densely dotted] (.2,0) +(60:1) -- +(300:1);
\draw[line width=.7] (-.2,0) +(120:1) -- +(240:1); \node[scale=.9] at (0,-1.7) {$S_5$}; \draw[->] (-.06,1.04) -- +(-.01,0); }\,\,\, %tikzmath
\tikzmath[scale=.38]{
\draw[line width=.7] (1.2,0)+(120:1) to[bend right=25] ($(-1.2,0)+(60:1)$); \draw[line width=.7]  (-1.2,0)+(-60:1) to[bend right=25] ($(1.2,0)+(-120:1)$);
\draw[line width=.7] (1.2,0) +(120:1) arc (120:-120:1); \draw[line width=.7] (-1.2,0) +(60:1) arc (60:300:1); \draw[densely dotted] (.2,0) +(60:1) -- +(300:1);
\draw[densely dotted] (-.2,0) +(120:1) -- +(240:1); \node[scale=.9] at (0,-1.7) {$S_6$}; \draw[->] (-.06,1.04) -- +(-.01,0); } %tikzmath
\]
then there is an associator coming from the associator of Connes fusion that makes the following diagram commute:
\[
\tikzmath{
\matrix [matrix of math nodes,column sep=1.2cm,row sep=7mm]
{ 
|(a)| \Rep_{S_1}(\cala)\,\times\,\Rep_{S_2}(\cala)\times\,\Rep_{S_3}(\cala) \pgfmatrixnextcell |(b)| \Rep_{S_1}(\cala)\,\times\,\Rep_{S_5}(\cala)\\ 
|(c)| \Rep_{S_4}(\cala)\,\times\,\Rep_{S_3}(\cala) \pgfmatrixnextcell |[xshift=3](d)| \Rep_{S_6}(\cala).\\ 
}; 
\draw[->] (a) -- (b);
\draw[->] (c) -- (d);
\draw[->] (a) -- (c);
\draw[->] (b) -- (d.north-|b);
}
\]
\begin{note} 
There should be a similar canonical associator when the circles $S_1,\ldots, S_6$ are arranged as follows:
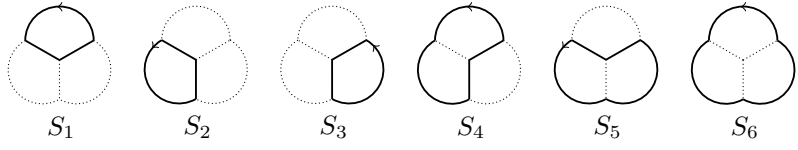
\[
\tikzmath[scale=.45]{
\draw[densely dotted] (120:1) arc (120:300:1);
\draw[densely dotted] (1,0) +(60:1) arc (60:-120:1);
\draw[line width=.7] (1,0) +(60:1) arc (0:180:1);
\draw[densely dotted] (0,0) +(30:.577) -- +(-60:1); 
\draw[line width=.7] (0,0) +(30:.577) -- +(120:1); 
\draw[line width=.7] (0,0) +(30:.577) -- +(30:1.732); 
\node at (.55,-1.7) {$S_1$}; \draw[->] (.5,1.86) ++ (-.05,0) -- +(-.01,0); }\hspace{.3cm}%tikzmath
\tikzmath[scale=.45]{
\draw[line width=.7] (120:1) arc (120:300:1);
\draw[densely dotted] (1,0) +(60:1) arc (60:-120:1);
\draw[densely dotted] (1,0) +(60:1) arc (0:180:1);
\draw[line width=.7] (0,0) +(30:.577) -- +(-60:1); 
\draw[line width=.7] (0,0) +(30:.577) -- +(120:1); 
\draw[densely dotted] (0,0) +(30:.577) -- +(30:1.732); 
\node at (.55,-1.7) {$S_2$}; \draw[->] (140:1)  -- +(227:.01); }\hspace{.3cm}%tikzmath
\tikzmath[scale=.45]{
\draw[densely dotted] (120:1) arc (120:300:1);
\draw[line width=.7] (1,0) +(60:1) arc (60:-120:1);
\draw[densely dotted] (1,0) +(60:1) arc (0:180:1);
\draw[line width=.7] (0,0) +(30:.577) -- +(-60:1); 
\draw[densely dotted] (0,0) +(30:.577) -- +(120:1); 
\draw[line width=.7] (0,0) +(30:.577) -- +(30:1.732); 
\node at (.55,-1.7) {$S_3$}; \draw[->] (1,0) ++(45:1)  -- +(133:.01); }\hspace{.3cm}%tikzmath
\tikzmath[scale=.45]{
\draw[line width=.7] (120:1) arc (120:300:1);
\draw[densely dotted] (1,0) +(60:1) arc (60:-120:1);
\draw[line width=.7] (1,0) +(60:1) arc (0:180:1);
\draw[line width=.7] (0,0) +(30:.577) -- +(-60:1); 
\draw[densely dotted] (0,0) +(30:.577) -- +(120:1); 
\draw[line width=.7] (0,0) +(30:.577) -- +(30:1.732); 
\node at (.55,-1.7) {$S_4$}; \draw[->] (.5,1.86) ++ (-.05,0) -- +(-.01,0); }\hspace{.3cm}%tikzmath
\tikzmath[scale=.45]{
\draw[line width=.7] (120:1) arc (120:300:1);
\draw[line width=.7] (1,0) +(60:1) arc (60:-120:1);
\draw[densely dotted] (1,0) +(60:1) arc (0:180:1);
\draw[densely dotted] (0,0) +(30:.577) -- +(-60:1); 
\draw[line width=.7] (0,0) +(30:.577) -- +(120:1); 
\draw[line width=.7] (0,0) +(30:.577) -- +(30:1.732); 
\node at (.55,-1.7) {$S_5$}; \draw[->] (140:1)  -- +(227:.01); }\hspace{.3cm}%tikzmath
\tikzmath[scale=.45]{
\draw[line width=.7] (120:1) arc (120:300:1);
\draw[line width=.7] (1,0) +(60:1) arc (60:-120:1);
\draw[line width=.7] (1,0) +(60:1) arc (0:180:1);
\draw[densely dotted] (0,0) +(30:.577) -- +(-60:1); 
\draw[densely dotted] (0,0) +(30:.577) -- +(120:1); 
\draw[densely dotted] (0,0) +(30:.577) -- +(30:1.732); 
\node at (.55,-1.7) {$S_6$}; \draw[->] (.5,1.86) ++ (-.05,0) -- +(-.01,0); } %tikzmath
\]
but we only know how to construct it non-canonically.
We will not discuss this construction.
\end{note} 

\subsubsection*{Unitality of fusion}
The unitality of fusion of sectors can be formulated as follows.
Recall that given a sector $H\in\Rep(\cala)$ on the standard circle and given another circle $S$, we denote by $H(S)\in \Rep_S(\cala)$ the corresponding sector on $S$ (well defined up to non-canonical isomorphism).

\begin{lemma}\label{lem: vacuum * vacuum = vacuum -- PRE}
Let $S_1$, $S_2$, $S_3$ and $I=S_1\cap S_2$ be as in \eqref{eq: Theta-graph}.
Then for any sector $H\in\Rep(\cala)$, there exists (non-canonical) unitary isomorphisms of $S_3$-sectors
\begin{equation*} 
\begin{split}
H(S_1) \boxtimes_{\cala(I)} H_0(S_2) \,\,\cong\,\, H(S_3)&\qquad\qquad\quad\\
\text{and}\qquad H_0(S_1) \boxtimes_{\cala(I)} H(S_2) \,\,\cong\,\, H(S_3)&.
\end{split}
\end{equation*}
\end{lemma}

\begin{proof}
We only show the first equality.
Let $K$ and $L$ be as above, let $j_2$ be the involution of $S_2$ coming from the action of $\mathfrak S_3$ on $S_1\cup S_2$,
and let $\varphi:=\mathrm{Id}_K\cup j_2|_L:S_3\to S_1$.
By definition, we can take $H_0(S_2) = L^2(\cala(I))$, with $S_2$-sector structure induced by $j_2$.
We then have
\[
H(S_1) \boxtimes_{\cala(I)} H_0(S_2)
\,=\,
H(S_1) \boxtimes_{\cala(I)} L^2(\cala(I))
\,\cong\, \varphi^* H(S_1) \,\cong\, H(S_3).
\]
For $J\subset S_3$ an interval, the above isomorphism is both $\cala(J\cap K)$ and $\cala(J\cap L)$-equivariant,
and hence $\cala(J)$-equivariant by strong additivity.
\end{proof}

\begin{corollary}\label{cor: vacuum * vacuum = vacuum}
Letting $S_1$, $S_2$, $S_3$, and $I$ be as in~(\ref{eq: Theta-graph}), we have
\begin{equation}\label{eq: HoxHo=Ho}
H_0(S_1) \boxtimes_{\cala(I)} H_0(S_2)\,\,\cong\,\, H_0(S_3).
\end{equation}
\end{corollary}
\vspace{2pt}

\subsubsection*{Monoidal fusion products}
There is another (closely related) notion fusion of sectors, for which the source and target categories are the same category.
Let $S$ be a circle.
Given an interval $I\subset S$, and an involution $j\in\Diff_-(S)$ fixing $\partial I$,
we can turn any $S$-sector into an $\cala(I)$-$\cala(I)$-bimodule by equipping it with the right action induced by $\cala(j):\cala(I)^\op\to \cala(I')$.
The fusion of $\cala(I)$-$\cala(I)$-bimodules then equips the category of $S$-sectors of $\cala$ with a monoidal structure:
\smallskip
\begin{equation}\label{eq: associative operation on A-reps}
\begin{cases}
\,\,\raisebox{.1cm}{\it product\,\rm:} &\begin{split}
\Rep_S(\cala)\,\times\,\Rep_S(\cala)\,\, \xrightarrow{\,\,\boxtimes_{I}} \,\,\Rep_S&(\cala)\\
H \,\boxtimes_{I} K \,\,:=\,\, H\,\boxtimes_{\cala(I)} K\,\,\,&
\end{split}\vspace{.2cm}\\
\,\,\text{\it unit object}:& \qquad\qquad\!\! L^2(\cala(I))\,\in\,\Rep_S(\cala)\vspace{.2cm}
\end{cases}
\end{equation}
(We usually drop $j$ from the notation, but occasionally write $\boxtimes_{I\!,j}$ when we need to be precise.)
The unit object is the vacuum sector of $\cala$ associated to $S$, $I$, and $j$.

We now explain why $H \boxtimes_{I} K$ is an $S$-sector.
The algebras $\cala(I)$ and $\cala(I')$ act on $H \boxtimes_{I} K$ by their respective actions on $H$ and $K$.
If $J\subset S$ is an interval that crosses $\partial I$ once, then we have
\begin{equation}\label{eq: A(J) can be re-expressed as a fusion}
\cala(J)\,\cong\, \cala(J\cup I')\circledast_{\cala(I)} \cala(J\cup I)
\end{equation}
by Proposition \ref{prop: nets-and-fiber-product}.
Here, the two maps $\cala(J\cup I')\leftarrow\cala(I)^\op$, $\cala(I)\to \cala(J\cup I)$ used in the definition of the right-hand side of \eqref{eq: A(J) can be re-expressed as a fusion}
are induced by $j:I\to J\cup I'$, and by the inclusion $I\hookrightarrow J\cup I$, respectively.
The algebra \eqref{eq: A(J) can be re-expressed as a fusion} then acts on $H \boxtimes_{I} K$ by Corollary~\ref{cor:cala(K)-acts}. 
This construction works for any $J\subset S$ that crosses $\partial I$ once, and so $H \boxtimes_I K$ is an $S$-sector by Lemma \ref{lem: open cover of circle => sector}.

\begin{proposition}\label{prop: all monoidal structures are equivalent}
Let $S$ be a circle, $I_1,I_2\subset S$ intervals, and $j_1,j_2\in\Diff_-(S)$ involutions that fix $\partial I_1$ and $\partial I_2$, respectively.
Let $\,\boxtimes_1$ and $\,\boxtimes_2$ be the monoidal structures on $\Rep_S(\cala)$ associated to $(I_1,j_1)$ and $(I_2,j_2)$, respectively, as in \eqref{eq: associative operation on A-reps}.
Then these two monoidal structures are equivalent (but non-canonically).
\end{proposition}

\begin{proof}
The first monoidal structure is given by the data of a monoidal product $(X, Y\mapsto X\boxtimes_1 Y)$, a unit object $\mathbf 1_1$, an associator $a_1:(X\boxtimes_1 Y)\boxtimes_1 Z \stackrel{\scriptscriptstyle\sim}\to X\boxtimes_1 (Y\boxtimes_1 Z)$, a left unit map $\ell_1:\mathbf 1_1\boxtimes_1 X\stackrel{\scriptscriptstyle\sim}\to X$, and a right unit map $r_1:X\boxtimes_1 \mathbf 1_1\stackrel{\scriptscriptstyle\sim}\to X$; the second monoidal structure is similarly given by $\boxtimes_2$, $\mathbf 1_2$, $a_2$, $\ell_2$, and $r_2$.

\noindent An equivalence between these two monoidal structures consists of a unitary natural transformation $\sigma:X\boxtimes_1 Y\stackrel{\scriptscriptstyle\sim}\to X\boxtimes_2 Y$,
and a unitary isomorphism $\mu:\mathbf 1_1\stackrel{\scriptscriptstyle\sim}\to \mathbf 1_2$ making the following three diagrams commute:
\[
\tikzmath{
\node (c) at (0,1) {$(X\boxtimes_1 Y)\boxtimes_1 Z$};
\node (a) at (0,0) {$X\boxtimes_1 (Y\boxtimes_1 Z)$};
\node (d) at (5,1) {$(X\boxtimes_2 Y)\boxtimes_2 Z$};
\node (b) at (5,0) {$X\boxtimes_2 (Y\boxtimes_2 Z)$};
\draw [->] (c) --node[above]{$\scriptstyle \sigma\circ(\sigma{\scriptscriptstyle \boxtimes} 1)$} (d);
\draw [->] (a) --node[above]{$\scriptstyle \sigma\circ(1{\scriptscriptstyle\boxtimes} \sigma)$} (b);
\draw [<-] (a) --node[left]{$\scriptstyle a_1$} (c);
\draw [<-] (b) --node[right]{$\scriptstyle a_2$} (d);
}
\]
\[
\tikzmath{
\node (a) at (0,1) {$\mathbf 1_1 \boxtimes_1 X$};
\node (b) at (3,1) {$\mathbf 1_2 \boxtimes_2 X$};
\node (c) at (1.5,0) {$X$};
\draw [->] (a) --node[above]{$\scriptstyle \sigma\circ(\mu{\scriptscriptstyle \boxtimes} 1)$} (b);
\draw [->] (a) --node[left]{$\scriptstyle \ell_1$} (c);
\draw [->] (b) --node[right]{\,$\scriptstyle \ell_2$} (c);
}\qquad
\tikzmath{
\node (a) at (0,1) {$X\boxtimes_1 \mathbf 1_1$};
\node (b) at (3,1) {$X\boxtimes_2 \mathbf 1_2$};
\node (c) at (1.5,0) {$X$};
\draw [->] (a) --node[above]{$\scriptstyle \sigma\circ(1{\scriptscriptstyle \boxtimes} \mu)$} (b);
\draw [->] (a) --node[left]{$\scriptstyle r_1$} (c);
\draw [->] (b) --node[right]{$\,\scriptstyle r_2$} (c);
}
\]
The claim is that it is possible to find such a pair $(\sigma, \mu)$, but that there is no canonical choice.

Pick a diffeomorphism $\varphi\in\Diff_+(S)$ that maps $I_1$ to $I_2$, and that intertwines the involutions $j_1$ with $j_2$.
By Proposition \ref{prop: phi* == psi*}, the functor $\varphi^*$ is naturally isomorphic to the identity functor on $\Rep_S(\cala)$.
Pick such a natural isomorphism $v$.
The pair $(\sigma,\mu)$ is then given by
\[
\sigma\,:\,X\boxtimes_1 Y \xrightarrow{\,v^{-1}{\scriptscriptstyle \boxtimes}\hspace{.2mm} v^{-1}} (\varphi^*X) \boxtimes_1 (\varphi^*Y) \cong \varphi^*(X \boxtimes_2 Y) \xrightarrow{\,\,v\,\,} X\boxtimes_2 Y,
\]
and
\[
\mu\,:\,L^2(\cala(I_1))\xrightarrow{L^2(\cala(\varphi|_{I_1}))}\varphi^*\big(L^2(\cala(I_2))\big)\xrightarrow{\,\,v\,\,} L^2(\cala(I_2)). \qedhere
\]
\end{proof}

\begin{definition}\label{def: fusion v h}
Given a circle $S$, we let $H, K\mapsto H\boxtimes K$ denote any one of the monoidal structures on $\Rep_S(\cala)$ considered in Proposition \ref{prop: all monoidal structures are equivalent},
and call it ``the fusion of $H$ and $K$''. It is well defined up to non-canonical unitary isomorphism.
\end{definition}

In the case when $S$ is the standard circle, there are two important special cases of \eqref{eq: associative operation on A-reps}:
the \emph{vertical fusion} and the \emph{horizontal fusion} on $\Rep(\cala)$, given by
\[
H\boxtimes^\mathsf{v}\!K\,:=\,H\boxtimes_{S^1_\top,j}K\qquad\qquad
H\boxtimes^\mathsf{h}\!K\,:=\,H\boxtimes_{S^1_\vdash,j'}K,
\]
respectively. Here, $S^1_\top$ is the upper half of the standard circle, 
$S^1_\vdash$ its left half, and $j$ and $j'$ are the reflections given by $j(z)= \bar z$ and $j'(z)= -\bar z$.

\subsection{Central decomposition} \label{subsec:central-decomposition}

Given a conformal net $\cala$ and an orientation-preserving embedding $f:I\to J$, we will show that $\cala(f):Z(\cala(I))\to Z(\cala(J))$ is always an isomorphism.
In fact, there is an algebra $Z(\cala)$, called the \emph{center} of $\cala$, that only depends on $\cala$, and that is canonically isomorphic to $Z(\cala(I))$ for every $I$.
See~\cite[Sec.~3]{Buchholz-DAntoni-Fredenhagen(univ-loc-structure)} 
for a similar discussion.

\begin{proposition} \label{lem: Z(A)}
Let $\cala$ be a conformal net.
Then there is an abelian von Neumann algebra $ Z(\cala)$ such that for
every interval $I$ we have a canonical isomorphism $ Z(\cala(I))\xrightarrow{\scriptscriptstyle\sim}  Z(\cala)$,
and such that for each embedding $J\hookrightarrow I$,
the inclusion $\cala(J)\hookrightarrow\cala(I)$ induces a commutative diagram
\begin{equation}\label{eq: three Z(A)'s}
\def\hh{2}\def\vv{1}
\tikzmath{
\draw[->] 
(-\hh,0) node (a) {$Z(\cala(J))$} 
(\hh,0) node (b) {$Z(\cala)$} 
(0,-\vv) node (c) {$Z(\cala(I))$}
(a.east) -- (b.west);
\draw[->] (a) -- (c);
\draw[->] (c) -- (b);
}
\end{equation}
\end{proposition}

\begin{proof}
We first show that $\zeta(I):= Z(\cala(I))$
is a functor from intervals with orientation-preserving embeddings to von Neumann algebras.  We define $Z(\cala):=\colim \zeta$.
The natural map $Z(\cala(I))\to Z(\cala)$ will be an isomorphism if and only if $\zeta$ is equivalent to a constant functor.
Checking the latter condition involves two things:
({\it a}) given an embedding $J\hookrightarrow I$, we show the induced map $\cala(J)\hookrightarrow\cala(I)$
sends $Z(\cala(J))$ isomorphically onto $Z(\cala(I))$,
and ({\it b})
given two orientation-preserving embeddings $\alpha,\beta:J\hookrightarrow I$, we show that
the induced isomorphisms $\alpha_*,\beta_*:Z(\cala(J))\to  Z(\cala(I))$ are equal to each other.

(\hspace{-.1mm}{\it a}) Let $J\hookrightarrow I$ be an embedding.
Without loss of generality, we may assume that $I$ and $J$ share a boundary point.
Let $J^c$ be the closure in $I$ of the complement of $J$.
By Lemma \ref{prop: [Haag duality]}, $\cala(J)$ and $\cala(J^c)$ are each other's commutants in $\cala(I)$.
Hence
\[
 Z\big(\cala(J)\big)=
 Z\big(\cala(J^c)\big)=\cala(J)\cap\cala(J^c).
\]
An element of $\cala(J)\cap\cala(J^c)$ commutes with both $\cala(J^c)$ and $\cala(J)$,
and so it also commutes with $\cala(I)=\cala(J)\vee\cala(J^c)$.
An element of $ Z(\cala(I))$ commutes with both $\cala(J)$ and $\cala(J^c)$,
and is therefore in the intersection $\cala(J^c)\cap\cala(J)$ of their commutants.
It follows that $ Z(\cala(I))= Z(\cala(J))$.

({\it b}) 
Without loss of generality, we may take $I=J$, and $\beta=\mathrm{Id}_I$.
Pick intervals $I\subset I_+\supset K$, and a diffeomorphism $\varphi:I_+\to I_+$ that restricts to $\alpha$ on $I$ and to the identity on $K$:
\[
\def\hh{1.7}\def\vv{1}
\tikzmath{
\node (a) at (-\hh,\vv) {$I$};\node (b) at (0,\vv) {$I_+$};\node (c) at (\hh,\vv) {$K$};
\node (a') at (-\hh,0) {$I$};\node (b') at (0,0) {$I_+$};\node (c') at (\hh,0) {$K$};
\draw [->] (a) -- node[above]{$\scriptstyle i$} (b);		\draw [->] (c) -- node[above]{$\scriptstyle j$} (b);
\draw [->] (a') -- node[above]{$\scriptstyle i$} (b');		\draw [->] (c') -- node[above]{$\scriptstyle j$} (b');
\draw [->] (a) -- node[left]{$\scriptstyle \varphi|_I=\alpha$} (a');
\draw [->] (b) -- node[left]{$\scriptstyle \cong$} node[right]{$\scriptstyle \varphi$} (b');
\draw [->] (c) -- node[right]{$\scriptstyle \varphi|_K=\mathrm{Id}_K$} (c');
}
\]
As the maps $i_*: Z(\cala(I))\to  Z(\cala(I_+))$ and $j_*: Z(\cala(K))\to  Z(\cala(I_+))$ are isomorphisms,
it follows that $\alpha_*: Z(\cala(I))\to  Z(\cala(I))$ is the identity map.

We have now canonically identified every $Z(\cala(I))$ with $Z(\cala)$.
To show that every diagram \eqref{eq: three Z(A)'s} is commutative, we still need to treat the case of orientation-reversing maps.
For that purpose, it is enough to analyze the identity map from $J$ to $\bar J$.
Namely, we need to show that the map 
\[
Z(\cala)\cong Z(\cala(J))\xrightarrow{\,i_*\,} Z(\cala(\bar J))\cong Z(\cala)
\]
induced by
$i:=\mathrm{Id}_J:J\to\, \bar J$ is the identity.
Assuming the contrary, there would be a non-zero central projection $p\in \cala(J)$
that is orthogonal to its image $q:=i_*(p)$.
Letting $I$ be an interval containing $J$ and sharing one endpoint (as in the statement of the vacuum axiom), we would then have the following commutative diagram
\[
\qquad\tikzmath{
\matrix [matrix of math nodes,column sep=1cm,row sep=4mm]
{ 
|(x1)| p\otimes q	\pgfmatrixnextcell [-1cm] \in	\pgfmatrixnextcell [-1.1cm] |(a2)| \cala(I) \ox_{\alg} \cala(I)^\op\\
|(x)| p\otimes q		\pgfmatrixnextcell \in			\pgfmatrixnextcell [-1.1cm] |(a1)| \cala(J) \ox_{\alg} \cala(J)^\op\\
|(y)| p\otimes p		\pgfmatrixnextcell \in			\pgfmatrixnextcell |(a)| \cala(J) \ox_{\alg} \cala(\bar J) \\ 
|(z)| \,p^2			\pgfmatrixnextcell 			\pgfmatrixnextcell |(c)| \cala(J \cup \bar{J}) \\ 
}; 
\node (b) at ($(a1.center)!.5!(a.center) + (3.5,0)$) {$\bfB(L^2\cala(I))$};
\draw[->] (a2) -- (b);
\draw[->] (a1) -- (a2);
\draw[->] (a1) -- (a);
\draw[->] (a) -- (c);
\draw[->,dashed] (c) -- (b);
\draw[->]  (x) -- (x1);
\draw[->] (x) -- (y);
\draw[->] (y) -- (z);
\foreach \zz/\loc in {x/north, x/south, y/south}
{\draw ($(\zz.\loc) + (.07,0)$) -- ($(\zz.\loc) - (.07,0)$);}
\node[right, xshift = -5, yshift = -1.4] at (z.east) {$=p\,\,\in$};
}
\]
The element $p\otimes p$ goes to $p^2$ under the bottom map because it is the product of $p\otimes 1$ with $1\otimes p$, and both get mapped to $p$.
Since $p\otimes q$ acts as zero on $L^2\cala(I)$, by the commutativity of the above diagram, so must $p$,
contradicting the fact that $p\not = 0$.
\end{proof}

Given an abelian von Neumann algebra $A$, let $\mathrm{Spec}(A)$ refer to any nice measure space $X$ equipped with an isomorphism $L^\infty(X)\cong A$.  Any von Neumann algebra can be written as a direct integral of 
factors~\cite{Dixmier(vN-algebras)} indexed over Spec of its center.
Applying this to the algebras $\cala(I)$, one can then write every conformal net $\cala$ as a direct integral of irreducible conformal nets:
\begin{equation}\label{eq: Dir Int}
\cala = \int_{x\in \mathrm{Spec}(Z(\cala))}^\oplus \cala_x
\end{equation}
Here, we have secretly used that $\INT$ is equivalent to a category with countably many objects (actually one object) and has separable hom spaces.

Conformal nets form a category: a morphism $\cala\to \calb$ is a natural transformation
$\tau:\cala \to \calb$ that assigns to each interval $I$ a unital 
homomorphism $\tau_I:\cala(I) \to\calb(I)$ of von Neumann algebras.
Objectwise spatial tensor product defines a symmetric monoidal structure on that category.

\begin{lemma}
Let $\cala$ and $\calb$ be conformal nets, and let $\tau:\cala\to \calb$ 
be a natural transformation.
Then $\tau(Z(\cala))\subset Z(\calb)$.
\end{lemma}
\begin{proof}
Let $I=[0,1]$, $J=[1,2]$, and $K=[0,2]$.
By Proposition \ref{lem: Z(A)}, the natural maps
$Z(\cala(I))\to Z(\cala(K)) \leftarrow Z(\cala(J))$ are isomorphisms.
By locality, the algebra $\tau(Z(\cala(I)))$ commutes with $\calb(J)$, 
and the algebra $\tau(Z(\cala(J)))$ commutes with $\calb(I)$.
It follows that the image of $Z(\cala(K))$ commutes with 
$\calb(I)\vee\calb(J)=\calb(K)$.
\end{proof}

\begin{corollary}
Let $\tau$ be a natural transformation between semisimple conformal nets (Definition \ref{def:irreducilbe-nets}).
Then $\tau$ is a direct sum of maps of the form
\[
\cala\,\longrightarrow\, \bigoplus_{i=1}^n\cala\,\xrightarrow{\oplus \iota_i}\, \bigoplus_{i=1}^n\calb_i,
\]
where $\cala$ and $\calb_i$ are irreducible conformal nets,
$n\in\IN$ ($n=0$ allowed), $\cala\to \bigoplus_{i=1}^n\cala$ is the diagonal map, and $\iota_i:\cala\to\calb_i$ are inclusions.
\end{corollary}

In view of the above results, the study of arbitrary conformal nets reduces to that of irreducible conformal nets.
From now on, we shall therefore assume that our conformal 
nets are irreducible.

\subsection{Conformal embeddings}

In this section, we discuss two classes of morphisms of special interest, namely finite morphisms and conformal embeddings, and we show that all finite morphisms are conformal embeddings.

Let $\cala$ and $\calb$ be irreducible conformal nets.
Recall that given an interval $I$, we denote by $\Diff_0(I)$ the subgroup of diffeomorphisms that restrict to the identity near the boundary of $I$.

\begin{definition} \label{def:finite nets maps}
A natural tranformation $\tau:\cala\to \calb$ is called \emph{finite} if for every interval $I$ the map $\tau_I:\cala(I)\to \calb(I)$ is a finite homomorphism (Definition \ref{def: finite homomrphism}).
\end{definition}

\nid We borrow the following terminology from affine Lie algebras~\cite{Bais-Bouwknegt(classification-subgroup-truncations)}:

\begin{definition}
A natural tranformation $\tau:\cala\to \calb$ is called a \emph{conformal embedding} if for every $\varphi\in\Diff_0(I)$, and every unitary $u\in \cala(I)$,
\[
\Ad(u) = \cala(\varphi)\quad\Rightarrow\quad \Ad (\tau(u)) = \calb(\varphi).
\]
\end{definition}

\begin{lemma}\label{lem: Rel com of A(I) in B(I) is C}
Let $\tau:\cala\to \calb$ be a finite natural transformation between irreducible conformal nets.
Then the relative commutant of $\tau(\cala(I))$ inside $\calb(I)$ is trivial.
\end{lemma}

\begin{proof}
Let $I=[0,1]$, $J=[1,2]$, $K=[0,2]$,
and let $\cala(I)^c$, $\cala(J)^c$, $\cala(K)^c$ be the commutants of $\cala(I)$, $\cala(J)$, $\cala(K)$ inside $\calb(K)$.
By locality, the algebras $\cala(I)^c\cap \calb(I)$ and $\cala(J)^c\cap \calb(J)$ commute with $\cala(K)=\cala(I)\vee \cala(J)$.
We therefore have inclusions
\[
\cala(I)^c\cap\calb(I)\,\hookrightarrow\, \cala(K)^c\,\hookleftarrow\, \cala(J)^c\cap\calb(J).
\]
Since $\tau$ is finite, these algebras are finite-dimensional by Lemma \ref{lem: rel comm is finite dim}.
They are all of the same dimension, and so the above inclusions are actually isomorphisms.
Moreover, $\cala(I)^c\cap\calb(I)$ and $\cala(J)^c\cap\calb(J)$ commute with $\calb(J)$ and $\calb(I)$, respectively.
It follows that $\cala(I)^c\cap\calb(I)= \cala(J)^c\cap\calb(J)$ is central in $\calb(I)\vee\calb(J)=\calb(K)$.
This finishes the argument since, by assumption, $Z(\calb(K))=\IC$.
\end{proof}

\begin{proposition}\label{prop: finite homom ==> conf emb}
If $\tau:\cala\to\calb$ is finite, then it is a conformal embedding.
\end{proposition}

\begin{proof}
Let $\varphi\in\Diff_0(I)$ be a diffeomorphism.
By the inner covariance axiom, there exists a unitary $v\in\calb(I)$ such that $\Ad(v)=\calb(\varphi)$.
Similarly, there exists $u\in\cala(I)$ such that $\Ad(u)=\cala(\varphi)=\calb(\varphi)|_{\cala(I)}$, where the restriction occurs along the morphism $\tau$.
It follows that $\Ad(\tau(u)v^*)|_{\cala(I)}=\mathrm{Id}_{\cala(I)}$, and hence that $\tau(u)v^*\in \cala(I)'\cap\calb(I)$.
By Lemma \ref{lem: Rel com of A(I) in B(I) is C}, $\tau(u)$ is therefore a scalar multiple of $v$.
It follows that $\Ad(\tau(u))=\Ad(v)=\calb(\varphi)$.
\end{proof}

%=====================================================================

\section{Covariance for the vacuum sector}
  \label{sec:covariance}

In this section, we study the natural projective actions of $\Diff(S^1)$ and its various subgroups on the vacuum sector of a conformal net.
The main result of this section is that the vacuum sector construction can be upgraded to a functor from the category of conformal circles to the category of Hilbert spaces.

From now on, all conformal nets are irreducible, unless stated otherwise.

\subsection{Implementation of diffeomorphisms}
Given a Hilbert space $H$, we let $\U_\pm(H)=\U(H)\,\cup\, \U_-(H)$ be the group of unitary and anti-unitary operators on $H$,
equipped with the strong operator topology.
This is a topological group\footnote{This might be surprising since, on $\bfB(H)$, the map $a\mapsto a^*$ is not continuous for the strong topology.}.
Note that, on $\U_\pm(H)$, the strong, weak, and ultraweak topologies all agree.

Let $S$ be a circle, let $I_0 \subset S$ be an interval, and let $j \colon S \to S$ be an orientation-reversing involution that fixes $\dd I_0$.
For a conformal net $\cala$, let $H_0 := L^2(\cala(I_0))\in \Rep_S(\cala)$ be the vacuum sector associated to $S$, $I_0$ and $j$, as in Section~\ref{subsec:sectors-for-nets}.

\begin{definition}\label{def:implements}
Let ${\varphi\in \Diff(S)}$ be a diffeomorphism, and
let $u \in \U_\pm(H_0)$ be an operator that is
complex linear if $\varphi\in\Diff_+(S)$, and complex antilinear otherwise.
We say that $u$ \emph{implements} $\varphi$ if 
$$u:H_0 \to \varphi^*H_0$$
is a morphism of $\cala$-sectors.
\end{definition}

Unpacking the definition, a unitary $u \in \U(H_0)$ implements a diffeomorphism $\varphi\in \Diff_+(S)$ if
\begin{equation} \label{eq:implement}
  \cala(\varphi) (a) = \Ad(u)(a) = u\, a\, u^*
\end{equation}
for all $I\subset S$ and $a \in \cala(I)$, and that
an anti-unitary $u \in \U_-(H_0)$ implements an orientation-reversing diffeomorphism $\varphi\in \Diff_-(S)$ if
\begin{equation} \label{eq:implement*}
  \cala(\varphi) (a) = \Ad(u)(a^*) = u\, a^* u^*.
\end{equation}
Here, the adjoint $u^*$ of an antilinear operator $u$ is defined by $\langle u\xi,\eta\rangle=\overline{\langle\xi,u^*\eta\rangle}$. 

Throughout this section, we will adopt the notation $I_0'$ for the closure of $S\setminus I_0$.

\begin{lemma}\label{lem: enough to check I_0 and I_0'}
Let $u$ be an (anti-)unitary operator on the Hilbert space $H_0$.
In order to check that $u$ implements a diffeomorphism $\varphi$,
it is enough to check \eqref{eq:implement} or \eqref{eq:implement*} for $a \in \cala(I_0)$ and $a \in \cala(I_0')$.
\end{lemma}

\begin{proof}
Let $\varphi$ be a diffeomorphism, and let $u$ be an (anti-)unitary on $H_0$ that satisfies \eqref{eq:implement} (or \eqref{eq:implement*}) for all $a$ in $\cala(I_0)$ and $\cala(I_0')$.
Let $I\subset S$ be an interval.
Consider the subalgebra of all elements $a\in \cala(I)$ that satisfy \eqref{eq:implement} (or \eqref{eq:implement*}).
That subalgebra  is closed in the ultraweak topology and contains $\cala(I\cap I_0)$ and $\cala(I\cap I_0')$ by assumption.
By strong additivity, it is therefore equal to $\cala(I)$.
\end{proof}

Recall that given a von Neumann algebra $A$, the modular conjugation $J:L^2(A)\to L^2(A)$ is an antilinear involution that satisfies $J(a\xi b)=b^*J(\xi)a^*$.

\begin{lemma}\label{lem: j is implemented}
The modular conjugation $J$ for $L^2(\cala(I_0))$ (see~\eqref{eq: main property of J}) implements $j$.
\end{lemma}

\begin{proof}
By Lemma \ref{lem: enough to check I_0 and I_0'}, it is enough to verify \eqref{eq:implement*} for $a\in\cala(I_0)$ and for $b\in\cala(I_0')$. 
For $a\in\cala(I_0)$ and $\xi\in H_0$, we have $\cala(j)(a)\xi=\xi a=J(a^* J(\xi))$, and
for $b\in\cala(I_0')$, we have $\cala(j)(b)\xi=J(J(\xi)\cala(j)(b^*)) =J(b^* J(\xi))$.
These equation are equivalent to \eqref{eq:implement*} because $J$ is self-adjoint.
\end{proof}

\begin{corollary}\label{phi* == psi* -- orientation-reversing}
For any diffeomorphism $\varphi\in\Diff(S)$, the $S$-sectors $H_0$ and $\varphi^*H_0$ are unitarily isomorphic.
\end{corollary}

\begin{proof}
For $\varphi\in\Diff_+(S)$, this is Corollary \ref{phi* == psi*}.
For $\varphi\in\Diff_-(S)$, write $\varphi=j\circ \psi$ for some $\psi\in\Diff_+(S)$.
By the previous lemma, we have $j^*H_0\cong H_0$. Therefore $\varphi^* H_0= \psi^* j^*H_0\cong \psi^*H_0 \cong H_0$.
\end{proof}

\begin{lemma} \label{lem:L^2=conformal-implementaion}
Let $\varphi \in \Diff_+(S, \dd I_0)$ be a diffeomorphism that commutes with $j$, and let $\varphi_0:=\varphi|_{I_0}$.
Then $L^2(\cala(\varphi_0))$ implements $\varphi$.
\end{lemma}

\begin{proof}
  By Lemma \ref{lem: enough to check I_0 and I_0'}, it is enough to verify \eqref{eq:implement} for $a\in\cala(I_0)$ and for $b\in\cala(I_0')$. 
  Given a von Neumann algebra $A$, and an automorphism $f:A\to A$, we always have $L^2(f)(a\xi)=f(a)L^2(f)(\xi)$.
  Substituting $A=\cala(I_0)$, $f=\cala(\varphi_0)$, and $\xi=L^2(\cala(\varphi_0))^*\eta$ for some $\eta\in H_0$, we get
  \[
  L^2(\cala(\varphi_0))\, a\, L^2(\cala(\varphi_0))^*\eta = \big(\cala(\varphi_0)a\big)\eta,
  \]
  which shows that \eqref{eq:implement} holds for $a\in \cala(I_0)$.
  
  Given an automorphism $f$ of a von Neumann algebra $A$, we also have $L^2(f)(\xi a)=\big(L^2(f)(\xi)\big)f(a)$.
  Substituting $A=\cala(I_0)$, $f=\cala(\varphi_0)$, $\xi=L^2(\cala(\varphi_0))^*\eta$, and $a=\cala(j)b$ for some $b\in\cala(I_0')$, we get
\[
L^2(\cala(\varphi_0))\big(\big(L^2(\cala(\varphi_0))^*\eta\big)\cala(j)b\big)=\eta\,\big(\cala(\varphi_0)\cala(j)b\big).
\]
The left hand side is given by 
\[
L^2(\cala(\varphi_0))\big(\big(L^2(\cala(\varphi_0))^*\eta\big)\cala(j)b\big)=L^2(\cala(\varphi_0))\, b\, L^2(\cala(\varphi_0))^*\eta
\]
and the right-hand side is 
\[
\eta\,\big(\cala(\varphi_0)\cala(j)b\big)=\big(\cala(j)\cala(\varphi_0)\cala(j)b\big)\,\eta=\big(\cala(\varphi|_{I_0'})b\big)\eta,
\]
which shows that \eqref{eq:implement} holds for $b\in \cala(I_0')$.
\end{proof}

Recall that, by Remark~\ref{rem: L^2(A^op)}, 
an anti-isomorphism $f:A\to B$ induces a linear isomorphism
$L^2(f):L^2(A)\to L^2(B)$ that exchanges left and right actions, that is, 
such that
\begin{equation}\label{eq:L^2(f)(a_1 xi a_2)}
L^2(f)(a_1\xi a_2)=f(a_2) L^2(f)(\xi) f(a_1).
\end{equation}

\begin{lemma} \label{lem:L^2=conformal-implementaion OP}
Let $\varphi \in \Diff_-(S)$ be a diffeomorphism that commutes with $j$ and exchanges the endpoints of $I_0$.
Let $\varphi_0:=\varphi|_{I_0}$.
Then $L^2(\cala(\varphi_0))\circ J$ implements $\varphi$.
\end{lemma}

\begin{proof}
Let $\varphi_0':=\varphi|_{I_0'}$.  Given $a\in \cala(I_0)$, then by \eqref{eq:L^2(f)(a_1 xi a_2)}, we have
\[
\begin{split}
\cala(\varphi_0)(a)\xi&=
L^2(\cala(\varphi_0))\Big(\big(L^2(\cala(\varphi_0))^*\xi\big)a\Big)\\&=
L^2(\cala(\varphi_0))Ja^*JL^2(\cala(\varphi_0))^*\xi\\&=
\big(L^2(\cala(\varphi_0))J\big)a^*\big(L^2(\cala(\varphi_0))J\big)^*\xi\hspace{.04cm}.\hspace{.4cm}
\end{split}
\]
For $b\in \cala(I_0')$, we also have that
\[
\begin{split}
\hspace{.55cm}\cala(\varphi_0')(b)\xi&=
\xi\big(\cala(j)\cala(\varphi_0')(b)\big)\\&=
\xi\big(\cala(\varphi_0)\cala(j)(b)\big)\\&=
L^2(\cala(\varphi_0))\big(\cala(j)(b)\big(L^2(\cala(\varphi_0))^*\xi\big)\big)\\&=
L^2(\cala(\varphi_0))J\Big(\big(JL^2(\cala(\varphi_0))^*\xi\big)\cala(j)(b^*)\Big)\\&=
\big(L^2(\cala(\varphi_0))J\big)b^*\big(L^2(\cala(\varphi_0))J\big)^*\xi\,,
\end{split}
\]
which finishes the proof by Lemma \ref{lem: enough to check I_0 and I_0'}.
\end{proof}

Given a Hilbert space $H$, equip $\PU_\pm(H)=\PU(H)\,\cup\, \PU_-(H):=\U_\pm(H)/S^1$ with the quotient strong topology\footnote{With this topology, the projection $\U_\pm(H)\to \PU_\pm(H)$ is a locally trivial bundle.}.
Recall that $H_0 := L^2(\cala(I_0))$ denotes the vacuum sector associated to $S$, $I_0$ and $j$, and that $\cala$ is assumed to be irreducible.

\begin{proposition} \label{prop:projective-implementation-diffeo} 
Let $\cala$ be a conformal net, and let $H_0$ be as above.
Then there is a unique continuous representation $\Diff(S)\to \PU_\pm(H_0)$, $\varphi\,\, \mapsto\,\, [u_\varphi]$
such that
\begin{enumerate}
\item $u_\varphi$ is complex linear for $\varphi\in\Diff_+(S)$, and complex antilinear otherwise.
\item $u_\varphi$ implements $\varphi$.
\end{enumerate}
Here, $u$ denotes any preimage of $[u]$ in $\U_\pm(H_0)$.
\end{proposition}

\begin{proof}
  The vacuum $H_0$ is an irreducible sector.
  By Schur's lemma, the implementation $u_\varphi$ of a diffeomorphism $\varphi$ is therefore unique up to phase.
  Moreover, an implementation always exists since, by Corollary \ref{phi* == psi* -- orientation-reversing}, $H_0\cong \varphi^*H_0$ for any diffeomorphism $\varphi$.

  It remains to show that the homomorphism $\Diff(S)\to \PU_\pm(H_0)$ is continuous.  
     For a subinterval $K \subset I$ whose boundary is contained in
     the interior of $I$, write $\Diff_{0,K}(I)$ for the diffeomorphisms
     of $I$ that fix the complement of $K$ pointwise. 
  The restrictions $\Diff_{0,K}(I)\to \PU_\pm(H_0)$ 
  are continuous by Lemma~\ref{lem:cont-inner-cov} below.
  The result then follows as the $\mathcal C^\infty$ topology on 
  $\Diff(S)$ is the finest one for which the inclusions
  $\Diff_{0,K}(I)\hookrightarrow \Diff(S)$ are continuous.
\end{proof}

Given an interval $I$, by the inner covariance axiom, we have a group homomorphism $\Diff_0(I) \to \Inn(\cala(I))\cong \PU(\cala(I)):=\U(\cala(I))/S^1$.
By definition the net $\cala$ is continuous for the $\calc^\infty$ topology on $\Diff_0(I)$ and the $u$-topology on $\Inn(\cala(I))$
(note that we do not claim that the $u$-topology and the quotient strong topology coincide under the identification $\Inn(\cala(I)) = \PU(\cala(I))$).

  \begin{lemma}
    \label{lem:cont-inner-cov}
    Let $\Diff_{0,K}(I)$ be as in the previous proof.
    Then the map $\Diff_{0,K}(I) \to \PU(\cala(I))$
    is continuous with respect to the $\calc^\infty$ topology on 
    $\Diff_{0,K}(I)$ and the quotient strong topology on 
    $\PU(\cala(I))=\U(\cala(I))/S^1$.
  \end{lemma}

  \begin{proof} 
  Pick an enlargement $\hat I$ of $I$, such that $I$ is contained in the interior of $\hat I$.
  By the split property axiom, the subfactor $\cala(I) \subset \cala(\hat I)$ satisfies
  the assumption of Proposition~\ref{prop:subspace-topology-on-PU(A_0)}.
  The two vertical maps in the following diagram are therefore homeomorphisms onto their images:
\[
\tikzmath{
\matrix [matrix of math nodes,column sep=1.2cm,row sep=7mm]
{ 
|(a)| \Diff_{0,K}(I) \pgfmatrixnextcell |(b)| \PU(\cala(I))\\ 
|(c)| \Diff(\hat I) \pgfmatrixnextcell |(d)| \Aut(\cala(\hat I)).\\ 
}; 
\draw[->] (a) -- (b);
\draw[->] (c) -- (d);
\draw[->] (a) -- (c);
\draw[->] (b) -- (d);
}
\]
The map $\Diff(\hat I) \to \Aut(\cala(\hat I))$ is continuous by our definition of conformal nets, and therefore so is
the map $\Diff_{0,K}(I) \to \PU(\cala(I))$.
\end{proof}

%-------------------------------------------------------------
\subsection{Conformal circles and their vacuum sectors}

\subsubsection*{Conformal circles}
  The group $\Conf(S^1)$ of conformal maps of the standard circle 
  $S^1 = \{ z \in \IC : |z| = 1 \}$ consists of all maps of the form
  \begin{equation*} 
  z \mapsto \frac{\alpha z + \beta}{\bar \beta z + \bar \alpha}
  \end{equation*}
  where $\alpha, \beta \in \IC$, $|\alpha|^2 - |\beta|^2 = \pm 1$.
Those are the maps that extend to conformal transformations of unit disc in $\IC$.
We let
\[
\Conf_+(S^1)\,\cong\, \mathrm{PSU}(1,1)=\Big\{\big(\begin{smallmatrix}a & b \\ c & d\end{smallmatrix}\big)\,\Big|\,
\text{\small det}\big(\begin{smallmatrix}a & b \\ c & d\end{smallmatrix}\big)=1,
\,\big(\begin{smallmatrix}a & b \\ c & d\end{smallmatrix}\big)^{-1}=
\big(\begin{smallmatrix}\bar a & -\bar c \\ -\bar b & \bar d\end{smallmatrix}\big)
\Big\}\big/\{\pm1\}
\]
be the subgroup where $|\alpha|^2 - |\beta|^2 = 1$,
  and $\Conf_-(S^1):=\Conf(S^1)\setminus \Conf_+(S^1)$ its complement.
  Elements in the former are orientation-preserving maps, 
  while elements in the latter are orientation-reversing.
  The subgroup $\Conf_+(S^1)$ can also be identified with 
  $\PSL_2(\IR)$ by congugating it with the Cayley transform.
  Explicitely, the identification sends the matrix $\big(\begin{smallmatrix} \alpha&\beta\\\bar\beta&\bar\alpha \end{smallmatrix}\big)\in \mathrm{PSU(1,1)}$ to
  $  
  \frac12\scriptstyle
    \big(\begin{smallmatrix} -1&1\\-i&-i  \end{smallmatrix}\big)
    \big(\begin{smallmatrix} \alpha&\beta\\\bar\beta&\bar\alpha \end{smallmatrix}\big)
    \big(\begin{smallmatrix} -1&i\\1&i \end{smallmatrix}\big)\textstyle\in \PSL_2(\IR)$.

\begin{definition} \label{def:conformal-circle}
  Let $S$ be a circle. A \emph{conformal structure} $\tau$ on $S$ is an  orbit of the right $\Conf(S^1)$ action on the set $\Diff(S,S^1)$.
  Thus, a conformal structure on $S$ is an identification $S \to S^1$ that is only determined up to elements of $\Conf(S^1)$.
A \emph{conformal circle} is a circle equipped with a conformal structure.

If $S$ and $S'$ are conformal circles, we write
$\Conf(S,S')$ for the set of all diffeomorphisms $S \to S'$ that are compatible with the conformal structures,
and abbreviate $\Conf(S,S)$ by $\Conf(S)$.
We also let $\Conf_+(S,S') := \Conf(S,S') \cap \Diff_+(S,S')$, $\Conf_+(S) := \Conf_+(S,S)$, and similarly for $\Conf_-$.
\end{definition}

The collection of all conformal circles forms a category. The objects of that category are conformal circles (always equipped with an orientation), and the morphisms from $S$ to $S'$ are 
given by $\Conf(S,S')=\Conf_+(S,S')\sqcup \Conf_-(S,S')$.

\subsubsection*{The vacuum sector functor}
The main goal of this section is to prove the following theorem.
Recall that our conformal nets are assumed to be irreducible.

\begin{theorem}\label{thm: Vaccum Sector}
Let $\cala$ be a conformal net.
There is a functor $$S \mapsto H_0(S,\cala)$$ from the category of conformal circles to the category of complex Hilbert spaces.
For $\varphi$ a conformal map, the operator $H_0(\varphi,\cala)$ is unitary when $\varphi$ is orientation-preserving, and anti-unitary when $\varphi$ is orientation-reversing.
Moreover, $H_0(S,\cala)$ is naturally equipped with the following structure:
\begin{enumerate}
\item \label{prop:vacuum-sector:local-actions}
The Hilbert space $H_0(S,\cala)$ is an $S$-sector of $\cala$, and it is a vacuum sector of $\cala$ associated to $S$, in the sense of Definition~\ref{def: non canonical vacuum}.
 \item \label{prop:vacuum-sector:covariant}
The representation $\varphi \mapsto H_0(\varphi,\cala)$ of $\Conf(S)$ on $H_0(S,\cala)$ extends to a continuous
projective representation $\varphi \mapsto [u_\varphi]$ of $\Diff(S)$ satisfying the two conditions listed in Proposition \ref{prop:projective-implementation-diffeo}.
\item \label{prop:vacuum-sector:L^2}
For any interval $I \subset S$, there is a unitary isomorphism of $S$-sectors
\[
v_I \colon H_0(S,\cala) \to L^2(\cala(I))
\]
between $H_0(S,\cala)$ and the vacuum sector associated to $S$, $I$, and $j$,
where $j\in\Conf_-(S)$ is the involution that fixes $\partial I$ (see Lemma \ref{lem:properties-of-Conf} below).
Moreover, given two conformal circles $S$ and $S'$, an interval $I\subset S$, and a conformal map $\varphi\in\Conf(S,S')$,
the diagrams
\begin{equation}\label{eq: naturality of v_I}
\qquad\quad\tikzmath{
\matrix [matrix of math nodes,column sep=1.4cm,row sep=7mm]
{ 
|(a)| H_0(S,\cala) \pgfmatrixnextcell |(c)| L^2(\cala(I))\\ 
|(b)| H_0(S',\cala) \pgfmatrixnextcell |(d)| L^2(\cala(\varphi(I)))\\ 
}; 
\draw[->] (a) -- node [left, xshift=-1]	{$\scriptstyle H_0(\varphi,\cala)$} (b);
\draw[->] (c) -- node [right, xshift=1]	{$\scriptstyle L^2(\cala(\varphi))$} (d);
\draw[->] (a) -- node [above]		{$\scriptstyle v_{I}$} (c);
\draw[->] (b) -- node [above]	{$\scriptstyle v_{\varphi(I)}$} (d);
}\qquad\text{if}\,\,\varphi\in\Conf_+(S,S'),
\end{equation}
\begin{equation}\label{eq: naturality of v_I -- BIS}
\qquad\quad\tikzmath{
\matrix [matrix of math nodes,column sep=1.4cm,row sep=7mm]
{ 
|(a)| H_0(S,\cala) \pgfmatrixnextcell |(c)| L^2(\cala(I))\\ 
|(b)| H_0(S',\cala) \pgfmatrixnextcell |(d)| L^2(\cala(\varphi j(I)))\\ 
}; 
\draw[->] (a) -- node [left, xshift=-1]	{$\scriptstyle H_0(\varphi,\cala)$} (b);
\draw[->] (c) -- node [right, xshift=1]	{$\scriptstyle L^2(\cala(\varphi j))\circ J$} (d);
\draw[->] (a) -- node [above]		{$\scriptstyle v_{I}$} (c);
\draw[->] (b) -- node [above]	{$\scriptstyle v_{\varphi j(I)}$} (d);
}\quad\text{if}\,\,\varphi\in\Conf_-(S,S'),
\end{equation}
commute.

\item \label{prop:vacuum-sector:J} 
If $j \in \Conf_-(S)$ is the involution that fixes the boundary of $I \subset S$, then
$v_I\circ H_0(j,\cala)\circ v_I^*$ is the modular conjugation on $L^2(\cala(I))$.
\end{enumerate}
\end{theorem}

\begin{remark}\label{rem: vacuum construction}
If the conformal net $\cala$ is not irreducible, then most of Theorem \ref{thm: Vaccum Sector} remains true.
Indeed, given the direct integral decomposition \eqref{eq: Dir Int} of $\cala$, we can define $H_0(\cala,S)$ as $\int^\oplus_{x\in\mathrm{Spec}(Z(\cala))} H_0(\cala_x,S)$.
The only piece of structure that is no longer present on $H_0(\cala,S)$ is the projective action of $\Diff(S)$.
The issue is that the direct sum or direct integral of two projective representations is typically no longer a projective representation (except if the 2-cocycles are equal).
\end{remark}

Before embarking on the proof of Theorem \ref{thm: Vaccum Sector},
we list a few of its consequences.

\begin{definition}
\label{def:vacuum-rep-on-conformal}
Let $S$ be a conformal circle and $\cala$ a conformal net.
The Hilbert space $H_0(S,\cala)$ constructed in Theorem~\ref{thm: Vaccum Sector} is called the \emph{vacuum sector} of $\cala$ on~$S$.
\end{definition}

The vacuum sector $H_0(S,\cala)$ is a unit for Connes fusion along any interval $I\subset S$.
Indeed, given a right $\cala(I)$-module $H$,
composing the isometry $v_I \colon H_0(S,\cala) \to L^2(\cala(I))$ with the unit map $H \boxtimes_{\cala(I)} L^2(\cala(I)) \cong H$, we obtain a natural isomorphism
\begin{equation}
\label{eq:left-unit-vacuum}
H \boxtimes_{\cala(I)} H_0(S,\cala) \xrightarrow{\cong} H.
\end{equation}
Recall that $\bar I$ denotes $I$ with the opposite orientation, and that there is a canonical isomorphism $i \colon L^2(\cala(I)) \cong L^2(\cala(\bar I))$ 
under which the left/right $\cala(I)$-actions on $L^2(\cala(I))$ corresponds to the right/left $\cala(\bar I)$-actions on $L^2(\cala(\bar I))$.
For every $I\subset S$, the vacuum sector $H_0(S,\cala)$ is a right $\cala(\bar I)$-module via the isomorphism $\cala(\mathrm{Id}_I):\cala(\bar I) \cong \cala(I)^\op$.
Composing $i$ and $v_I$, we obtain a right $\cala(\bar I)$ linear isomorphism $H_0(S,\cala) \to L^2(\cala(\bar I))$.
Let $K$ be any left $\cala(\bar I)$-module.
Using the left unit map $L^2(\cala(\bar I)) \boxtimes_{\cala(\bar I)} K \to K$, we obtain a natural isomorphism
\begin{equation}
\label{eq:right-unit-vacuum}
H_0(S,\cala) \boxtimes_{\cala(\bar I)} K \xrightarrow{\cong} K
\end{equation}
similar to~\eqref{eq:left-unit-vacuum}.

Now let $S_1$ and $S_2$ be conformal circles,
let $I_1\subset S_1$ and $I_2\subset S_2$ be intervals, and let $\varphi:I_2\to I_1$ be an orientation-reversing diffeomorphism.
Let us also assume that $\varphi$ is the restriction of some element in $\Conf_-(S_2,S_1)$.
Let $I_i'$ denote the closure of $S_i\setminus I_i$.
Finally, let $S_3=I_1'\cup_{\partial I_2} I_2'$ be the circle obtained by gluing $S_1$ and $S_2$ along $\varphi$, 
and then removing the interior of $I_1$.

The circle $S_3$ is given a conformal structure as follows.
Letting $j_1\in\Conf_-(S_1)$ be the involution that fixes $\partial I_1$, the conformal structure on $S_3$ is the one making $j_1|_{I_1'}\cup \mathrm{Id}_{I_2'}:S_3\to S_2$ into a conformal map.
Note that $\mathrm{Id}_{I_1'}\cup j_2|_{I_2'}:S_3\to S_1$ is then also a conformal map, where $j_2\in \Conf_-(S_2)$ is the involution that fixes $\partial I_2$.\\
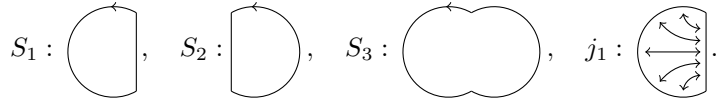
\[
S_1:
\tikzmath[scale=\displscale]{
\draw (60:12) arc (60:300:12);
\draw (60:12) -- (300:12);
\draw[->] (90:12) ++ (-.5,0) -- +(-.1,0);
}\,,\quad %tikzmath
S_2:
\tikzmath[scale=\displscale]{
\draw (12,0) +(120:12) arc (120:-120:12);
\draw (12,0) +(120:12) -- +(-120:12);
\draw[->] (12,0) ++ (90:12) ++ (-.5,0) -- +(-.1,0);
},\quad %tikzmath
S_3:
\tikzmath[scale=\displscale]{
\draw (60:12) arc (60:300:12);
\draw (12,0) +(120:12) arc (120:-120:12);
\draw[->] (90:12) ++ (-.5,0) -- +(-.1,0);
},\quad %tikzmath
j_1:
\tikzmath[scale=\displscale]{
\draw (60:12) arc (60:300:12);
\draw (60:12) -- (300:12);
\draw[<->] (180:10) -- (0:4.5);
\draw[<->] (128:10) to[in=178, out=-50] (4.5,3);
\draw[<->] (90:10.2) to[in=175, out=-80] (4.5,6.2);
\draw[<->] (-128:10) to[in=-178, out=50] (4.5,-3);
\draw[<->] (-90:10.2) to[in=-175, out=80] (4.5,-6.2);
}\,. %tikzmath
\]

The following is a refinement of Corollary \ref{cor: vacuum * vacuum = vacuum}
in the sense that \eqref{eq: HoxHo=Ho} is now replaced by a \emph{canonical} isomorphism:

\begin{corollary}\label{cor: Conformal version of vacuum * vacuum = vacuum}
Let $S_1$, $S_2$ and $S_3$ be as above.
View $H_0(S_1,\cala)$ as a right $\cala(I_2)$-module via $\cala(\varphi^{-1})$.
Then there is a canonical unitary isomoprhism of $S_3$-sectors of $\cala$:
\[
H_0(S_1,\cala) \boxtimes_{\cala(I_2)} H_0(S_2,\cala)\cong H_0(S_3,\cala).
\]
\end{corollary}

\begin{proof}
The isomorphism is given by
\[\begin{split}
H_0(S_1,\cala) \boxtimes_{\cala(I_2)} H_0(S_2,\cala)
\xrightarrow{v_{I_1'}\boxtimes\hspace{.3mm} v_{I_2\phantom{'}}\!}\,\,& L^2(\cala(I_1'))\boxtimes_{\cala(I_2)}L^2(\cala(I_2))\\
\xrightarrow{\cong}\,\, &L^2(\cala(I_1')) \xrightarrow{v_{I_1'}^*}H_0(S_3,\cala).
\end{split}\]
\end{proof}

\subsubsection*{Vacuum representations of the conformal group}
We now discuss a number of results we will need for the proof of Theorem~\ref{thm: Vaccum Sector}.
Most importantly, we construct an action of the group $\Conf(S^1)$ on the vacuum sector of a conformal net.
We begin with two well known facts about conformal transformations:

\begin{lemma}
  \label{lem:properties-of-Conf}
  Let $S$ be a conformal circle and let $\zeta, \zeta' \in S$ be two distinct points.
  \begin{enumerate}
  \rm\item\it \label{lem:properties-of-Conf:unique-reflection}
     There is a unique 
     $j \in \Conf_{-}(S^1)$ that fixes $\zeta$ and $\zeta'$ and
     is an involution.
  \rm\item\it \label{lem:properties-of-Conf:Moebius-flow}
     The subgroup $\Conf_+(S^1,\{\zeta, \zeta'\})$ 
     of $\Conf_+(S^1)$ that fixes both $\zeta$ and $\zeta'$
     is isomorphic to $\IR$. 
     The elements of this subgroup commute with the unique involution
     $j \in \Conf_{-}(S)$ that fixes $\zeta$ and $\zeta'$.
  \end{enumerate}
\end{lemma}
 
\begin{proof}
  We may assume without loss of generality that $S=S^1$, $\zeta = 1$ and $\zeta' = -1$. 
  The stabilizer of the pair $(1,-1)$ is given by
  \[
  \Conf\big(S^1,\{1, -1\}\big) = \left.\left\{\frac{az+b}{bz+a}\,\right|\,a,b\in\IR,\,a^2-b^2=\pm1\right\}.
  \]
  There is an isomorphism from $\IR\times \IZ/2$ to the above group that sends $(t,0)$ to $\frac{\cosh(t)z + \sinh(t)}{\sinh(t)z + \cosh(t)}$ 
  and $(t,1)$ to $\frac{\sinh(t)z + \cosh(t)}{\cosh(t)z + \sinh(t)}=\frac{\cosh(t)\bar z + \sinh(t)}{\sinh(t)\bar z + \cosh(t)}$.
  The result follows since $\IR\times \IZ/2$ has a unique element of order two, and this element is central.
\end{proof}

\begin{lemma} \label{lem:no-central-ext-for-conf}
Let $\PSL_2(\IR)\to \PU(H)$ be a continuous representation for the quotient strong topology.
Then it lifts uniquely to a continuous representation 
$\widetilde{\PSL_2}(\IR)\to \U(H)$ of the universal cover of 
$\PSL_2(\IR)$.
\end{lemma}
\begin{proof}
The composite $\widetilde{\PSL_2}(\IR)\twoheadrightarrow \PSL_2(\IR)\to \PU(H)$ is a continuous projective representation of a semisimple simply-connected Lie group,
and thus lifts to $\U(H)$ by the main theorem of~\cite{Bargmann(unitary-ray-reps)}.
Moreover, this lift is unique: any two lifts differ by a character, but the abelianization of $\widetilde{\PSL_2}(\IR)$ is trivial, and so it has no characters.
\end{proof}

Let $\cala$ be a conformal net and $S$ a conformal circle.
We let $j \in \Conf_-(S)$ be an involution, $I \subset S$ an interval whose boundary is fixed by $j$, and we consider the Hilbert space $H_0 := L^2(\cala(I))$.
By Proposition~\ref{prop:projective-implementation-diffeo}, 
$H_0$ carries a projective $\Diff(S)$ action $\varphi \mapsto [u_\varphi]$ implementing diffeomorphims.
On the subgroup $\Conf(S)$, this action lifts to an honest representation $\varphi \mapsto u_\varphi$:

\begin{proposition}\label{prop: Conf --> U_pm}
Let $S$ be a conformal circle, and let $H_0$ be as above.
Then there exists a unique lift
\begin{equation}\label{eq: Conf lifts}
\begin{matrix}
\begin{tikzpicture}%[>=latex]
\node (conf) at (-1.6,1) {$\Conf_+(S)$}; 
\node (U) at (1.6,1) {$\U(H_0)$} edge [<-,dashed] (conf); 
\node (Diff) at (-1.6,-.2) {$\Diff_+(S)$} edge [to-to reversed] (conf);
\node (PU) at (1.6,-.2) {$\PU(H_0)$} edge [<-] (Diff) edge [<<-] (U); 
\node at (0,1.2) {$\varphi\mapsto u_\varphi$};
\node at (0,0.8) {$\scriptstyle \exists !$};
\end{tikzpicture} 
\end{matrix}
\end{equation}
of the projective action $\Conf_+(S)\to \PU(H_0)$ constructed in Proposition \ref{prop:projective-implementation-diffeo}
to an honest action $\Conf_+(S) \to \U(H_0)$.
There also exists a lift
\begin{equation}\label{eq: Conf lifts*}
\begin{matrix}
\begin{tikzpicture}%[>=latex]
\node (conf) at (-1.6,1) {$\Conf(S)$}; 
\node (U) at (1.6,1) {$\U_\pm(H_0)$} edge [<-,dashed] (conf); 
\node (Diff) at (-1.6,-.2) {$\Diff(S)$} edge [to-to reversed] (conf);
\node (PU) at (1.6,-.2) {$\PU_\pm(H_0)$} edge [<-] (Diff) edge [<<-] (U); 
\node at (0,0.8) {$\scriptstyle \exists$};
\end{tikzpicture} 
\end{matrix}
\end{equation}
of the projective action $\Conf(S)\to \PU_\pm(H_0)$ 
to an honest action $\Conf(S) \to \U_\pm(H_0)$.
The lift \ref{eq: Conf lifts*} is unique up to multiplication by the character $\Conf(S)\twoheadrightarrow \pi_0(\Conf(S))\cong \{\pm1\}$.
\end{proposition}

\begin{proof}
The following proof is based on a trick from \cite[Prop. 1.1]{Guido-Longo(The-conformal-spin-and-statistics-theorem)}.
Let $\widetilde\Conf_+(S)$ be the universal cover of $\Conf_+(S)$.
By Lemma \ref{lem:no-central-ext-for-conf}, the projective action of $\Conf_+(S)$ lifts uniquely to an action $\widetilde\Conf_+(S)\to \U(H_0)$.
Let us denote by $u_\varphi\in\U(H_0)$ the image of an element $\varphi\in\widetilde\Conf_+(S)$.
By construction, $u_\varphi$ implements the image $\bar\varphi\in\Conf_+(S)$ of $\varphi$.
The conjugation action of $j$ on $\Conf_+(S)$ lifts to an action, again denoted $\varphi\mapsto j\varphi j$, on its universal cover.
Note that
\begin{equation}\label{eq:widetildeConf}
\begin{split}
\widetilde\Conf_+(S)\,&\to\,\, \U(H_0)\\
\varphi\,\,\,&\mapsto\, Ju_{j\varphi j} J
\end{split}
\end{equation}
is a group homomorphism. 
By Lemma \ref{lem: j is implemented}, $Ju_{j\varphi j} J$ also implements $\bar\varphi$.
The uniqueness in Lemma \ref{lem:no-central-ext-for-conf} then implies that $u_\varphi=Ju_{j\varphi j} J$.
Equivalently, we have $u_{j\varphi j}=Ju_\varphi J$.

Let $f\in \Conf_-(S)$ be an involution that exchanges the two boundary points of $I$, and let $F:=L^2(\cala(f|_{I}))\circ J$. 
The operator $F$ squares to one, commutes with $J$, and implements $f$ by Lemma \ref{lem:L^2=conformal-implementaion OP}.
Let $r\in\widetilde\Conf_+(S)$ be a lift of `rotation by $\pi/2$', so that $r^4$ generates the kernel of the projection map $\widetilde\Conf_+(S)\twoheadrightarrow \Conf_+(S)$,
and its image $\bar r\in \Conf_+(S)$ satisfies $\bar rj\bar r^{-1}=f$.
\[
I:\,\,\,\tikzmath[scale=.8]{\draw (.7,0) arc (0:180:.7);\draw[white] (0,-.7); \draw[->] (0,.7) -- (-.01,.7);}
\qquad j:\,\,\,\tikzmath[scale=.8]{\draw circle (.7);\draw[dashed]  (-.95,0)--(1,0); \draw[<->] (.9,-.2) -- (.9,.2);}
\qquad f:\,\,\,\tikzmath[scale=.8]{\draw circle (.7);\draw[dashed] (0,-.95)--(0,1); \draw[<->] (-.2,.9) -- (.2,.9);}
\qquad \bar r:\,\,\,\tikzmath[scale=.8]{\draw circle (.7); \fill circle (.03); \draw[->, dashed] (.9,0) arc (0:90:.9);}
\]
Finally, let $R:=u_r$ be the unitary implementing $r$.
By the above computation, the inverse of $R$ is given by
\[
R^{-1}=u_{r^{-1}}=u_{jrj} = J u_r J = JRJ.
\]
This implies $R=JR^{-1}J$.
The involutions $F$ and $RJR^{-1}$ both implement $f$. It follows that $RJR^{-1}=\lambda F$ for some $\lambda=\pm 1$.
Now we compute
\[
R^4=(RJR^{-1}J)^2=(\lambda FJ)^2=\lambda^2F^2J^2=1,
\]
which shows that the action \eqref{eq:widetildeConf} descends to $\Conf_+(S)$.

To extend it to $\Conf(S)$, it is enough to specify where $j$ should go.
The only anti-unitary involutions implementing $j$ are $J$ and $-J$.
Both assignments $j\mapsto J$ and $j\mapsto -J$ produce homomorphisms $\Conf(S)\to \U_\pm(H_0)$.
\end{proof}

\begin{convention}\label{conv: j to J}
As shown above, there are two possible actions of $\Conf(S)$ on $H_0$.
Henceforth, we will always consider the one that sends $j$ to $J$.
\end{convention}

\begin{lemma}\label{lem: L^2(cala(varphi|_{I}))=u_varphi}
Let $\varphi\in\Conf_+(S,\partial I)$ be a conformal map that commutes with $j$,
and let $u_\varphi$ be the unitary constructed in \ref{prop: Conf --> U_pm} and \ref{conv: j to J}.
Then $L^2(\cala(\varphi|_{I}))=u_\varphi$.
\end{lemma}

\begin{proof}
By Lemma~\ref{lem:properties-of-Conf}~
\ref{lem:properties-of-Conf:Moebius-flow} 
the group $\Conf_+(S,\partial I)$ is isomorphic to $\IR_{\ge 0}$, and in particular is commutative.
Consequently, given elements $\varphi$ and $\psi$ of that group, 
$u_\varphi$ and $u_\psi$ commute.
If $v$ is any other unitary implementing $\psi$,
then $v=\theta u_\psi$ for some unit complex number $\theta\in S^1$, and hence $u_\varphi$ and $v$ commute.
In particular, $u_\varphi$ commutes with $L^2(\cala(\psi|_{I}))$.
The map
\begin{equation}\label{eq: homomorphism Conf_+(S,)-->S^1}
 \begin{split}
\Conf_+(S,\partial I) &\to\,\, S^1\\
\varphi\,\,\, &\mapsto\, 
 c_{\varphi}:=u_\varphi^*\circ L^2(\cala(\varphi|_{I}))
 \end{split}
\end{equation}
is therefore a homomorphism.

Let $f\in \Conf_-(S)$ and $F=L^2(\cala(f|_{I}))\circ J$ be as in the proof of Proposition~\ref{prop: Conf --> U_pm}.
The involution implementing $f$ being unique up to sign, we have $F=\pm u_f$.
From the equations
\[
Fu_\varphi F=u_fu_\varphi u_f=u_{f\varphi f}=u_{\varphi^{-1}}
\]
and
\[
 \begin{split}
F\, L^2(\cala(\varphi|_{I}))\, F 
&= L^2(\cala(f|_{I}))\,L^2(\cala(\varphi|_{I}))\,
 L^2(\cala(f|_{I}))\\
&= L^2(\cala((f\varphi f)|_{I})) = L^2(\cala(\varphi^{-1}|_{I})),
 \end{split}
\]
it follows that
\[
c_\varphi = F c_\varphi F = 
 \big(Fu_\varphi^*F\big)\circ\big(FL^2(\cala(\varphi|_{I}))F\big) = 
 u_{\varphi^{-1}}^*\circ L^2(\cala(\varphi^{-1}|_{I})) = c_{\varphi^{-1}}.
\]
The homomorphism~\eqref{eq: homomorphism Conf_+(S,)-->S^1} therefore takes its values in $\{\pm1\}$.
Since $\Conf_+(S,\partial I)$ is connected, it must be trivial.
\end{proof}

Given a conformal circle $S$ and an interval $I\subset S$,
we write $L^2(\cala(I))$ for the vacuum sector associated to $S$, $I$, and the involution $j\in\Conf_-(S)$ that fixes $\partial I$.

\begin{lemma}
\label{lem:T(varphi,I)}
For $\varphi \in \Conf_+(S)$, and $I \subset S$ an interval, consider the map
\begin{equation}\label{eq: def of T(phi,I)}
T(\varphi,I) \,:=\,u_\varphi^* \circ L^2(\cala(\varphi|_{I})) \colon L^2(\cala(I)) \to L^2(\cala(\varphi(I))).
\end{equation}
Then
\begin{enumerate}
\item \label{lem:T(varphi,I):intertwines}
$T(\varphi,I)$ is a morphism of $S$-sectors.
\item \label{lem:T(varphi,I):independent}
$T(\varphi,I)$ depends only on $I$ and on $\varphi(I)$, but not on $\varphi$:
if $\varphi$ and $\psi \in \Conf_+(S)$ are such that $\varphi(I) = \psi(I)$, then $T(\varphi,I) = T(\psi,I)$.
\item \label{lem:T(varphi,I):functorial}
If $\varphi$, $\psi \in \Conf_+(S)$, then $T(\psi \circ \varphi,I) = T(\psi,\varphi(I)) \circ T(\varphi, I)$. 
\end{enumerate}
\end{lemma}

\begin{proof}
{\it \ref{lem:T(varphi,I):intertwines}}
By naturality of the vacuum sector construction, 
the map $L^2(\cala(\varphi|_I))$ is a morphism of sectors from $L^2(\cala(I))$ to $\varphi^*L^2(\cala(\varphi(I)))$.
The map $u_\varphi:L^2(\cala(\varphi(I))) \to \varphi^*L^2(\cala(\varphi(I)))$ is also a morphism of sectors.
The composite $u_\varphi^*\, L^2(\cala(\varphi|_{I}))$ is therefore also a morphism of sectors.\\
{\it \ref{lem:T(varphi,I):independent}} 
Let $K:=\varphi(I) = \psi(I)$. By applying Lemma~\ref{lem: L^2(cala(varphi|_{I}))=u_varphi}
to $\varphi \circ \psi^{-1}$, we get
\[\begin{split}
T(\varphi,I) & =u_{\varphi}^*\, L^2(\cala({\varphi}|_{I})) \\
& =u^*_\psi\, u_{{\varphi} \psi^{-1}}^*\, L^2\big(\cala({\varphi} \psi^{-1}|_K)\big)\, L^2\big(\cala(\psi|_I)\big) \\
& = u^*_\psi\, L^2(\cala(\psi|_I))=T(\psi,I). 
\end{split}\]
{\it \ref{lem:T(varphi,I):functorial}}
The operator $T(\psi,\varphi(I))$ is a isomorphism of $S$-sectors.
Since the $S$-sector structure on a vacuum sector uniquely determines the action of $\Conf_+(S)$ on the sector, 
the operator $T(\psi,\varphi(I)):L^2(\cala(\varphi(I)))\to L^2(\cala(\psi\circ\varphi(I)))$ intertwines the two actions of $\Conf_+(S)$.
Therefore
\[\begin{split}
T(\psi,\varphi(I)) \, T(\varphi,I) 
& =T(\psi,\varphi(I)) \, u_{\varphi}^* \, L^2(\cala(\varphi|_{I})) \\
& =u_{\varphi}^* \, T(\psi,\varphi(I))\, L^2(\cala(\varphi|_{I})) \\
& =u_{\varphi}^* \, u_{\psi}^* \, 
 L^2(\cala(\psi|_{\varphi(I)})) \, L^2(\cala(\varphi|_{I})) \hspace{1.9cm}\\
& =u_{\psi \, \varphi}^* \, 
 L^2(\cala(\psi\varphi|_{I})) 
= T(\psi \circ \varphi,I). 
\end{split}
\]
\end{proof}

\subsubsection*{Construction of the vacuum sector}
After all the above preparation, we can finally construct the vacuum sector associated to a conformal circle:

\begin{proof}[Proof of Theorem \ref{thm: Vaccum Sector}]
Given a conformal circle $S$, let $\cali$ be the category whose objects are the intervals of $S$, and in which every hom-set $\mathrm{Hom}(I,J)$ contains exactly one element.
Recall that for an interval $I\subset S$, we write $L^2(\cala(I))$ for the vacuum sector associated to $S$, $I$, and the involution $j\in\Conf_-(S)$ that fixes $\partial I$.

The assignment $I\mapsto L^2(\cala(I))$ extends to a functor from $\cali$ to the category of $S$-sectors of $\cala$ in the following way:
given two intervals $I,J\in\cali$, pick a map $\varphi\in \Conf_+(S)$ that sends $I$ to $J$.
The value of the functor on the unique morphism from $I$ to $J$ is then given by $u_\varphi^* \circ L^2(\cala(\varphi|_{I})) \colon L^2(\cala(I)) \to L^2(\cala(\varphi(I)))$. 
By Lemma \ref{lem:T(varphi,I)}, this assignment is well defined, independent of the choice of $\varphi$, and functorial.
We can therefore define
\[
H_0(S,\cala):=\lim_{I\in \cali}\, L^2(\cala(I)).
\]
We can therefore think of a vector in $H_0(S,\cala)$ as a collection of vectors ${\big\{\xi_I\! \in\! L^2(\cala(I))\big\}_{I\in\cali}}$ subject to the condition that 
\begin{equation}\label{eq: u_phi^* o L2(A(phi|_I)))(xi_I)=xi_phi(I)}
u_\varphi^* \, L^2(\cala(\varphi|_{I}))\,(\xi_I)=\xi_{\varphi(I)}
\end{equation}
for every $I\subset S$, and every $\varphi\in \Conf_+(S)$.
Let
\begin{equation}\label{eq: construction of v_I}
\begin{split}
v_I \,\colon\, H_0(S,\cala)\, &\to L^2(\cala(I))\\
\{\xi_{J}\}_{J\in\cali}\,&\mapsto \,\,\xi_I
\end{split}
\end{equation}
be the maps that exhibit $H_0(S,\cala)$ as the limit of all the $L^2(\cala(I))$, and note that they are all isomorphisms as $\cali$ is equivalent to the trivial category (with only one object and one identity arrow).
We equip $H_0(S,\cala)$ with the Hilbert space structure that makes the maps $v_I$ unitary.

So far, we have only constructed $S \mapsto H_0(S,\cala)$ as a functor from conformal circles and orientation-preserving maps.
Explicitly, for $\{\xi_I\}_{I\subset S_1}\in H_0(S_1,\cala)$ and $\varphi:S_1\to S_2$ an orientation-preserving map, 
we have
\[
H_0(\varphi,\cala)\big(\{\xi_I\}\big)=\big\{L^2(\cala(\varphi|_{\varphi^{-1}(K)}))(\xi_{\varphi^{-1}(K)})\big\}\in H_0(S_2,\cala)
\]
for $K\subset S_2$. Equivalently, the map $H_0(\varphi,\cala)$ is characterized by the fact that
\begin{equation}\label{eq:H_0 reversing -- SIB}
H_0(\varphi,\cala)= v_{\varphi(I)}^*\circ L^2(\cala(\varphi|_{I})) \circ v_{I}
\end{equation}
for every $I\subset S_1$.

If now $\varphi\in\Conf_-(S_1,S_2)$ is an orientation-reversing map, define 
\begin{equation}\label{eq:H_0 reversing}
H_0(\varphi,\cala):= v_{\varphi j(I)}^*\circ L^2(\cala(\varphi j|_{I})) \circ J \circ v_{I},
\end{equation}
where $I\subset S_1$ is an interval, and $j\in \Conf_-(S_1)$ is the involution that fixes $\partial I$.
To see that \eqref{eq:H_0 reversing} is independent of the choice of interval $I$,
given another interval $\tilde I\subset S_1$, pick $\psi\in\Conf_+(S_1)$ such that $\psi(I)=\tilde I$, and
let $\tilde j=\psi j\psi^{-1}$ be the involution that fixes $\partial \tilde I$.
Letting $A_1=\cala(I)$, $\tilde A_1=\cala(\tilde I)$, $A_2=\cala(\varphi j(I))$, and $\tilde A_2=\cala(\varphi \tilde j(\tilde I))$, 
we have to show that the following diagram is commutative
\[
\,\,\,\begin{matrix}
\begin{tikzpicture}[scale = 0.9]%[>=latex]
\node (H0_1) at (-2.2,0) {$H_0(S_1,\cala)$};
\node (H0_1') at (8,0) {$H_0(S_2,\cala)$};
\node (L2A1) at (0.2,1.5) {$L^2(A_1)$};
\node (L2A1') at (2.8,1.5) {$L^2(A_1)$};
\node (L2A2) at (5.6,1.5) {$L^2(A_2)$};
\node (L2A1a) at (0.2,0) {$L^2(\tilde A_1)$};
\node (L2A1b) at (2.8,0) {$L^2(\tilde A_1)$};
\node (L2A1d) at (0.2,-1.5) {$L^2(\tilde A_1)$};
\node (L2A1e) at (2.8,-1.5) {$L^2(\tilde A_1)$};
\node (L2A2a) at (5.6,0) {$L^2(\tilde A_2)$};
\node (L2A2b) at (5.6,-1.5) {$L^2(\tilde A_2)$};
\draw (H0_1) to[->,out=85,in=180] node[above, xshift=-4] {$\scriptstyle v_{I}$} (L2A1);
\draw (H0_1) to[->,out=-85,in=180] node[below, xshift=-4] {$\scriptstyle v_{\tilde I}$} (L2A1d);
\draw (L2A1) to[->] node[above] {$\scriptstyle J$} (L2A1');
\draw (L2A1') to[->] node[above] {$\scriptstyle L^2(\varphi j)$} (L2A2);
\draw (L2A1a) to[->] node[above] {$\scriptstyle J$} (L2A1b);
\draw (L2A1b) to[->] node[above] {$\scriptstyle L^2(\varphi\tilde j)$} (L2A2a);
\draw (L2A1e) to[->] node[above] {$\scriptstyle L^2(\varphi\tilde j)$} (L2A2b);
\draw (L2A1d) to[->] node[above] {$\scriptstyle J$} (L2A1e);
\draw (L2A2) to[->,out=0,in=95] node[above, xshift=4] {$\scriptstyle v_{\varphi j(I)}^*$} (H0_1');
\draw (L2A2b) to[->,out=0,in=-95] node[below, xshift=4] {$\scriptstyle v_{\varphi \tilde j(\tilde I)}^*$} (H0_1');
\draw (L2A1) to[->] node[left] {$\scriptstyle L^2(\psi)$} (L2A1a);
\draw (L2A1a) to[->] node[left] {$\scriptstyle u_\psi^*$} (L2A1d);
\draw (L2A1') to[->] node[left] {$\scriptstyle L^2(\psi)$} (L2A1b);
\draw (L2A1b) to[->] node[left] {$\scriptstyle u_{\tilde j\psi\tilde j}^*$} (L2A1e);
\draw (L2A2) to[->] node[fill = white, inner ysep = 0, xshift=7] {$\scriptstyle L^2(\varphi\psi\varphi^{-1})$} (L2A2a);
\draw (L2A2a) to[->] node[right] {$\scriptstyle u_{\varphi\psi\varphi^{-1}}^*$} (L2A2b);
\end{tikzpicture} 
\end{matrix},
\]
where $L^2(\psi)$ stands for $L^2(\cala(\psi|_{I}))$, and we have used similar abbreviations for $L^2(\cala(\varphi j|_{I}))$, 
$L^2(\cala(\varphi \tilde j|_{\tilde I}))$, and $L^2(\cala(\varphi\psi\varphi^{-1}|_{\varphi j(I)}))$.
The commutativity of the left and right parts of the diagram are given by \eqref{eq: u_phi^* o L2(A(phi|_I)))(xi_I)=xi_phi(I)}.
The lower left square is commutative since $J=u_{\tilde j}$ on $L^2(\tilde A_1)$.
The commutativity of the remaining three squares is clear.

Given $\varphi\in\Conf(S_1,S_2)$ and $\psi\in\Conf(S_2,S_3)$, we still need to show that $H_0(\psi,\cala)\circ H_0(\varphi,\cala) = H_0(\psi\circ \varphi,\cala)$.
The case when both $\psi$ and $\varphi$ are in $\Conf_+$ and the case when $\psi\in\Conf_+$ and $\varphi\in\Conf_-$ are clear.  The case $\psi\in\Conf_-$ and $\varphi\in\Conf_+$ is checked as follows:
\[
\begin{split}
H_0(\psi,\cala)\circ H_0(\varphi,\cala)&= \big(v_{\psi\varphi j(I)}^*\, L^2(\cala(\psi \varphi j \varphi^{-1}|_{\varphi(I)})) \, J \, v_{\varphi(I)}\big)\,
\big(v_{\varphi(I)}^*\, L^2(\cala(\varphi|_{I})) \, v_{I}\big)\\
&=v_{\psi\varphi j(I)}^*\, L^2(\cala(\psi \varphi j \varphi^{-1}|_{\varphi(I)})) \, J \, L^2(\cala(\varphi|_{I})) \, v_{I}\\
&=v_{\psi\varphi j(I)}^*\, L^2(\cala(\psi \varphi j \varphi^{-1}|_{\varphi(I)})) \, L^2(\cala(\varphi|_{I}))\, J \, v_{I}\\
&=v_{\psi\varphi j(I)}^*\, L^2(\cala(\psi\varphi j|_{I})) \, J \, v_{I}\\
&=H_0(\psi\circ \varphi,\cala).
\end{split}
\] 
The last case $\varphi, \psi \in\Conf_-$ follows from the previous ones since
$H_0(\varphi,\cala) H_0(\psi,\cala) = 
H_0(\varphi j,\cala) H_0(j,\cala) H_0(j,\cala) H_0(j \psi,\cala) = 
H_0(\varphi j,\cala) H_0(j \psi,\cala) = 
H_0(\varphi\psi,\cala)$.
Here, $j\in\Conf_-$ is an arbitrary involution, and we have used {\it \ref{prop:vacuum-sector:J}} below to know that $H_0(j,\cala)$ is an involution.

Finally, we summarize the structure on $H_0(S,\cala)$:\\
{\it (i)} The vector space $H_0(S,\cala)$ is an $S$-sector by construction.\\
{\it (ii)} By Proposition \ref{prop:projective-implementation-diffeo}, $H_0(S,\cala)$ comes equipped with a projective $\Diff(S)$ action $\varphi\mapsto [u_\varphi]$,
uniquely determined by the requirement that $u_\varphi$ implements $\varphi$.
For $\varphi\in\Conf(S)$, the map $H_0(\varphi,\cala)$ also implements $\varphi$, and so $[H_0(\varphi,\cala)]=[u_\varphi]$.\\
{\it (iii)} The maps $v_I \colon H_0(S,\cala) \to L^2(\cala(I))$ are defined in \eqref{eq: construction of v_I};
diagram \eqref{eq: naturality of v_I} is equation \eqref{eq:H_0 reversing -- SIB}, and diagram \eqref{eq: naturality of v_I -- BIS} is equation \eqref{eq:H_0 reversing}.\\
{\it (iv)}
Letting $\varphi=j$ in \eqref{eq:H_0 reversing}, we get $H_0(j,\cala)= v_{I}^*\, J\, v_{I}$, which is equivalent to $J=v_I\,H_0(j,\cala)\,v_I^*$.
\end{proof}

%==================================================================

\section{Finite conformal nets and their sectors}\label{sec:The finiteness condition}

\subsection{The index of a conformal net}\label{subs: The index of a conformal net}

We have defined a conformal net as functor $\cala:\INT\to \VN$ from the category of intervals to that of von Neumann algebras.
In the second paper of this series \cite{BDH(all-together)}, we will see that a conformal net can also be used to assign von Neumann 
algebras to arbitrary compact 1-manifolds, i.e., disjoint unions of intervals and circles.
For now we focus on disjoint unions of intervals, extending $\cala$ by setting $\cala(I_1\cup\ldots\cup I_n):=\cala(I_1)\,\bar\ox\,\ldots\,\bar\ox\,\cala(I_n)$.

\begin{definition}\label{def:index-for-nets}
  Let $S$ be a circle, split into four intervals $I_1$, $I_2$, $I_3$, $I_4$ such that each $I_i$ intersects $I_{i+1}$ (cyclic numbering) in a single point:
  \begin{equation}\label{picture: circle with four intervals}
  \tikzmath[scale=.07]
  { \useasboundingbox (-20,-22) rectangle (20,22);
        \draw (0,0) circle (15);
        \foreach \thet in {45,135, ..., 315} { \draw[thick] (\thet:14) -- (\thet:16);}
        \foreach \ii in {1, ..., 4} {  \node at (\ii*90-90:19) {$I_\ii$}; }}
  \end{equation} 
  Let $\cala$ be a conformal net (always assumed irreducible).
  The algebras
  $\cala(I_1\cup I_3)= \cala(I_1)\,\bar\otimes\,\cala(I_3)$ and
  $\cala(I_2\cup I_4)= \cala(I_2)\,\bar\otimes\,\cala(I_4)$ act on
  $H_0(S,\cala)$ and commute with each other.
  The \emph{index} $\mu(A)$ of the conformal net $\cala$ is the minimal index of the inclusion $\cala(I_1\cup I_3)\subseteq \cala(I_2\cup I_4)'$:
  \begin{equation*}
  \mu(\cala) := [\cala(I_2\cup I_4)':\cala(I_1\cup I_3)],
  \end{equation*}
  where the commutant is computed in $\bfB(H_0(S,\cala))$.  (See Definition~\ref{def: stat dim} and the paper~\cite{BDH(Dualizability+Index-of-subfactors)} for recollections about the notion of the minimal index of a subfactor.)
\end{definition}
Note  
that the index is an invariant of the net (that is, it does not depend on the choice of circle and intervals), and that it satisfies $\mu(\cala \otimes \calb) = \mu(\cala) \cdot \mu(\calb)$.

\begin{remark}
In~\cite{Kawahigashi-Longo-Mueger(2001multi-interval)}, 
nets of finite index were called \emph{completely rational}.
In our context however, that terminology would be quite misleading.
Consider any unitary conformal field theory (not necessarily rational, e.g., the free boson compactified on a circle of \emph{irrational} squared radius)
viewed as a net $\calo\mapsto \cala(\calo)$ on two-dimensional Minkowski space.  We believe that the restriction of such a net to the zero time slice provides an example of a conformal net with index $1$.
\end{remark}

Recall that a bimodule ${}_AH_B$ is dualizable (Definition~\ref{def:dual}) if and only if the inclusion $A\subset B'$ has finite index~\cite{BDH(Dualizability+Index-of-subfactors)}. 
A conformal net $\cala$ therefore has finite index if and only if the bimodule ${}_{\cala(I_1\cup I_3)}H_0(S)_{\cala(I_2\cup I_4)^\op}$ has a dual.
It is useful to identify this dual:

\begin{lemma}\label{lem: H_0(-S)}
Let $S$ be a circle split into intervals $I_1$, $I_2$, $I_3$, $I_4$ as above, and
let $\bar S$, $\bar I_1,\ldots,\bar I_4$ be the same manifolds with the reverse orientation.
Let $\cala$ be conformal net with finite index.
The dual of the bimodule ${}_{\cala(I_1\cup I_3)}H_0(S)_{\cala(I_2\cup I_4)^\op}$
is
\[
{}_{\cala(\bar I_2\cup \bar I_4)}H_0(\bar S)_{\cala(\bar I_1\cup \bar I_3)^\op}
\]
using the canonical identifications
$\cala(\bar I_1\cup \bar I_3)^\op\cong \cala(I_1\cup I_3)$ and
$\cala(\bar I_2\cup \bar I_4)^\op\cong \cala(I_2\cup I_4)$.
\end{lemma}

\begin{proof}
Let $j\in \Diff_-(S)$ be a diffeomorphism that exchanges $I_1$ with $I_4$, and $I_2$ with $I_3$, and let us abbreviate $\cala(j)$ by $j_*$.
The $S$-sector $H_0(S,\cala)$ can be taken to be $L^2(\cala(I_1\cup I_2))$, with actions
\[
\begin{split}
a\xi= a\xi\quad\qquad&\text{for } a \in \cala(I_1\cup I_2)\\
b\xi= \xi j_*(b)\qquad&\text{for } b \in \cala(I_3\cup I_4)
\end{split}
\]
for $\xi\in L^2(\cala(I_1\cup I_2))$.
The actions of the bimodule ${}_{\cala(I_1\cup I_3)}H_0(S)_{\cala(I_2\cup I_4)^\op}$ are then given by
\[
(a_1\otimes a_3)\cdot\xi\cdot (a_2\otimes a_4)=(a_1a_2)\,\xi\, j_*(a_3a_4)
\]
for $a_1 \in \cala(I_1), a_2 \in \cala(I_2)^\op, a_3 \in \cala(I_3)$, and $a_4 \in \cala(I_4)^\op$.
The dual bimodule is the complex conjugate $\overline{L^2(\cala(I_1\cup I_2))}$, with actions
\begin{equation}\label{eq:aaxaa1}
\begin{split}
(a_2\otimes a_4)\cdot\bar\xi\cdot (a_1\otimes a_3)&=\overline{(a_1\otimes a_3)^*\!\cdot\xi\cdot (a_2\otimes a_4)^*}\\
&=\overline{(a_1^*a_2^*)\,\xi\, j_*(a_3^*a_4^*)}
\end{split}
\end{equation}
for $a_1 \in \cala(I_1), a_2 \in \cala(I_2)^\op, a_3 \in \cala(I_3)$, and $a_4 \in \cala(I_4)^\op$.  Here $\bar\xi$ denotes the vector $\xi \in L^2$ viewed as an element of $\overline{L^2}$.

Note that $H_0(\bar S,\cala)=L^2(\cala(\overline{I_1\cup I_2}))=L^2(\cala(I_1\cup I_2)^\op)$ has actions given by
\[
\begin{split}
a\eta= a\eta\quad\qquad&\text{for } a \in \cala(I_1\cup I_2)^\op\\
b\eta= \eta j_*(b)\qquad&\text{for } b \in \cala(I_3\cup I_4)^\op
\end{split}
\]
for $\eta\in L^2(\cala(I_1\cup I_2)^\op)$.
Using the canonical identification between $L^2(\cala(I_1\cup I_2)^\op)$ and $L^2(\cala(I_1\cup I_2))$ that exchanges the left $\cala(I_1\cup I_2)^\op$-module structure with the right $\cala(I_1\cup I_2)$-module structure
and the right $\cala(I_1\cup I_2)^\op$-module structure with the left $\cala(I_1\cup I_2)$-module structure, this becomes
\[
\begin{split}
a\xi= \xi a\quad\qquad&\text{for } a \in \cala(I_1\cup I_2)^\op\\
b\xi= j_*(b)\xi\qquad&\text{for } b \in \cala(I_3\cup I_4)^\op
\end{split}
\]
for $\xi\in L^2(\cala(I_1\cup I_2))$.
The bimodule ${}_{\cala(\bar I_2\cup \bar I_4)}H_0(\bar S)_{\cala(\bar I_1\cup \bar I_3)^\op}$ is therefore given by
\begin{equation}\label{eq:aaxaa2}
(a_2\otimes a_4)\cdot\xi\cdot (a_1\otimes a_3)= j_*(a_3a_4)\,\xi\,(a_1a_2),
\end{equation}
$a_1 \in \cala(I_1), a_2 \in \cala(I_2)^\op, a_3 \in \cala(I_3)$, $a_4 \in \cala(I_4)^\op$.
The isomoprhism that intertwines \eqref{eq:aaxaa1} and \eqref{eq:aaxaa2}
is given by the modular conjugation $J:\overline{L^2(\cala(I_1\cup I_2))}\to L^2(\cala(I_1\cup I_2))$.
\end{proof}

Instead of splitting the circle in four as in \eqref{picture: circle with four intervals}, one can also split it into $2n$ intervals, for arbitrary $n\ge 2$.
\begin{lemma}[\cite{Kawahigashi-Longo-Mueger(2001multi-interval)}]
Let $S$ be a circle, split into $2n$ intervals $I_1, I_2,\ldots,I_{2n}$ such that each $I_i$ intersects $I_{i+1}$ (cyclic numbering) in a single point:
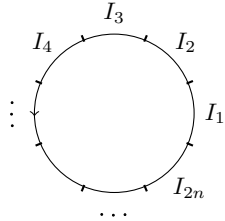
\[
\tikzmath[scale=.07]{\useasboundingbox (-20,-20) rectangle (20,20);\draw (0,0) circle (15);\foreach \thet in {22.5,67.5, ..., 360} { \draw[thick] (\thet:14.7) -- (\thet:16);}
\foreach \ii in {1,2,3,4} {\node[scale=.9] at (\ii*45-45:19) {$I_\ii$}; }\node[scale=.9] at (-45:20) {$I_{2n}$}; 
\foreach \th in {174,180.5,...,190,276,269.5,...,260}{\filldraw (\th:19.3) circle (.165);}\draw [->] (182:15) -- (182.01:15);}
\]
Let $\cala$ be a conformal net with finite index, and let $\nu$ be the square root of the index of $\cala$.
Then the bimodule
\begin{equation}\label{eq: 2n-interval bimodule}
{}_{\cala(I_1\cup I_3\cup\ldots\cup I_{2n-1})}H_0(S)_{\cala(I_2\cup I_4\cup\ldots\cup I_{2n})^\op}
\end{equation}
is dualizable,
and its statistical dimension is given by $\nu^{n-1}$.
\end{lemma}

\begin{proof}
Glue intervals $J_1,\ldots,J_n$ to the circle $S$ as in the following picture
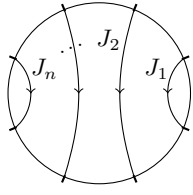
\[
\tikzmath[scale=.08]{\useasboundingbox (-20,-15) rectangle (20,15);\draw (0,0) circle (15);
\draw (22.5:15) to[bend right = 51]node[scale=.9,pos=.17,left, xshift=1]{$J_1$} (-22.5:15);
\draw (67.5:15) to[bend right = 17]node[scale=.9,pos=.17,left, xshift=1.3]{$J_2$} (-67.5:15);
\draw (112.5:15) to[bend left = 17]node[scale=.8,pos=.21, rotate=22,xshift=-1.5]{$.\,\raisebox{.5pt}{.}\,.$} (-112.5:15);
\draw (157.5:15) to[bend left = 51]node[scale=.9,pos=.17,right, xshift=-.5]{$J_n$} (-157.5:15);
\foreach \thet in {22.5,67.5, ..., 360} { \draw[thick] (\thet:14.7) -- (\thet:16);}
\foreach \xx in {11.25,3.38,-3.38,-11.25} {\draw [->] (\xx,0) -- (\xx, -.01);}}%tikzmath
\]
and let $S_1:=I_1\cup J_1$, $S_i:= I_i\cup J_i \cup I_{2n+2-i} \cup \bar J_{i-1}$ for $2\le i \le n$, $S_{n+1} := I_{n+1}\cup \bar J_n$:
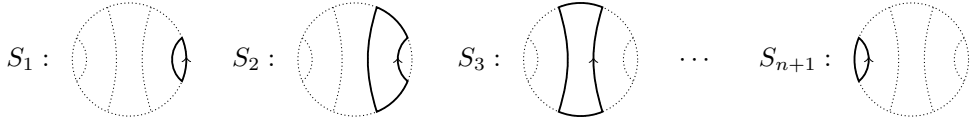
\[
S_1: \tikzmath[scale=.05]{\useasboundingbox (-19,-15) rectangle (20,15);\draw[densely dotted] (0,0) circle (15);
\draw[densely dotted] (22.5:15) to[bend right = 51] (-22.5:15);\draw[densely dotted] (67.5:15) to[bend right = 17] (-67.5:15);
\draw[densely dotted] (112.5:15) to[bend left = 17] (-112.5:15);\draw[densely dotted] (157.5:15) to[bend left = 51] (-157.5:15);
\draw[line width=.7] (22.5:15) to[bend right = 51] (-22.5:15) arc (-22.5:22.5:15) -- cycle;\draw [->] (15,.6) -- (15, .61);}%tikzmath
\quad S_2: \tikzmath[scale=.05] {\useasboundingbox (-19,-15) rectangle (20,15);\draw[densely dotted] (0,0) circle (15);
\draw[densely dotted] (22.5:15) to[bend right = 51] (-22.5:15);\draw[densely dotted] (67.5:15) to[bend right = 17] (-67.5:15);
\draw[densely dotted] (112.5:15) to[bend left = 17] (-112.5:15);\draw[densely dotted] (157.5:15) to[bend left = 51] (-157.5:15);
\draw[line width=.7] (22.5:15) to[bend right = 51] (-22.5:15) arc (-22.5:-67.5:15) to[bend left = 17] (67.5:15) arc (67.5:22.5:15) -- cycle;\draw [->] (11.25,.4) -- (11.25, .41);}%tikzmath
\quad S_3: \tikzmath[scale=.05]{\useasboundingbox (-19,-15) rectangle (20,15);\draw[densely dotted] (0,0) circle (15);
\draw[densely dotted] (22.5:15) to[bend right = 51] (-22.5:15);\draw[densely dotted] (67.5:15) to[bend right = 17] (-67.5:15);
\draw[densely dotted] (112.5:15) to[bend left = 17] (-112.5:15);\draw[densely dotted] (157.5:15) to[bend left = 51] (-157.5:15);
\draw[line width=.7] (67.5:15) to[bend right = 17] (-67.5:15) arc (-67.5:-112.5:15) to[bend right = 17] (112.5:15) arc (112.5:67.5:15) -- cycle;\draw [->] (3.38,.4) -- (3.38, .41);}%tikzmath
\,\,\,\,\cdots\quad\,\,\,\,S_{n+1}:\tikzmath[scale=.05]{\useasboundingbox (-19,-15) rectangle (20,15);\draw[densely dotted] (0,0) circle (15);
\draw[densely dotted] (22.5:15) to[bend right = 51] (-22.5:15);\draw[densely dotted] (67.5:15) to[bend right = 17] (-67.5:15);
\draw[densely dotted] (112.5:15) to[bend left = 17] (-112.5:15);\draw[densely dotted] (157.5:15) to[bend left = 51] (-157.5:15);
\draw[line width=.7] (-157.5:15) to[bend right = 51] (157.5:15) arc (157.5:202.5:15) -- cycle;\draw [->] (-11.25,.6) -- (-11.25, .61);}
\]
The bimodule \eqref{eq: 2n-interval bimodule} can then be factored as
\[
\bigg(\!H_0(S_1)\otimes H_0(S_3) \otimes\ldots\bigg) \underset{\cala(\bar J_1)\,\bar\otimes\,\cala(J_2)\,\bar\otimes\,\cala(\bar J_3)\,\bar\otimes\,\cala(J_4)\,\bar\otimes\,\ldots\!}\boxtimes
\bigg(\!H_0(S_2)\otimes H_0(S_4) \otimes\ldots\bigg).
\]
Its statistical dimension is therefore (see~\cite[Prop 5.2]{BDH(Dualizability+Index-of-subfactors)}) the product of 
\[
\begin{split}
\quad\qquad&\dim\bigg(
{}_{\cala(I_1\cup I_3\cup\ldots\cup I_{2n-1})}\big(H_0(S_1)\otimes H_0(S_3) \otimes\ldots\big)_{\cala(\bar J_1\cup J_2\cup \bar J_3\cup J_4\ldots)}\bigg)\\
=\,\,&\dim\Big( {}_{\cala(I_1)}H_0(S_1)_{\cala(\bar J_1)}\Big)\cdot
\dim\Big( {}_{\cala(I_3\cup I_{2n-1})}H_0(S_3)_{\cala(J_2\cup\bar J_3)}\Big)\\
&\,\phantom{
\dim\Big( {}_{\cala(I_1)}H_0(S_1)_{\cala(\bar J_1)}\Big)} %end phantom
\cdot\dim\Big( {}_{\cala(I_5\cup I_{2n-3})}H_0(S_5)_{\cala(J_4\cup\bar J_5)}\Big)\cdots\\
=\,\,& 1 \cdot \nu \cdot \nu\cdot\ldots \,=\, \nu^{\lfloor\frac{n-1}2\rfloor}
\intertext{and}
&\dim\bigg(
{}_{\cala(\bar J_1\cup J_2\cup \bar J_3\cup J_4\ldots)}\big(H_0(S_2)\otimes H_0(S_4) \otimes\ldots\big)_{\cala(I_2\cup I_4\cup\ldots\cup I_{2n})^\op}\bigg)\\
=\,\,&
\dim\Big( {}_{\cala(\bar J_1\cup J_2)}H_0(S_2)_{\cala(I_2\cup I_{2n})^\op}\Big)\cdot
\dim\Big( {}_{\cala(\bar J_3\cup J_4)}H_0(S_4)_{\cala(I_4\cup I_{2n-2})^\op}\Big)\cdots\\
=\,\,& \nu \cdot \nu\cdot\ldots \,=\, \nu^{\lceil\frac{n-1}2\rceil},
\end{split}
\]
namely $\nu^{\lfloor\frac{n-1}2\rfloor}\cdot \nu^{\lceil\frac{n-1}2\rceil}=\nu^{n-1}$.
\end{proof}

\subsection{Finiteness properties of the category of sectors}\label{sec: Hilbert space for annulus}
The goal of this section is to prove Theorem \ref{thm: KLM -- all irreducible sectors are finite}, which says that the category of sectors of a conformal net with finite index is a fusion category (Definition \ref{def: fusion category}).

Let $\cala$ be an (irreducible) conformal net with finite index.
Hitherto, we have been talking about the vacuum sector $H_0(S)=H_0(S,\cala)$ as being associated to a circle~$S$.
However, it is sometimes convenient to think of it as being associated to a disk $\ID$ with $\partial \ID = S$.
More generally, given an oriented topological surface $\Sigma$ whose boundary $\partial \Sigma$ is equipped with a smooth structure, we will associate to it a Hilbert space $V(\Sigma)$,
well defined up to canonical-up-to-phase isomorphism.
The construction of $V(\Sigma)$ is rather involved, and we will only sketch it in Section \ref{sec: The Hilbert space associated to a surface}.
A thorough discussion, will appear in the second paper of this series \cite{BDH(all-together)}.
In this subsection, we will describe the case when $\Sigma$ is an annulus.

\subsubsection*{The sector of an annulus}
Let $S_l$ be a circle, decomposed into four intervals $I_1, \ldots, I_4$, as in \eqref{picture: circle with four intervals}, and let $S_r$ be another circle, similarly decomposed into four intervals $I_5, \ldots, I_8$.
Let $\varphi:I_5\to I_1$ and $\psi:I_7\to I_3$ be orientation-reversing diffeomorphisms.  These diffeomorphisms equip $H_0(S_l)$ with the structure of a right $\cala(I_5)\,\bar\otimes\,\cala(I_7)$-module.
We are interested in the Hilbert space
\[
H_{\mathit{ann}}\,\,:=\,\,\,\,H_0(S_l)\!\underset{\cala(I_5)\bar\otimes\cala(I_7)}\boxtimes\! H_0(S_r)
\]
This is the space $V(\Sigma)$ associated to the annulus $\Sigma=\ID_l\cup_{I_5\cup I_7} \ID_r$, where $\ID_l$ and $\ID_r$ are disks bounding $S_l$ and $S_r$.
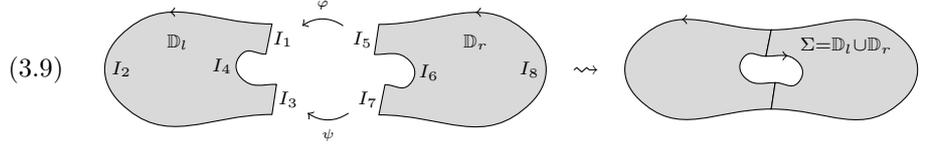
\begin{equation}\label{eq: lem: H_0 otimes H_0   <  H_Ann}
\tikzmath[scale=.4]{
\coordinate (a) at (0,0);\coordinate (b) at (.15,1);\coordinate (c) at (-.2,2);\coordinate (d) at (0,3);\coordinate (e) at (-5,.4);\coordinate (f) at (-5,2.6);\coordinate (g) at (-1,1.2);\coordinate (h) at (-1,2);
\begin{scope}[xshift = 100, yshift = 85, rotate= 180]\coordinate(a') at (0,0);\coordinate(b') at (.15,1);\coordinate(c') at (-.2,2);\coordinate(d') at (0,3);\coordinate (e') at (-5,.4);\coordinate (f') at (-5,2.6);\coordinate (g') at (-1,1.2);\coordinate (h') at (-1,2);\end{scope}
\filldraw[fill = gray!30] (b) to node (I3) [right, xshift = -2] {$\scriptstyle I_3$} (a) to [out = 180, in = -45, looseness=1.1] 
(e) to [out = -45 + 180, in = 225, looseness=1.1] node[right, xshift = -1] {$\scriptstyle I_2$} (f) to [out = 225 + 180, in = 180, looseness=1.1] node (a1) [pos = .37] {}
(d) to [looseness=0] node (I1) [right, xshift = -2] {$\scriptstyle I_1$} (c) to [out = 180, in = 45, looseness=1.1] 
(h) to [out = 45 + 180, in = -225, looseness=1.1] node[left, xshift = 2] {$\scriptstyle I_4$} (g) to [out = -225 + 180, in = 180, looseness=1.1] (b);
\draw[->] (a1.center) -- ++ (180:0.01); \filldraw[fill = gray!30] (b') to node (I5) [left, xshift = 2] {$\scriptstyle I_5$} (a') to [out = 0, in = -45 + 180, looseness=1.1] node 
(a1') [pos = .6] {} (e') to [out = -45, in = 225 + 180, looseness=1.1] node[left, xshift = 1] {$\scriptstyle I_8$}
(f') to [out = 225, in = 0, looseness=1.1] (d') to [looseness=0] node (I7) [left, xshift = 2] {$\scriptstyle I_7$} (c') to [out = 0, in = 45 + 180, looseness=1.1] 
(h') to [out = 45, in = -225 + 180, looseness=1.1] node[right, xshift = -2] {$\scriptstyle I_6$} (g') to [out = -225, in = 0, looseness=1.1] (b');
\draw[->] (a1'.center) -- ++ (180:0.01); \node at (-3.1,2.4)  {$\scriptstyle \ID_l$}; \node at (6.7,2.4)  {$\scriptstyle \ID_r$};
\draw (I1) edge [<-,bend left=35] node [above]  {$\scriptscriptstyle \varphi$} (I5); \draw (I3) edge [<-,bend right=35] node [below]  {$\scriptscriptstyle \psi$} (I7);
} % tikzmath
\,\,\,\leadsto\,\,\, \tikzmath[scale=.35]{
\coordinate (a) at (0,0);\coordinate (b) at (.15,1);\coordinate (c) at (-.2,2);\coordinate (d) at (0,3);\coordinate (e) at (-5,.4);\coordinate (f) at (-5,2.6);\coordinate (g) at (-1,1.2);\coordinate (h) at (-1,2);
\coordinate (a') at (d);\coordinate (b') at (c);\coordinate (c') at (b);\coordinate (d') at (a);\begin{scope}[yshift = 85, rotate= 180]\coordinate (e') at (-5,.4);\coordinate (f') at (-5,2.6);\coordinate (g') at (-1,1.2);\coordinate (h') at (-1,2);\end{scope}
\filldraw[fill = gray!30] (b) to (a) to [out = 180, in = -45, looseness=1.1] (e) to [out = -45 + 180, in = 225, looseness=1.1](f) to [out = 225 + 180, in = 180, looseness=1.1] node (a1) [pos = .37] {}
(d) to [looseness=0](c) to [out = 180, in = 45, looseness=1.1] (h) to [out = 45 + 180, in = -225, looseness=1.1](g) to [out = -225 + 180, in = 180, looseness=1.1] (b);
\draw[->] (a1.center) -- ++ (180:0.01); \filldraw[fill = gray!30] (b') to (a') to [out = 0, in = -45 + 180, looseness=1.1] (e') to [out = -45, in = 225 + 180, looseness=1.1] (f') to [out = 225, in = 0, looseness=1.1](d') to [looseness=0](c') to [out = 0, in = 45 + 180, looseness=1.1](h') to [out = 45, in = -225 + 180, looseness=1.1](g') to [out = -225, in = 0, looseness=1.1] node (b1) [pos = .37] {}(b');
\draw[->] (b1.center) -- ++ (0:0.01); \node at (2.8,2.4)  {$\scriptstyle \Sigma=\ID_l\cup \ID_r$};
}\,, % tikzmath
\end{equation}
Let $S_b:=I_2\cup I_8$ and $S_m:=I_4\cup I_6$ be the two boundary circles of this annulus.

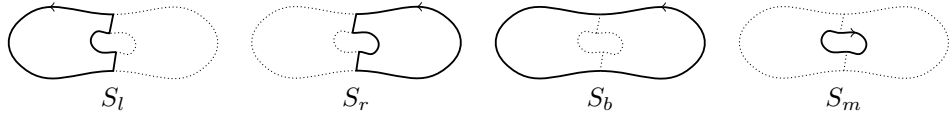
\[
\tikzmath[scale=.25]{\coordinate (a) at (0,0);\coordinate (b) at (.15,1);\coordinate (c) at (-.2,2);\coordinate (d) at (0,3);\coordinate (e) at (-5,.4);\coordinate (f) at (-5,2.6);\coordinate (g) at (-1,1.2);\coordinate (h) at (-1,2);\coordinate (a') at (d);\coordinate (b') at (c);\coordinate (c') at (b);\coordinate (d') at (a);\begin{scope}[yshift = 85, rotate= 180]\coordinate (e') at (-5,.4);\coordinate (f') at (-5,2.6);\coordinate (g') at (-1,1.2);\coordinate (h') at (-1,2);\end{scope}
\draw[line width=.7] (b) to (a) to [out = 180, in = -45, looseness=1.1] (e) 
	to [out = -45 + 180, in = 225, looseness=1.1] (f) to [out = 225 + 180, in = 180, looseness=1.1] node (a1) [pos = .37] {}
	(d) to [looseness=0] (c) to [out = 180, in = 45, looseness=1.1] 
	(h) to [out = 45 + 180, in = -225, looseness=1.1] (g) to [out = -225 + 180, in = 180, looseness=1.1] (b);
\draw[densely dotted] (a') to [out = 0, in = -45 + 180, looseness=1.1]node (a1') [pos = .6] {} (e') to [out = -45, in = 225 + 180, looseness=1.1]  
	(f') to [out = 225, in = 0, looseness=1.1] (d');
\draw[densely dotted] (c') to [out = 0, in = 45 + 180, looseness=1.1] (h') to [out = 45, in = -225 + 180, looseness=1.1]
	(g') to [out = -225, in = 0, looseness=1.1] node (b1) [pos = .37] {} (b');
\node at (0,-1.5) {$S_l$};
\draw[->] (a1.center) -- ++ (180:0.01);
} % tikzmath
\quad\tikzmath[scale=.25]{\coordinate (a) at (0,0);\coordinate (b) at (.15,1);\coordinate (c) at (-.2,2);\coordinate (d) at (0,3);\coordinate (e) at (-5,.4);\coordinate (f) at (-5,2.6);\coordinate (g) at (-1,1.2);\coordinate (h) at (-1,2);\coordinate (a') at (d);\coordinate (b') at (c);\coordinate (c') at (b);\coordinate (d') at (a);\begin{scope}[yshift = 85, rotate= 180]\coordinate (e') at (-5,.4);\coordinate (f') at (-5,2.6);\coordinate (g') at (-1,1.2);\coordinate (h') at (-1,2);\end{scope}
\draw[densely dotted] (a) to [out = 180, in = -45, looseness=1.1] (e) 
	to [out = -45 + 180, in = 225, looseness=1.1] (f) to [out = 225 + 180, in = 180, looseness=1.1] node (a1) [pos = .37] {} (d);
\draw[densely dotted] (c) to [out = 180, in = 45, looseness=1.1] 
	(h) to [out = 45 + 180, in = -225, looseness=1.1] (g) to [out = -225 + 180, in = 180, looseness=1.1] (b);
\draw[line width=.7] (b') to (a') to [out = 0, in = -45 + 180, looseness=1.1]node (a1') [pos = .6] {} (e') to [out = -45, in = 225 + 180, looseness=1.1]  
	(f') to [out = 225, in = 0, looseness=1.1] (d') to [looseness=0] (c') to [out = 0, in = 45 + 180, looseness=1.1] (h') to [out = 45, in = -225 + 180, looseness=1.1]
	(g') to [out = -225, in = 0, looseness=1.1] node (b1) [pos = .37] {} (b');
\node at (0,-1.5) {$S_r$};
\draw[->] (a1'.center) -- ++ (180:0.01);
} % tikzmath
\quad\tikzmath[scale=.25]{\coordinate (a) at (0,0);\coordinate (b) at (.15,1);\coordinate (c) at (-.2,2);\coordinate (d) at (0,3);\coordinate (e) at (-5,.4);\coordinate (f) at (-5,2.6);\coordinate (g) at (-1,1.2);\coordinate (h) at (-1,2);\coordinate (a') at (d);\coordinate (b') at (c);\coordinate (c') at (b);\coordinate (d') at (a);\begin{scope}[yshift = 85, rotate= 180]\coordinate (e') at (-5,.4);\coordinate (f') at (-5,2.6);\coordinate (g') at (-1,1.2);\coordinate (h') at (-1,2);\end{scope}
\draw[line width=.7] (a) to [out = 180, in = -45, looseness=1.1] (e) 
	to [out = -45 + 180, in = 225, looseness=1.1] (f) to [out = 225 + 180, in = 180, looseness=1.1] node (a1) [pos = .37] {} (d);
\draw[densely dotted] (c) to [out = 180, in = 45, looseness=1.1] 
	(h) to [out = 45 + 180, in = -225, looseness=1.1] (g) to [out = -225 + 180, in = 180, looseness=1.1] (b);
\draw[line width=.7] (a') to [out = 0, in = -45 + 180, looseness=1.1]node (a1') [pos = .6] {} (e') to [out = -45, in = 225 + 180, looseness=1.1]  
	(f') to [out = 225, in = 0, looseness=1.1] (d');
\draw[densely dotted] (c') to [out = 0, in = 45 + 180, looseness=1.1] (h') to [out = 45, in = -225 + 180, looseness=1.1]
	(g') to [out = -225, in = 0, looseness=1.1] node (b1) [pos = .37] {} (b');
\draw[densely dotted] (b') to (a')  (d') to (c');
\node at (0,-1.5) {$S_b$};
\draw[->] (a1'.center) -- ++ (180:0.01);
} % tikzmath
\quad\tikzmath[scale=.25]{\coordinate (a) at (0,0);\coordinate (b) at (.15,1);\coordinate (c) at (-.2,2);\coordinate (d) at (0,3);\coordinate (e) at (-5,.4);\coordinate (f) at (-5,2.6);\coordinate (g) at (-1,1.2);\coordinate (h) at (-1,2);\coordinate (a') at (d);\coordinate (b') at (c);\coordinate (c') at (b);\coordinate (d') at (a);\begin{scope}[yshift = 85, rotate= 180]\coordinate (e') at (-5,.4);\coordinate (f') at (-5,2.6);\coordinate (g') at (-1,1.2);\coordinate (h') at (-1,2);\end{scope}
\draw[densely dotted] (a) to [out = 180, in = -45, looseness=1.1] (e) 
	to [out = -45 + 180, in = 225, looseness=1.1] (f) to [out = 225 + 180, in = 180, looseness=1.1] node (a1) [pos = .37] {} (d);
\draw[line width=.7] (c) to [out = 180, in = 45, looseness=1.1] 
	(h) to [out = 45 + 180, in = -225, looseness=1.1] (g) to [out = -225 + 180, in = 180, looseness=1.1] (b);
\draw[densely dotted] (a') to [out = 0, in = -45 + 180, looseness=1.1]node (a1') [pos = .6] {} (e') to [out = -45, in = 225 + 180, looseness=1.1]  
	(f') to [out = 225, in = 0, looseness=1.1] (d');
\draw[line width=.7] (c') to [out = 0, in = 45 + 180, looseness=1.1] (h') to [out = 45, in = -225 + 180, looseness=1.1]
	(g') to [out = -225, in = 0, looseness=1.1] node (b1) [pos = .37] {} (b');
\draw[densely dotted] (b') to (a')  (d') to (c');
\node at (0,-1.5) {$S_m$};
\draw[->] (b1.center) -- ++ (0:0.01);
} % tikzmath
\]
We equip $S_b$ and $S_m$ with smooth structures that are compatible with those on $S_l$ and $S_r$ in the sense that,
locally around the trivalent points, there exist actions of $\mathfrak S_3$ as in \eqref{eq: Theta-graph}
whose restriction to each subinterval is smooth.
Equivalently, the smooth structures around a trivalent point should be modeled on \eqref{eq: Y-graph}.

The Hilbert space $H_{\mathit{ann}}$ is equipped with actions of the algebras $\cala(J)$ associated to the following subintervals of $S_m$ and $S_b$:
\begin{list}{$\bullet$}{\leftmargin=1cm \parsep=1pt}
\item For each $J\subset I_2$ or  $I_4$, the algebra $\cala(J)$ acts on $H_0(S_l)$, and thus on $H_{\mathit{ann}}$.
\item For each $J\subset I_6$ or  $I_8$, the algebra $\cala(J)$ acts on $H_0(S_r)$, and thus on $H_{\mathit{ann}}$.
\item A special case of \eqref{eq: compatibility between circledast and bartimes} says that whenever $D$ is a factor, we have
\[
(A\,\bar\otimes\,D) \circledast_{C\,\bar\otimes\,D} (B\,\bar\otimes\,D)\cong A \circledast_C B.
\]
If $J\subset S_m$ or $S_b$ intersects $\partial I_5\cup \partial I_7$ in one point, and that point is in the interior of $J$
then by Proposition \ref{prop: nets-and-fiber-product} and the above isomorphism, we have
\[
\qquad\cala(J)\,\cong\, \cala\big((J\cap S_l)\cup I_1\cup I_3\big)\circledast_{\cala(I_5)\,\bar\otimes \cala(I_7)} \cala\big((J\cap S_r)\cup I_5\cup I_7\big),
\]
where the map $(\cala(I_5)\,\bar\otimes \cala(I_7))^\op\to \cala((J\cap S_l)\cup I_1\cup I_3)$ is induced by $\varphi\cup\psi$.
That algebra then acts on $H_{\mathit{ann}}$ by 
Lemma \ref{lem: extends to A (*)_C B}.
\end{list}
By Lemma \ref{lem: open cover of circle => sector}, it follows that $H_{\mathit{ann}}$ is both
an $S_m$-sector and an $S_b$-sector, and that those two structures commute.
We abbreviate this by saying that $H_{\mathit{ann}}$ is an $S_m$-$S_b$-sector.

\begin{lemma}\label{lem: H_0 otimes H_0   <  H_Ann}
  Let $\cala$ be a conformal net with finite index, and let $S_l, S_r, S_b, S_m$, and $I_1,\ldots, I_8$ be as above.
  Then the $S_m$-$S_b$-sector $H_0(S_m) \otimes H_0(S_b)$ is a direct summand of $H_{\mathit{ann}}$, with multiplicity one.
\end{lemma}

\begin{proof}
Let
$A:=\cala(I_2\cup I_4)$, $B:=\cala(I_1\cup I_3)^\op\cong\cala(I_5\cup I_7)$, $C:=\cala(I_6\cup I_8)^\op$,
and
\[
H_l:=H_0(S_l),\quad H_r:=H_0(S_r),\quad H_b:=H_0(S_b),\quad H_m:=H_0(S_m).
\]
Since $\cala$ is irreducible, and since the intervals $I_2$, $I_4$, $I_6$, $I_8$ cover $S_m\cup S_b$, the Hilbert space $H_m\otimes H_b$ is an irreducible $A$-$C$-bimodule.
In order to show that $H_m\otimes H_b$ is a direct summand of $H_{\mathit{ann}} := H_l\boxtimes_B H_r$ with multiplicity one,
it is therefore enough to show that
\begin{equation}\label{eq: hom_A,C}
\mathrm{Hom}_{A,C}(H_m\otimes H_b, H_l\boxtimes_B H_r)
\end{equation}
is 1-dimensional.

Since $\cala$ has finite index, the bimodule ${}_A (H_l)_B$ is dualizable, and its dual is $H_0(\bar S_l)$ by Lemma \ref{lem: H_0(-S)}.
Letting $\check H_l:=H_0(\bar S_l)$, we may rewrite \eqref{eq: hom_A,C} as
\[
\mathrm{Hom}_{B,C}\big(\check H_l\boxtimes_A (H_m\otimes H_b), H_r\big).
\]
By Corollary \ref{cor: vacuum * vacuum = vacuum}, $\check H_l\boxtimes_A (H_m\otimes H_b)\cong H_m\boxtimes_{\cala(\bar I_4)}\check H_l\boxtimes_{\cala(I_2)} H_b$ is isomorphic to $H_0(S_r)$.
The above expression then reduces to $\mathrm{Hom}_{B,C}(H_r, H_r)$, which is indeed 1-dimensional.
\end{proof}

\subsubsection*{Sectors of nets of finite index}
Using $H_{\mathit{ann}}$ as a tool, we now show that given a finite index conformal net $\cala$,
there are finitely many isomorphism classes of irreducible $\cala$-sectors, and that
every irreducible sector is finite, in the following sense.

\begin{definition}\label{def:sector-of-a-conf-net : finite}
Let $S$ be a circle, $I\subset S$ an interval, and $I'$ the closure of its complement.
An $\cala$-sector on $S$ is called \emph{finite} if it is dualizable as an $\cala(I)$-$\cala(I')^\op$-bimodule.
\end{definition}

Note that the choice of interval $I$ in the above definition is irrelevant:

\begin{lemma} \label{lem: I in def of finite sector}
For $I_1$ and $I_2$ subintervals of a circle $S$, an $S$-sector is dualizable as $\cala(I_1)$-$\cala(I_1')^\op$-bimodule
if and only if it  is dualizable as $\cala(I_2)$-$\cala(I_2')^\op$-bimodule.
\end{lemma}

\begin{proof}
Pick a diffeomorphism $\varphi\in\Diff_+(S)$ that sends $I_1$ to $I_2$.
A sector $H$ is dualizable as $\cala(I_2)$-$\cala(I_2')^\op$-bimodule if and only if 
$\varphi^*H$ is dualizable as $\cala(I_1)$-$\cala(I_1')^\op$-bimodule.
The sector $\varphi^*H$ is isomorphic to $H$ by Corollary \ref{phi* == psi*}.
\end{proof}

Recall from Section \ref{subsec:sectors-for-nets} that for a sector $K\in\Rep(\cala)$ on the standard circle $S^1$, given another circle $S$, we denote by $K(S)\in \Rep_S(\cala)$ the corresponding sector on $S$.
It is given by $K(S):=\varphi^*K$ for any $\varphi\in\Diff_+(S,S^1)$,
and is only well defined up to non-canonical isomorphism.

Let us also recall that our conformal nets are irreducible, and that all the von Neumann algebras are assumed to have separable preduals.

\begin{theorem}[Lemma 13 and Corollaries 14 and 39 of \cite{Kawahigashi-Longo-Mueger(2001multi-interval)}]     \label{thm: KLM -- all irreducible sectors are finite}
Let $\cala$ be a conformal net with finite index.
Then all $\cala$-sectors are (possibly infinite) direct sums of irreducible sectors, and all irreducible $\cala$-sectors are finite.
Moreover, there are only finitely many isomorphism classes of irreducible sectors.
\end{theorem}

\begin{proof}
Let $S_l$, $S_r$, $S_b$, $S_m$ and $I_1,I_2,\ldots,I_8$ be as before:
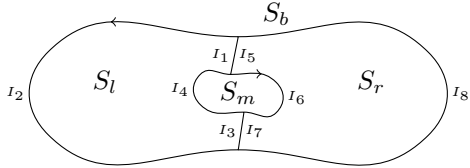
\[
\tikzmath[scale=.5]{\coordinate (a) at (0,0);\coordinate (b) at (.15,1);\coordinate (c) at (-.2,2);\coordinate (d) at (0,3);\coordinate (e) at (-5,.4);\coordinate (f) at (-5,2.6);\coordinate (g) at (-1,1.2);\coordinate (h) at (-1,2);\coordinate (a') at (d);\coordinate (b') at (c);\coordinate (c') at (b);\coordinate (d') at (a);\begin{scope}[yshift = 85, rotate= 180]\coordinate (e') at (-5,.4);\coordinate (f') at (-5,2.6);\coordinate (g') at (-1,1.2);\coordinate (h') at (-1,2);\end{scope}
\draw (b) to node (I3) [left, xshift = 2] {$\scriptscriptstyle I_3$} node (I7) [right, xshift = -2] {$\scriptscriptstyle I_7$} (a) to [out = 180, in = -45, looseness=1.1] (e) 
	to [out = -45 + 180, in = 225, looseness=1.1] node[left, xshift = 2] {$\scriptscriptstyle I_2$} (f) to [out = 225 + 180, in = 180, looseness=1.1] node (a1) [pos = .37] {}
	(d) to [looseness=0] node (I1) [left, xshift = 2] {$\scriptscriptstyle I_1$}node (I5) [right, xshift = -2] {$\scriptscriptstyle I_5$} (c) to [out = 180, in = 45, looseness=1.1] 
	(h) to [out = 45 + 180, in = -225, looseness=1.1] node[left, xshift = 2] {$\scriptscriptstyle I_4$} (g) to [out = -225 + 180, in = 180, looseness=1.1] (b);
\draw (a') to [out = 0, in = -45 + 180, looseness=1.1]node (a1') [pos = .6] {} (e') to [out = -45, in = 225 + 180, looseness=1.1] node[right, xshift = -2] {$\scriptscriptstyle I_8$} 
	(f') to [out = 225, in = 0, looseness=1.1] (d') (c') to [out = 0, in = 45 + 180, looseness=1.1] (h') to [out = 45, in = -225 + 180, looseness=1.1] node[right, xshift = -2] {$\scriptscriptstyle I_6$}
	(g') to [out = -225, in = 0, looseness=1.1] node (b1) [pos = .37] {} (b');
\node at (-3.5,1.8) {$S_l$};
\node at (3.5,1.8) {$S_r$};
\node at (0,1.5) {$S_m$};
\node at (1,3.6) {$S_b$};
\draw[->] (a1.center) -- ++ (180:0.01); \draw[->] (b1.center) -- ++ (0:0.01);
} % tikzmath
\]
and again let $H_l:=H_0(S_l)$, $H_r:=H_0(S_r)$, $H_b:=H_0(S_b)$, $H_m:=H_0(S_m)$, and
$H_{\mathit{ann}} := H_0(S_l)\boxtimes_{\cala(I_5)\bar\otimes\cala(I_7)} H_0(S_r)$.
We will also use the abbreviations
\begin{gather*}
A:=\cala(I_2\cup I_4),\quad B:=\cala(I_1\cup I_3)^\op\cong\cala(I_5\cup I_7),\quad C:=\cala(I_6\cup I_8)^\op,
\\
A_l:=\cala(I_2),\quad A_m:=\cala(I_4)^\op,\quad C_m:=\cala(I_6)^\op,\quad C_r:=\cala(I_8).
\end{gather*}
Since ${}_A(H_l){}_B$ and ${}_B(H_r){}_C$ are dualizable bimodules, $H_{\mathit{ann}}=H_l\boxtimes_B H_r$ is dualizable as an $A$-$C$-bimodule, and
therefore splits into finitely many irreducible summands (see Lemma \ref{lem: endolgebra dim<oo}).

Forget the actions of $\{\cala(I)\}_{I\subset S_m}$ on $H_{\mathit{ann}}$, and only view it as an $S_b$-sector.
The von Neumann algebra generated by $\{\cala(I)\}_{I\subset S_b}$ on $H_{\mathit{ann}}$ has a finite-dimensional center.  (Otherwise, it would contradict the fact that ${}_A(H_{\mathit{ann}})_C$ splits into finitely many irreducible summands:
every central projection in that algebra commutes with $\cala(I_2)$, $\cala(I_4)$, $\cala(I_6)$, and $\cala(I_8)$, and therefore induces a non-trivial direct sum decomposition of $H_{\mathit{ann}}$
as $A$-$C$-bimodule.)
We can therefore write $H_{\mathit{ann}}$ as a direct sum of finitely many factorial $S_b$-sectors:
\begin{equation}\label{eq: decomposition of Hann}
H_{\mathit{ann}} \cong K_1(S_b)\oplus\ldots\oplus K_n(S_b).
\end{equation}
Here $K_1, \ldots, K_n$ are $\cala$-sectors, and a sector is called factorial if its endomorphism algebra is a factor.

Given an arbitrary factorial $\cala$-sector $K$, we now show that there exists a $K_i$ in the above list to which $K$ is stably isomorphic, i.e., such that
$K\otimes \ell^2\cong K_i\otimes \ell^2$.
Letting $S_2:=I_2\cup_{\partial I_2} \bar I_2$ and $S_4:=I_4\cup_{\partial I_4} \bar I_4$, 
we have the following isomorphisms of $S_l$-sectors:
\[
K(S_2)\boxtimes_{A_l} H_l \,\,\cong\,\, K(S_l) \,\,\cong\,\, H_l \boxtimes_{A_m} K(S_4)
\]
Fusing with $H_r$ over $B$, we thus get
\[
K(S_2)\boxtimes_{A_l}H_{\mathit{ann}}\cong H_{\mathit{ann}} \boxtimes_{A_m} K(S_4).
\]
We also have $K(S_b)\cong K(S_2) \boxtimes_{A_l}H_b$.
By Lemma \ref{lem: H_0 otimes H_0   <  H_Ann}, it follows that
\[
\begin{split}
K(S_b)\otimes H_m \,\cong\,    K(S_2)\,&\boxtimes_{A_l}(H_b\otimes H_m)\\ 
                               \subset\,  K(S_2)\,&\boxtimes_{A_l}H_{\mathit{ann}}\,\,\cong\,\, H_{\mathit{ann}}\boxtimes_{A_m} K(S_4).
\end{split}
\]
Since $A_m$ is a factor, it has only one stable isomorphism class of modules.
In particular, $K(S_4)$ and $L^2A_m$ are stably isomorphic as $A_m$-modules.
We therefore get a (non-canonical) inclusion of $S_b$-sectors:
\[
\begin{split}
K(S_b) \otimes \ell^2 \cong
K(S_b) \otimes H_m \otimes \ell^2 &\subset
H_{\mathit{ann}} \boxtimes_{A_m} K(S_4) \otimes \ell^2 \\ & \cong 
H_{\mathit{ann}} \boxtimes_{A_m} L^2A_m \otimes \ell^2 \cong
H_{\mathit{ann}} \otimes \ell^2,
\end{split}
\]
where the first equality uses an arbitrary unitary isomorphism $\ell^2 \cong H_m \otimes \ell^2$ of Hilbert spaces.
The sector $K(S_b)$ is factorial. 
It therefore maps to a single summand $K_i\otimes \ell^2$ of $H_{\mathit{ann}} \otimes \ell^2$.
It follows that $K$ and $K_i$ are stably isomorphic.
In particular, this shows that there are at most finitely many stable isomorphism classes of factorial $\cala$-sectors.

By \cite[Appendix C]{Kawahigashi-Longo-Mueger(2001multi-interval)}, since the algebras $\cala(I)$ have separable preduals, any $\cala$-sector can be disintegrated into irreducible ones.
As a consequence, if there exists a factorial sector of type~{\it II} or~{\it III}, then, again by the arguments in \cite[Appendix C]{Kawahigashi-Longo-Mueger(2001multi-interval)}, there must be uncountably many non-isomorphic irreducible $\cala$-sectors.
This is impossible, and so all factorial sectors must be of type~$I$.

Let us now go back to $H_{\mathit{ann}}$ and analyse it as an $S_b$-$S_m$-sector.
Since every summand $K_i(S_b)$ in the decomposition \eqref{eq: decomposition of Hann} is a type~$I$ factorial $S_b$-sector of $\cala$,
we can write it as $P_i\otimes Q_i$, where $P_i$ is an irreducible $S_b$-sector, and $Q_i=\mathrm{Hom}(P_i,H_{\mathit{ann}})$ is some multiplicity space.
The multiplicity spaces $Q_i$ then carry residual $S_m$-sector structures, inherited from that of $H_{\mathit{ann}}$.
The decomposition \eqref{eq: decomposition of Hann} then becomes
\[
{}_{A_l\,\bar\otimes\,A_m} (H_{\mathit{ann}})\, {}_{C_r\,\bar\otimes\,C_m} 
\,\,\cong\,\,\,
\bigoplus_i
{}_{A_l} (P_i)\, {}_{C_r} \otimes {}_{A_m} (Q_i)\, {}_{C_m} .
\]
Since $H_{\mathit{ann}}$ is a dualizable $A$-$C$-bimodule, the bimodules ${}_{A_l} (P_i) {}_{C_r}$ are also dualizable.
To finish the argument, recall that any irreducible $\cala$-sector is isomorphic to one of the $P_i$, and so any irreducible sector is finite.
\end{proof}

\subsubsection*{Duals of finite sectors}
Given a conformal net $\cala$ with finite index, let $\Delta=\Delta_\cala$ be the finite set of isomorphism classes of irreducible $\cala$-sectors.
For every $\lambda\in \Delta$, let $H_\lambda$ be a representative of the isomorphism class.
The set $\Delta$ has an involution $\lambda\mapsto \bar\lambda$ given by sending a Hilbert space $H_\lambda$ to its 
pullback $H_{\bar\lambda}\cong j^*H_\lambda$ along some element $j\in\Diff_-(S^1)$, as defined in \eqref{eq: functor phi^*}
(note that $H_{\bar\lambda}$ is only well defined up to non-canonical isomorphism).

\begin{lemma}\label{lem: Hl(bS) cong i^*Hbl(S)}
Let $i:\bar S \to S$ denote the identity map.
Then, there is an isomorphism of $\bar S$-sectors $i^*H_\lambda(S)\cong H_{\bar\lambda}(\bar S)$.
\end{lemma}

\begin{proof}
Without loss of generality, we take $S$ to be the standard circle $S^1$.
Let $j:S^1\to S^1$ be a reflection, and let us denote by $\bar j$
the same map, viewed as an orientation-preserving map from $\bar S^1$ to $S^1$. 
We have $H_{\bar \lambda}(\bar S^1)\cong \bar j^*H_{\bar\lambda}$ and $H_{\bar\lambda}\cong j^* H_\lambda$.
It follows that $H_{\bar\lambda}(\bar S^1)\cong \bar j^* j^* H_\lambda \cong i^*H_\lambda$.
\end{proof}

The following two lemmas describe the duals of sectors with respect to Connes fusion of bimodules and the monoidal product on sectors, respectively.

\begin{lemma}\label{lem: dual of H_lambda}
Let $S$ be a circle, decomposed into two subintervals $I$ and~$I'$.
Then the dual of the bimodule ${}_{\cala(I)}H_\lambda(S)_{\cala(\bar I')}$ is ${}_{\cala(\bar I')}H_{\bar \lambda}(\bar S)_{\cala(I)}$.
\end{lemma}

\begin{proof}
Letting $i:\bar S\to S$ be the identity map, we have $H_{\bar \lambda}(\bar S)\cong i^*H_\lambda(S)$ by the previous lemma.
Here, the sector $H_{\bar \lambda}(\bar S)$ is the complex conjugate of $H_\lambda$, with action
$a\overline \xi := \overline{a^*\xi}$ for $a\in \cala(\bar J)$, $J\subset S$.
The bimodule ${}_{\cala(I)}H_\lambda(S)_{\cala(\bar I')}$ has actions given by
\begin{equation}\label{eq:1 axib BIS}
\qquad \qquad\quad a \xi b:=ab\xi \qquad\quad a\in \cala(I),\, b\in\cala(\bar I'),
\end{equation}
and the bimodule ${}_{\cala(\bar I')}H_{\bar \lambda}(\bar S)_{\cala(I)}$ has actions given by
\begin{equation}\label{eq:2 axib BIS}
\qquad \qquad b \bar \xi a:=ab\bar\xi = \overline{a^*b^*\xi} \qquad\quad a\in \cala(I),\, b\in\cala(\bar I').
\end{equation}
Comparing \eqref{eq:1 axib BIS} and \eqref{eq:2 axib BIS}, we see that $b\bar\xi a=\overline{a^*\xi b^*}$, and so
the two bimodules are dual to each other (see the discussion after~\ref{def:dual} or~\cite[Cor 6.12]{BDH(Dualizability+Index-of-subfactors)}).
\end{proof}

\begin{lemma}\label{lem: identification of the dual sector}
The sector $H_{\bar\lambda}(S)$ is dual to $H_\lambda(S)$
with respect to the monoidal structure \eqref{eq: associative operation on A-reps} on $\Rep_S(\cala)$.
\end{lemma}

\begin{proof}
Let $j\in\Diff_-(S^1)$ be an involution fixing $\partial I$, for $I \subset S^1$.
The sector $H_{\bar\lambda} \cong j^*H_\lambda$ is the complex conjugate $\overline {H_\lambda}$, with actions
$a\overline \xi := \overline{\cala(j)(a^*)\xi}$ for $a\in \cala(J)$, $J\subset S^1$.
Following \eqref{eq: associative operation on A-reps}, we view $H_\lambda$ as an $\cala(I)$-$\cala(I)$-bimodule, with actions
\begin{equation}\label{eq:1 axib}
\qquad a \xi b:=a\cala(j)(b)\xi \qquad\quad a,b\in \cala(I).
\end{equation}
The same procedure on $j^*H_\lambda$ yields the following left and right actions on $\overline{H_\lambda}$:
\begin{equation}\label{eq:2 axib}
\qquad a \bar\xi b:=\overline{b^*\cala(j)(a^*)\xi} \qquad\quad a,b\in \cala(I).
\end{equation}
Comparing \eqref{eq:1 axib} and \eqref{eq:2 axib}, we see that $a\bar\xi b=\overline{b^*\xi a^*}$, and so
$j^*H_\lambda$ is the dual of $H_\lambda$ (again by~\cite[Cor 6.12]{BDH(Dualizability+Index-of-subfactors)}).
\end{proof}

\subsubsection*{Computation of the annular sector}
The following result, even though phrased in a rather different language, is essentially equivalent to~\cite[Theorem 9]{Kawahigashi-Longo-Mueger(2001multi-interval)}.
Recall that $\cala$ is irreducible.

\begin{theorem}
  \label{thm:KLM}
  Let $\cala$ be a conformal net with finite index, and let $S_m$, $S_b$, and $H_{\mathit{ann}}$ be as in \eqref{eq: lem: H_0 otimes H_0   <  H_Ann}.
  We then have a non-canonical isomorphism of $S_m$-$S_b$-sectors
  \begin{equation} \label{eq:KLM}
    H_{\mathit{ann}} \,\cong\,\, \bigoplus_{\lambda\in\Delta} H_\lambda(S_m) \otimes H_{\bar \lambda}(S_b).
  \end{equation}
\end{theorem}
  We draw this isomorphism as
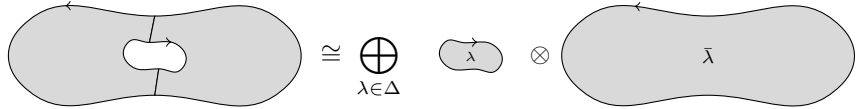
\[
\def\coords{
  \coordinate (a) at (0,0);
  \coordinate (b) at (.15,1);
  \coordinate (c) at (-.2,2);
  \coordinate (d) at (0,3);
  \coordinate (e) at (-5,.4);
  \coordinate (f) at (-5,2.6);
  \coordinate (g) at (-1,1.2);
  \coordinate (h) at (-1,2);
  \coordinate (a') at (d);
  \coordinate (b') at (c);
  \coordinate (c') at (b);
  \coordinate (d') at (a);
\begin{scope}[yshift = 85, rotate= 180]
  \coordinate (e') at (-5,.4);
  \coordinate (f') at (-5,2.6);
  \coordinate (g') at (-1,1.2);
  \coordinate (h') at (-1,2);
\end{scope}
}
\tikzmath[scale=.35]{
\coords
\filldraw[fill = gray!30] (b) to 
	(a) to [out = 180, in = -45, looseness=1.1] 
	(e) to [out = -45 + 180, in = 225, looseness=1.1]
	(f) to [out = 225 + 180, in = 180, looseness=1.1] node (a1) [pos = .37] {}
	(d) to [looseness=0]
	(c) to [out = 180, in = 45, looseness=1.1] 
	(h) to [out = 45 + 180, in = -225, looseness=1.1]
	(g) to [out = -225 + 180, in = 180, looseness=1.1] (b);
\draw[->] (a1.center) -- ++ (180:0.01);
\filldraw[fill = gray!30] (b') to
	(a') to [out = 0, in = -45 + 180, looseness=1.1] 
	(e') to [out = -45, in = 225 + 180, looseness=1.1]
	(f') to [out = 225, in = 0, looseness=1.1]
	(d') to [looseness=0]
	(c') to [out = 0, in = 45 + 180, looseness=1.1] 
	(h') to [out = 45, in = -225 + 180, looseness=1.1]
	(g') to [out = -225, in = 0, looseness=1.1]                                     node (b1) [pos = .37] {}
	(b');
\draw[->] (b1.center) -- ++ (0:0.01);
} % tikzmath
  \; \cong \;
  \bigoplus_{\lambda\in\Delta}\, \;
\tikzmath[scale=.35]{\coords
\useasboundingbox (-2,.4) rectangle (2,2.6);
\filldraw[fill = gray!30] 
	(c) to [out = 180, in = 45, looseness=1.1] 
	(h) to [out = 45 + 180, in = -225, looseness=1.1]
	(g) to [out = -225 + 180, in = 180, looseness=1.1]
	(c') to [out = 0, in = 45 + 180, looseness=1.1] 
	(h') to [out = 45, in = -225 + 180, looseness=1.1]
	(g') to [out = -225, in = 0, looseness=1.1]                                     node (b1) [pos = .7] {}
	(b');
\draw[->] (b1.center) -- ++ (0:0.01);                        \node at (0,1.5)  {$\scriptscriptstyle\lambda$};
} % tikzmath
\otimes
\tikzmath[scale=.35]{
\coords
\filldraw[fill = gray!30]
	(a) to [out = 180, in = -45, looseness=1.1] 
	(e) to [out = -45 + 180, in = 225, looseness=1.1]
	(f) to [out = 225 + 180, in = 180, looseness=1.1] node (a1) [pos = .5] {}
	(a') to [out = 0, in = -45 + 180, looseness=1.1] 
	(e') to [out = -45, in = 225 + 180, looseness=1.1]
	(f') to [out = 225, in = 0, looseness=1.1] (d');                        \node at (0,1.5)  {$\scriptstyle\bar \lambda$};
\draw[->] (a1.center) -- ++ (175:0.01);
} % tikzmath
\]

\begin{proof}
Let $H_l=H_0(S_l)$, $H_r=H_0(S_r)$, $\check H_l=H_0(\bar S_l)$, and $A$, $B$, $C$, $A_l$, $A_m$, $C_m$, $C_r$
be as in the proofs of Lemma \ref{lem: H_0 otimes H_0   <  H_Ann} and Theorem \ref{thm: KLM -- all irreducible sectors are finite}.
The Hilbert space $H_{\mathit{ann}} = H_l\boxtimes_B H_r$ is a finite $A$-$C$-bimodule and therefore splits into finitely many irreducible summands.
By the argument in the proof of Theorem \ref{thm: KLM -- all irreducible sectors are finite}, each irreducible summand is the tensor product of an
irreducible $S_m$-sector and an irreducible $S_b$-sector, and so we can write $H_{\mathit{ann}}$ as a direct sum
\begin{equation*} 
H_{\mathit{ann}} \cong \bigoplus_{\lambda,\mu\in \Delta} N_{\lambda\mu}\, H_{\lambda}(S_m) \otimes H_{\mu}(S_b)
\end{equation*}
with finite multiplicities $N_{\lambda\mu}\in\IN$.

Given $\lambda,\mu\in \Delta$, we now compute $N_{\lambda\mu}$. 
By slight abuse of notation, we abbreviate $H_\lambda:=H_\lambda(S_m)$ and $H_\mu:=H_\mu(S_b)$.
We also let $K:=(H_\lambda\boxtimes H_\mu)(S_r)$, where the operation of fusion of sectors is described in Definition \ref{def: fusion v h}.
We then have
\[\begin{split}
\mathrm{Hom}_{A,C}\big(H_\lambda \otimes H_{\mu},H_{\mathit{ann}}\big)
\cong\,\,&\mathrm{Hom}_{A,C}\big(H_\lambda \otimes H_{\mu},H_l\boxtimes_B H_r\big)\\
\cong\,\,&\mathrm{Hom}_{B,C}\big(\check H_l\boxtimes_A(H_\lambda \otimes H_{\mu}),H_r\big)\\
\cong\,\,&\mathrm{Hom}_{B,C}\big(H_\lambda\boxtimes_{A_m^\op}\check H_l\boxtimes_{A_l^{\phantom{\op}}}\!\! H_{\mu},H_r\big)\\
\cong\,\,&\mathrm{Hom}_{B,C}\big(K,H_r\big)
\cong\,\,\begin{cases}
\IC&\text{if $\mu=\bar\lambda$}\\
0&\text{otherwise,}\\
\end{cases}
\end{split}\]
where the last equality is given by
Lemma \ref{lem: characterization of duals}.
If follows that $N_{\lambda\mu}=\delta_{\mu\bar\lambda}$.
\end{proof}

\begin{remark}
The isomorphism \eqref{eq:KLM} is non-canonical.
It does not even make sense to ask whether or not it is canonical since the right-hand side of the equation is only well defined up to non-canonical isomorphism.
\end{remark}

Given $\lambda\in \Delta$, consider the statistical dimension $d_\lambda:=\dim({}_{\cala(S^1_\top)}(H_\lambda)_{\cala(S^1_\bot)})$ of $H_\lambda$ as a bimodule for two complementary intervals.

\begin{corollary}
If a conformal net $\cala$ has finite index, then the index satisfies
\[
\mu(\cala) \,=\, \sum_{\lambda\in\Delta} d_\lambda^2.
\]
\end{corollary}
\begin{proof}
Let $I_1,\ldots, I_8$ be as in \eqref{eq: lem: H_0 otimes H_0   <  H_Ann}.  By the multiplicativity of dimension under Connes fusion~\cite[Prop 5.2]{BDH(Dualizability+Index-of-subfactors)}, the statistical dimension of $H_{\mathit{ann}}$ as an $\cala(I_2\cup I_4)$-$\cala(I_6\cup I_8)$-bimodule is the square of the statistical dimension of
${}_{\cala(\tikzmath[scale=.1]{
\useasboundingbox (-1.1,-1) rectangle (1.1,1);
\draw (-45:1) arc (-45:45:1) (135:1) arc (135:225:1);
\draw[line cap = round, line width = .2](.65,-.2) to[bend right = 25] (1,.2)(1,.2) to[bend right = 25] (1.35,-.2);
\pgftransformrotate{180}
\draw[line cap = round, line width = .2](.65,-.2) to[bend right = 25] (1,.2)(1,.2) to[bend right = 25] (1.35,-.2);
})}H_0(\cala)_{\cala(
\tikzmath[scale=.1]{
\useasboundingbox (-1,-1) rectangle (1,1);
\draw (45:1) arc (45:135:1) (-45:1) arc (-45:-135:1);
\pgftransformrotate{90}
\draw[line cap = round, line width = .2](.65,-.2) to[bend right = 25] (1,.2)(1,.2) to[bend right = 25] (1.35,-.2);
\pgftransformrotate{180}
\draw[line cap = round, line width = .2](.65,-.2) to[bend right = 25] (1,.2)(1,.2) to[bend right = 25] (1.35,-.2);
})}$.
In other words,
\[
\dim\big({}_{\cala(I_2\cup I_4)}(H_{\mathit{ann}})_{\cala(I_6\cup I_8)}\big) = \mu(\cala). 
\]
The result follows from \eqref{eq:KLM} and the fact~\eqref{eq: dim (A H_B)= dim (B H_A)} that $d_\lambda=d_{\bar \lambda}$.
\end{proof}

\subsection{The Hilbert space associated to a surface}\label{sec: The Hilbert space associated to a surface}

Given a closed 1-manifold $M$, we call a Hilbert space $H$ an \emph{$M$-sector of $\cala$} if it comes equipped with compatible actions $\rho_I:\cala(I)\to \bfB(H)$
for all the intervals~$I\subset M$.
Here, compatible means that $\rho_J=\rho_I|_{\cala(J)}$ whenever $J$ is contained in $I$,
and that $\cala(I)$ and $\cala(J)$ commute whenever $I$ and $J$ have disjoint interiors.
We denote by $\Rep_M(\cala)$ the category of $M$-sectors of $\cala$.

\begin{lemma}\label{lem: open cover of circle => sector ++}
Let $M$ be a closed 1-manifold, and let $I_i\subset M$ be intervals whose interiors cover $M$.
Let $\rho_{i} \colon \cala(I_i) \to \bfB(H)$ be actions 
subject to the conditions:
\begin{enumerate}
 \item $\rho_{i}|_{\cala(I_i\cap I_j)}=\rho_{j}|_{\cala(I_i\cap I_j)}$,
 \item if $J\subset I_i$ and $K\subset I_j$ are disjoint, then $\rho_{i}(\cala(J))$ commutes with $\rho_{j}(\cala(K))$.
\end{enumerate}
These actions extend in a unique way to the structure of an 
$M$-sector of $\cala$.

Moreover, if $I \subset M$ and $J\subset M$ are disjoint intervals, then the action of $\cala(I)\otimes_\alg\cala(J)$ on $H$ extends to the spatial tensor product $\cala(I)\,\bar\otimes\,\cala(J)$.
\end{lemma}

\begin{proof}
The first statment is an immediate generalization of Lemma \ref{lem: open cover of circle => sector}.

If the disjoint intervals $I$ and $J$ belong to the same connected component of $M$, then
we may find an interval $K$ that contains both, in which case the map $\cala(I)\otimes_\alg\cala(J)\to \cala(K)\to \bfB(H)$ extends to $\cala(I)\,\bar\otimes\,\cala(J)$ by the split property.
Assume now that $I$ and $J$ belong to two different connected components; call those components $S_1$ and $S_2$.
Applying Theorem \ref{thm: KLM -- all irreducible sectors are finite} to $H$ viewed as an $S_1$-sector, we may decompose it as
\begin{equation}\label{eq: sum decomposition of arbitrary sector*}
H \,\,\cong\,\, \bigoplus_{\lambda\in \Delta}\, H_\lambda(S_1)\otimes M_\lambda,
\end{equation}
where $\Delta$ is the set of isomorphism classes of irreducible $S_1$-sectors of $\cala$, the multiplicity spaces $M_\lambda=\hom_{\Rep_{S_1}(\cala)}(H_\lambda(S_1),H)$ are Hilbert spaces, and the tensor product is the completed tensor product of Hilbert spaces.
The multiplicity spaces $M_\lambda$ have residual actions of $\cala(I')$ for every interval $I' \subset M$ not contained in $S_1$.
In particular, they are equipped with actions of $\cala(J)$. 
It now follows from the form of the decomposition \eqref{eq: sum decomposition of arbitrary sector*} that $H$ supports an action of $\cala(I)\,\bar\otimes\,\cala(J)$.
\end{proof}

In this section, we give a construction of a Hilbert space $V(\Sigma)\in\Rep_{\partial\Sigma}(\cala)$
associated to a topological surface with smooth boundary---we insist that every connected component of the surface have non-empty boundary.
The construction depends on the auxiliary choice of particular kind of cell decomposition of $\Sigma$.
We will show later, in the second paper of this series \cite{BDH(all-together)}, that it is actually independent of any choice,
and that the construction also makes sense for surfaces without boundary (at least when the conformal net has finite index).

A cell decomposition of a topological surface is called \emph{regular} if all the attaching maps are injective. 
We call a cell decomposition $\Sigma=\ID_1\cup\ldots\cup\ID_n$ \emph{ordered} if the set $\{\ID_1,\ldots,\ID_n\}$ of 2-cells is ordered.
Finally, a cell decomposition is \emph{collared} if the 1-cells are equipped with smooth structures and with germs of (1-dimensional) local coordinates at their two end-points.
The construction of $V(\Sigma)$ that we present here depends on the choice of a regular ordered collared cell decomposition of $\Sigma$.

The idea of the construction is to associate to each 2-cell $\ID_i\subset \Sigma$ the vacuum sector $H_0(\partial \ID_i)$, and to then glue them using Connes fusion.
Let $\Sigma_i:=\ID_1\cup\ldots\cup\ID_i$, and let us assume that the $M_i:=\ID_i\cap\Sigma_{i-1}$ contain no isolated points.
Give the 1-manifolds $M_i$ the orientations coming from $\partial \ID_i$.
Note that these manifolds have natural smooth structures due to the presence of collars
(although they are typically not smoothly embedded in $\Sigma$, even if the latter is smooth):
in any dimension, if two smooth manifolds have collars along their boundary, then glueing them along some boundary components will again produce a smooth manifold.
We also make the assumption that the manifolds $M_i$ are disjoint unions of intervals (in the sequel paper, that condition will be removed).
Finally, we define inductively $V(\Sigma_i)$ by
\begin{equation}\label{eq: V(D_i) -- BIS}
V(\Sigma_i):=\begin{cases}
H_0(\partial \ID_1)& \text{for }i=1\\
V(\Sigma_{i-1})\boxtimes_{\cala(M_i)}H_0(\partial \ID_i)& \text{for } 1<i\le n
\end{cases}
\end{equation}
where the algebras $\cala(M_i)$ are described in the beginning of Section \ref{subs: The index of a conformal net}.

For the above construction to work, we need to check that the algebras $\cala(M_i)$ act on $V(\Sigma_{i-1})$ and on $H_0(\partial \ID_i)$.
The existence of an action of $\cala(M_i)$ on $H_0(\partial \ID_i)$ follows from the split property axiom.
To see that $\cala(M_i)$ acts on $V(\Sigma_{i-1})$, it is enough to show that $V(\Sigma_{i-1})$ is a $\partial \Sigma_{i-1}$-sector of $\cala$.
Assuming by induction that $V(\Sigma_{i-2})$ is a $\partial \Sigma_{i-2}$-sector, that $V(\Sigma_{i-1})$ is a $\partial \Sigma_{i-1}$-sector is a consequence of the following lemma.

\begin{lemma}\label{lem: NNN}
Let $N_1$, $N_2$, $N_3$ be 1-manifolds equipped with collars, and with an identification of their boundaries $\partial N_1=\partial N_2=\partial N_3$.
Orient them so that
\[
M_1:=N_1\cup \bar N_2\qquad M_2:=N_2\cup N_3\qquad M_3:=N_1\cup N_3
\]
are closed and oriented (though not necessarily connected):
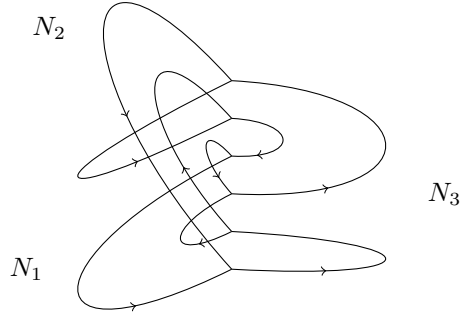
\[
\tikzmath{\useasboundingbox (-3,-.7) rectangle (3.1,3.7);
\node at (-2.7,0) {$N_1$};\node at (-2.4,3.2) {$N_2$};\node at (2.8,1) {$N_3$};\pgftransformcm{.9}{-.05}{0}{.5}{\pgfpoint{0pt}{0pt}} 
\draw (0,2) .. controls (3,2) and (3,5) .. (0,5);\draw (0,3) .. controls (1,3) and (1,4) .. (0,4);\draw (0,0) .. controls (3,0) and (3,1) .. (0,1);
\draw[->](1.37,.0965) -- (1.38,.099);\draw[->](1.4,2.29) -- (1.41,2.295);\draw[<-](0.37,3.056) -- (0.38,3.06);\pgftransformcm{-1}{-1}{0}{1}{\pgfpoint{0pt}{0pt}} 
\draw (0,0) .. controls (3,0) and (3,3) .. (0,3);\draw (0,1) .. controls (1,1) and (1,2) .. (0,2);\draw (0,4) .. controls (3,4) and (3,5) .. (0,5);
\draw[<-](1.4,.29) -- (1.41,.295);\draw[<-](1.37,4.0965) -- (1.38,4.099);\draw[->](0.47,1.096) -- (0.48,1.1);\pgftransformcm{.5}{1.45}{0}{1}{\pgfpoint{0pt}{0pt}} 
\draw (0,0) .. controls (5,0) and (5,5) .. (0,5);\draw (0,1) .. controls (3,1) and (3,4) .. (0,4);\draw (0,2) .. controls (1,2) and (1,3) .. (0,3);
\draw[->](1.4,1.29) -- (1.41,1.295);\draw[<-](0.37,2.056) -- (0.38,2.06);\draw[<-](3.06,0.99) -- (3.072,1.002);}
\]
Assume furthermore that none of the connected components of $N_2$ are circles.

Let $\cala$ be a conformal net, let $H_1$ be an $M_1$-sector, and let $H_2$ be an $M_2$-sector.
Then $H_3:=H_1\boxtimes_{\cala(N_2)}H_2$ is an $M_3$-sector of $\cala$.
\end{lemma}

\begin{proof}
Let $J\subset M_3$ be an interval that intersects $N_1\cap N_3$ in at most one point.
If $J\subset N_1$ (respectively if $J\subset N_3$), then $\cala(J)$ acts on $H_3$ via its action on $H_1$ (respectively on $H_2$).
If the point $J\cap N_1\cap N_3$ is in the interior of $J$ then, by Proposition \ref{prop: nets-and-fiber-product}, we have
\[
\qquad\cala(J)\,\cong\, \cala\big((J\cap N_1)\cup N_2\big)\circledast_{\cala(N_2)} \cala\big((J\cap N_3)\cup N_2\big),
\]
and that algebra acts on $H_1\boxtimes_{\cala(N_2)}H_2$ by Lemma \ref{lem: extends to A (*)_C B}.
The Hilbert space $H_3$ is therefore an $M_3$-sector by Lemma \ref{lem: open cover of circle => sector ++}.
\end{proof}

The requirement that $N_2$ does not have closed components is unnecessary.
In the second paper of this series, we will define $\cala(N_2)$ for any 1-manifold $N_2$, and in that more general case the proof of the above lemma goes through unchanged.
The collars can also be replaced by a slightly weaker piece of structure: one only needs smooth structures on $M_1$, $M_2$, and $M_3$ 
whose relationship to each other around a trivalent point is modeled on \eqref{eq: Y-graph}.

\begin{corollary}\label{cor: it's a partialSigma-sector}
The Hilbert space $V(\Sigma)$ is well defined (at this point, it still depends on a choice of cell decomposition of $\Sigma$), and is naturally equipped with the structure of a $\partial\Sigma$-sector.
\end{corollary}

\begin{proof}
Assuming by induction that $V(\Sigma_{i-1})$ is a $\partial\Sigma_{i-1}$-sector,
we apply Lemma \ref{lem: NNN} to $H_1:=V(\Sigma_{i-1})$ and $H_2:=H_0(\partial\ID_i)$,
with $N_2 := M_i$ as in \eqref{eq: V(D_i) -- BIS},
$N_1$ the closure of $\partial \Sigma_{i-1}\setminus M_i$, and $N_3$ the closure of $\partial\ID_i\setminus M_i$.
\end{proof}

\subsection{Characterization of finite-index conformal nets}

Recall that a dagger category is a category equipped with an involutive, identity-on-objects, contravariant endofunctor.  Denote by $\mathsf{Hilb}$ the dagger category of Hilbert spaces and bounded linear maps.
Let us call a dagger category $\calc$ a $\mathsf{Hilb}$-category if it is a module over $(\mathsf{Hilb},\otimes)$, that is, if there is a functor $\odot:\mathsf{Hilb}\times \calc\to \calc$
along with invertible associator and unitality natural transformations
\[
\odot\circ (\mathrm{id}\times\odot)\Rightarrow\odot\circ (\otimes\times\mathrm{id}):\mathsf{Hilb}\times\mathsf{Hilb}\times \calc\to \calc,\quad
\odot\circ\, (1\times \mathrm{id})\Rightarrow\mathrm{id}:\calc\to \calc
\]
subject to the obvious compatibility conditions.

Let us call a $\mathsf{Hilb}$-category $\calc$ a \emph{tensor category} if it is equipped with a monoidal structure that is compatible with the $\mathsf{Hilb}$-module structure.
\begin{definition}\label{def: fusion category}
A tensor category is called a \emph{fusion category}\footnote{Strictly speaking, we should be calling these `fusion $\mathsf{Hilb}$-categories',
as the term `fusion category' usually refers to categories where each object is a finite direct sum of simples.}
if its underlying $\mathsf{Hilb}$-category is equivalent to $\mathsf{Hilb}^n$ (the category whose objects are $n$-tuples of Hilbert spaces) for some finite $n$,
and if all its irreducible objects are dualizable.
\end{definition}
\nid In other words, a tensor category is fusion if it is semisimple with finitely many simples, and every simple is dualizable.
(Note that the two possible definitions of dualizable, namely the one with and the one without the normalization condition in Definition \ref{def:dual} are equivalent to each other,
as shown in \cite[Thm.~4.12]{BDH(Dualizability+Index-of-subfactors)}.)

We know from Theorem \ref{thm: KLM -- all irreducible sectors are finite} and Lemma~\ref{lem: identification of the dual sector} 
that if $\cala$ is a conformal net with finite index, 
then the category of $S$-sectors of $\cala$ is fusion with respect to 
the monoidal structure~\eqref{eq: associative operation on A-reps}.
By a result of Longo-Xu~\cite{Longo-Xu(dichotomy)}, the converse
also holds.
In this section, we give an alternative proof of this result using
the coordinate-free point of view.  

\begin{theorem}
  \label{thm: characterization of finite conformal nets}
If $\Rep(\cala)$ is a fusion  category, then the conformal net $\cala$ has finite index.
\end{theorem}

The proof is based on the following technical lemma:

\begin{lemma}\label{lem: technical lemma --> finite conformal nets}
Let $A$ and $B$ be von Neumann algebras, and let ${}_AH_B$ and ${}_BK_A$ be bimodules.
Let ${}_B\overline H_A$ be the complex conjugate of ${}_AH_B$.
If ${}_AH_B$ is irreducible and
\begin{list}{$\bullet$}{\leftmargin=.5cm \parsep=3pt}
\item[(1)] ${}_AH\boxtimes_BK_A$ is a direct sum of dualizable $A$-$A$-bimodules,
\item[(2)] ${}_BK\boxtimes_AH_B$ is a direct sum of dualizable $B$-$B$-bimodules,
\item[(3)] ${}_AH\boxtimes_BK\boxtimes_A\overline K_B$ is a direct sum of irreducible $A$-$B$-bimodules,
\item[(4)] ${}_A\overline K\boxtimes_BK\boxtimes_AH_B$ is a direct sum of irreducible $A$-$B$-bimodules,
\end{list}\smallskip
(where those direct sums are possibly infinite) then ${}_AH_B$ is a dualizable bimodule.
\end{lemma}

\begin{proof}
Pick dualizable bimodules ${}_A(M_i)_A$ and maps $f_i:M_i\to H\boxtimes_BK$ so that $F:=\bigoplus f_i:\bigoplus M_i\to H\boxtimes_BK$ is an isomorphism.
Similarly, pick dualizable bimodules ${}_B(N_j)_B$ and maps $g_j: K\boxtimes_AH\to N_j$ so that $G:=\bigoplus g_j: K\boxtimes_AH\to\bigoplus N_j$ is an isomorphism.
The composite
\[
K\boxtimes_A \textstyle\big(\bigoplus M_i\big) \xrightarrow{1\otimes F} 
K \boxtimes_AH\boxtimes_BK \xrightarrow{G \otimes 1}
\big(\bigoplus N_j\big)\boxtimes _B K
\]
being an isomorphism, one can chose indices $i$ and $j$ so that
\[
K\boxtimes_A M_i \xrightarrow{1\otimes f_i} 
K \boxtimes_AH\boxtimes_BK \xrightarrow{g_j \otimes 1}
N_j\boxtimes _B K
\]
is non-zero.
Write $M$ for $M_i$, and $N$ for $N_j$.
By duality, the composite
\[
\begin{split}
\overline N\boxtimes _B K \xrightarrow{1\otimes1\otimes \mathit{coev}_{M}}&\,
\overline N\boxtimes _B K \boxtimes_A M\boxtimes_A \overline M\\\xrightarrow{1\otimes1\otimes f\otimes 1}\,&\,
\overline N\boxtimes _B K \boxtimes_AH\boxtimes_BK\boxtimes_A \overline M\\\xrightarrow{1\otimes g\otimes 1\otimes 1}\,&\,
\overline N\boxtimes _B N \boxtimes_B K\boxtimes_A \overline M\xrightarrow{\mathit{ev}_N\otimes 1} K\boxtimes_A \overline M
\end{split}
\]
is also non-zero.

The bimodules $\overline N\boxtimes _B K$ and $K\boxtimes_A \overline M$ are direct summands of 
$\overline H\boxtimes_A\overline K\boxtimes_BK$ and $K\boxtimes_A\overline K\boxtimes_B\overline H$, respectively.
By assumption, they can therefore be written as direct sums of irreducible bimodules:
\[
\textstyle \overline N\boxtimes _B K\cong \bigoplus P_\alpha\qquad\qquad K\boxtimes_A \overline M\cong\bigoplus Q_\beta\,.
\]
Pick $\alpha$ and $\beta$ so that the map
\begin{equation}\label{eq: From P to Q}
\begin{split}
P_\alpha\,\hookrightarrow\,
\overline N\boxtimes _B K \xrightarrow{1\otimes1\otimes \mathit{coev}_{M}}&\,
\overline N\boxtimes K \boxtimes M\boxtimes \overline M\\\xrightarrow{1\otimes1\otimes f\otimes 1}\,&\,
\overline N\boxtimes K \boxtimes H\boxtimes K\boxtimes \overline M\\\xrightarrow{1\otimes g\otimes 1\otimes 1}\,&\,
\overline N\boxtimes N \boxtimes K\boxtimes \overline M\xrightarrow{\mathit{ev}_N\otimes 1} K\boxtimes_A \overline M
\,\twoheadrightarrow\, Q_\beta
\end{split}
\end{equation}
is non-zero.
Since $P_\alpha$ and $Q_\beta$ are irreducible, that map is actually an isomorphism.
Use \eqref{eq: From P to Q} to identify $P_\alpha$ and $Q_\beta$, and call it simply $P$.
The maps
\[
\mathit{Ev}\,:\,\,P\boxtimes_A H\,\hookrightarrow\, \overline N\boxtimes _B K \boxtimes _A H
\xrightarrow{1\otimes g}
\overline N\boxtimes _B N \xrightarrow{\mathit{ev}_N} L^2B
\]
\[
\mathit{Coev}:\,
L^2A \xrightarrow{\mathit{coev}_M} M\boxtimes _A \overline M 
\xrightarrow{f\otimes 1}H\boxtimes_B K\boxtimes _A \overline M\twoheadrightarrow H\boxtimes_B P
\]
then satisfy $(\mathit{Ev}\boxtimes 1_P)\circ(1_P\boxtimes \mathit{Coev})=1_P$.
Since $H$ is irreducible, there is some $\lambda\in \IC$ for which
\[
(1_H\boxtimes \mathit{Ev})\circ(\mathit{Coev}\boxtimes 1_H)=\lambda 1_H.
\]
Moreover, by evaluating $(\mathit{Ev}\boxtimes\mathit{Ev}\boxtimes 1_P)\circ(1_P\boxtimes \mathit{Coev}\boxtimes \mathit{Coev})$
in two different ways, one can see that $\lambda=1$, and so $H$ is dualizable.
\end{proof}

\begin{proof}[Proof of Theorem \ref{thm: characterization of finite conformal nets}]
Let $\Delta$ be the set of isomorphism classes of $\cala$-sectors.
Since by assumption the category $\Rep(\cala)$ is semisimple, any object $H$ can be decomposed as
\begin{equation}\label{eq: sum decomposition of arbitrary sector}
H \,\,\cong\,\, \bigoplus_{\lambda\in \Delta}\, H_\lambda\otimes \mathrm{Hom}_{\Rep(\cala)}\big(H_\lambda,H\big),
\end{equation}
where the multiplicity spaces $\mathrm{Hom}_{\Rep(\cala)}(H_\lambda,H)$ are Hilbert spaces, and the tensor product is the completed tensor product of Hilbert spaces.

We need to show that
\[
H\,\,:=\,\,{}_{\cala(I_1\cup I_3)}H_0(S)_{\cala(I_2\cup I_4)^\op}
\]
is a dualizable bimodule,
where $S$ and $I_1, \ldots, I_4$ are as in \eqref{picture: circle with four intervals}.
For that, we verify the assumptions in Lemma~\ref{lem: technical lemma --> finite conformal nets} for the bimodules $H$ and $K:={}_{\cala(I_2\cup I_4)^\op}H_0(\bar S)_{\cala(I_1\cup I_3)}
$.
We only check the first and third conditions of the lemma, as the other two are entirely similar.
Following the notation of the lemma, we let $A:=\cala(I_1\cup I_3)$ and $B:= \cala(I_2\cup I_4)^\op$.

(1)\,\, The fusion ${}_AH\boxtimes_B K_A$ is the Hilbert space $H_{\mathit{ann}}$ studied in Section \ref{sec: Hilbert space for annulus}.
It is an $S_1$-$S_3$-sector for the circles $S_1:=I_1\cup_{\partial I_1}\bar I_1$ and $S_3:=I_3\cup_{\partial I_3}\bar I_3$.
By applying \eqref{eq: sum decomposition of arbitrary sector} to $H_{\mathit{ann}}$, viewed as an $S_1$-sector,
we get
\[
H_{\mathit{ann}} \,\cong\, \bigoplus_{\lambda\in \Delta} H_\lambda(S_1)\otimes \mathrm{Hom}\big(H_\lambda(S_1),H_{\mathit{ann}}\big).
\]
The multiplicity space $\mathrm{Hom}(H_\lambda(S_1),H_{\mathit{ann}})$ is itself an $S_3$-sector, and so we can apply \eqref{eq: sum decomposition of arbitrary sector}
once more to write
\[
H_{\mathit{ann}} \,\cong\, \bigoplus_{\lambda,\mu\in \Delta} H_\lambda(S_1)\otimes H_\mu(S_3)\otimes V_{\lambda\mu},
\]
where $V_{\lambda\mu}$ is a multiplicity space.  The $A$-$A$-bimodule $H_\lambda(S_1)\otimes H_\mu(S_3)$ is a tensor product of
${}_{\cala(I_1)}H_\lambda(S_1)_{\cala(\bar I_1)^\op}$ and ${}_{\cala(I_3)}H_\mu(S_3)_{\cala(\bar I_3)^\op}$, which are dualizable by assumption.
We have thus verified the first condition of Lemma~\ref{lem: technical lemma --> finite conformal nets}.

(3)\,\, Observe that $H\boxtimes_BK\boxtimes_A\overline K$ is the Hilbert space $V(\Sigma)$ associated to the following surface (with the indicated cell decomposition)
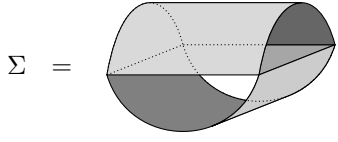
\[
\Sigma \,\,\,\,=\quad \tikzmath{
\coordinate (a) at (2,0);\coordinate (b) at (1,.4);\coordinate (c) at ($(a)+(b)$);

\filldraw[fill=gray] (0,0) .. controls (.3,-1) and (1.7,-1) .. (a) -- (c) ..controls ($(1.7,-1)+(b)$) and ($(.3,-1)+(b)$) .. (b) --cycle;
\filldraw[fill=black!35, line join = bevel] (a) -- (c) -- (b) -- cycle;
\filldraw[fill=black!60, line join = bevel] (a) .. controls ($(.2,1.2)+(a)$) and ($(.8,1.2)+(a)$) .. (c) -- (b) -- cycle;

\draw (.56,.958) -- +(a);\draw (1.38,-.68) -- +(b);
\filldraw[line join = bevel, fill=gray!40] (1.38,-.68) .. controls (1.71,-.55) and (1.91,-.3) .. (a) -- (c) ..controls ($(1.91,-.3)+(b)$) and ($(1.71,-.55)+(b)$) .. ($(1.38,-.68)+(b)$) -- cycle;

\filldraw [fill=gray!30] (.56,.958) .. controls (.35,.95) and (.12,.7) .. (0,0) -- (a) ..controls ($(.12,.7)+(a)$) and ($(.35,.95)+(a)$) .. ($(.56,.958)+(a)$) -- cycle;
\draw [densely dotted] (0,0) -- (b) -- (2.09,.4);
\draw [densely dotted] (0,0) .. controls (.2,1.2) and (.8,1.2) .. (b)
(c) ..controls ($(1.7,-1)+(b)$) and ($(.3,-1)+(b)$) .. (b);
}
\]
with boundary $\partial \Sigma = S$.
By  Corollary \ref{cor: it's a partialSigma-sector}, it is not only an $A$-$B$-bimodule, but also an $S$-sector.
By strong additivity, the forgetful functor from $S$-sectors to $A$-$B$-bimodules is fully faithful.
Therefore, a subspace of $H\boxtimes_BK\boxtimes_A\overline K$ is an irreducible sub-$S$-sector if and only if it is an irreducible sub-$A$-$B$-bimodule.
Since $\Rep_S(\cala)$ is semisimple by assumption, every $S$-sector can be written as a direct sum of irreducible $S$-sectors.
Decomposing $H\boxtimes_BK\boxtimes_A\overline K$ as a direct sum of irreducible $S$-sectors then also provides a decomposition into irreducible $A$-$B$-bimodules.
This verifies the third condition of Lemma~\ref{lem: technical lemma --> finite conformal nets}.
\end{proof}

%==================================================================

\section{Comparing conformal and positive-energy  nets}
  \label{sec:pos-energy-nets}

\subsection{Circle-based nets}
     \label{subsec:circle-nets}

In this section, we present an alternative version of the definition of conformal nets that will provide an intermediary between our notion of coordinate-free conformal nets and existing notions of conformal nets in the 
literature \cite{Gabbiani-Froehlich(OperatorAlg-CFT), Longo(Lectures-on-Nets-II)}. 
Recall that $S^1:=\{z\in\mathbb C:|z|=1\}$ denotes the standard circle, and $S^1_\top := \{ z \in S^1 : \Im\mathrm{m} (z) \geq 0 \}$ its upper half.
Let $\INT_{S^1}$ be the poset of subintervals of $S^1$.
Given a Hilbert space $H$ (always assumed separable), we write $\VN_H$ for the poset of von Neumann subalgebras of $\bfB(H)$.
Recall that given an interval $I\subset S^1$, we denote by $I'$ the closure of its complement.

A \emph{net on the circle} $(\bfA,H)$ is a Hilbert space $H$ equipped with a continuous projective action 
$\Diff(S^1) \to \PU_\pm(H)$, $\varphi \mapsto [u_\varphi]$ and an order preserving map
\begin{equation*}
\INT_{S^1}\to\, \VN_H\,,\,\,\,\,\, I \,\mapsto \bfA(I).
\end{equation*} 
It is required that $u_\varphi$ be complex linear if $\varphi$ is orientation-preserving and complex antilinear if $\varphi$ is orientation-reversing,
where $u_\varphi$ is any representative of $[u_\varphi]$.
Moreover, the continuity condition refers to the $\calc^\infty$ topology on $\Diff(S^1)$, and the quotient of the strong topology on $\U_\pm(H)$.
     
\begin{definition}
\label{def:conformal-net-circle}
A \emph{conformal net on the circle} is a net on the circle $(\bfA,H)$, and a unitary isomorphism $v:H \xrightarrow{\scriptscriptstyle\cong} L^2(\bfA(S^1_\top))$, subject to the following conditions.
Here, $I$, $K$, and $L$ will denote subintervals of $S^1$. 
\begin{enumerate}
\item\emph{Locality:}
If $I$ and $K$ have disjoint interiors, then $\bfA(I)$ and $\bfA(K)$ are commuting subalgebras of $\bfB(H)$.
\item \emph{Strong additivity:}
If $L = I \cup K$, then $\bfA(L) = \bfA(I) \vee \bfA(K)$.
\item \emph{Split property:}
If $I$ and $K$ are disjoint, then the ultraweak closure of the algebraic tensor product $\bfA(I) \ox_{\alg} \bfA(K)\subset\bfB(H)$ is the spatial tensor product $\bfA(I)\,\bar{\ox}\,\bfA(K)$.
\item \label{def:conformal-net-circle:cov}
\emph{Covariance:}
For $\varphi \in \Diff(S^1)$, we have $u_\varphi \bfA(I) u_\varphi^* = \bfA(\varphi(I))$.
If $\varphi \in \Diff(S^1)$ restricts to the identity on $I'$, then $u_\varphi \in \bfA(I)$. 
\item \label{def:conformal-net-circle:vacuum} 
\emph{Vacuum:}  
The map $v$ intertwines the $\bfA(S^1_\top)$-module structures on $H$ and on $L^2(\bfA(S^1_\top))$.
Moreover, letting $j:S^1\to S^1$ be complex conjugation, and
$J$ the modular conjugation on $L^2(\bfA(S^1_\top))$,
then we have $u_j = v^* J\, v$ in $\PU_-(H)$.
\end{enumerate}
\end{definition}

\noindent We remind the reader that intervals are by definition closed. 
In particular, there is a gap between the disjoint intervals
appearing in the preceding definition.

\begin{construction}[coordinate-free to circle-based conformal nets] \label{constr:nets-->nets-on-circle}
Let $\cala$ be an irreducible conformal net (Definition~\ref{def:conformal-net}).
By Theorem~\ref{thm: Vaccum Sector}, there is a canonical vacuum sector $H_0(S^1,\cala)$ associated to the standard circle,
and it is equipped with a continuous projective action $\varphi \mapsto [u_\varphi]$ of $\Diff(S^1)$ such that $u_\varphi$ implements $\varphi$. 
We therefore obtain a net on the circle by setting $H := H_0(S^1,\cala)$, and defining $\bfA(I)$ to be the image of $\cala(I)$ under its action on $H_0(S^1,\cala)$. 
\end{construction}

\begin{proposition}  \label{prop:nets-->nets-on-circle}
Let $\cala$ be an irreducible conformal net (Definition~\ref{def:conformal-net}).
Then the above construction produces a conformal net on the circle (Definition~\ref{def:conformal-net-circle}).
\end{proposition}

\begin{proof}
The locality, strong additivity, and split property axioms
of Definition \ref{def:conformal-net-circle} follow immediately from the corresponding axioms of Definition \ref{def:conformal-net}.
We have $u_\varphi \bfA(I) u_\varphi^* = \bfA(\varphi(I))$ because $u_\varphi$ implements $\varphi$.
If $\varphi$ restricts to the identity on $I'$, then since $u_\varphi$ implements $\varphi$, $u_\varphi$ commutes with $\bfA(I')$, and so $u_\varphi \in \bfA(I')' = \bfA(I)$ by Haag duality, Proposition~\ref{prop: [Haag duality]}.
The isomorphism $v:H \to L^2(\bfA(S^1_\top))$ is the map $v_{S^1_\top}$ from part {\it (iii)} of Theorem~\ref{thm: Vaccum Sector}.
It is a morphism of $\bfA(S^1_\top)$-modules, and we have $u_j=H_0(j,\cala)=v^*J\,v$, by parts {\it (ii)} and {\it (iv)} of Theorem~\ref{thm: Vaccum Sector}.
\end{proof}

\begin{construction}[circle-based to coordinate-free conformal nets]  \label{constr:circle-net-->net}
Let now $(\bfA,H)$ be a conformal net on the circle.
Given an abstract interval $I \in \INT$, we let $\cale_I$ be the category whose objects are smooth embeddings of $I$ into $S^1$, and in which every hom-set contains exactly one element.
We define a functor 
\[
\bfA_I \,\colon\, \cale_I \to\, \VN
\]
as follows. At the level of objects, it is given by $\bfA_I(\iota):=\bfA(\iota(I))$.

To define $\bfA_I$ on morphisms, we need to provide an \text{(anti-)}isomorphism $\bfA(\iota(I))\to \bfA(\iota'(I))$ for every pair of elements $\iota, \iota' \in \cale_I$.
Given $\iota, \iota' \in \cale_I$, pick an extension $\varphi \in \Diff(S^1)$ of $\iota' \circ \iota^{-1} \colon \iota(I) \to \iota'(I)$.
If $\varphi$ is orientation-preserving then $a \mapsto u_\varphi a u_\varphi^*$ defines an isomorphism $\bfA(\iota(I)) \to \bfA(\iota'(I))$,
and if $\varphi$ is orientation-reversing then $a \mapsto u_\varphi a^* u_\varphi^*$ defines an anti-isomorphism $\bfA(\iota(I)) \to \bfA(\iota'(I))$.
If $\tilde\varphi$ is another extension of $\iota' \circ \iota^{-1}$, then $\varphi^{-1} \tilde\varphi$ is the identity on $\iota(I)$,
and so $\Ad(u_{\varphi^{-1} \tilde\varphi})$ is the identity on $\bfA(\iota(I))$ by the covariance and locality axioms.
Therefore, the \text{(anti-)}isomorphism $\bfA(\iota(I)) \to \bfA(\iota'(I))$ defined above is independent of the choice of extension $\varphi$.

If $\iota'' \colon I \to S^1$ is a third embedding, one checks, by picking compatible
extensions $\varphi$, $\varphi'$, $\varphi''$ of $\iota'' \circ \iota^{-1}$, $\iota'' \circ \iota'{}^{-1}$, and $\iota' \circ \iota^{-1}$, that
$\Ad(u_{\varphi}) = \Ad(u_{\varphi'}) \circ \Ad(u_{\varphi''})$.
The above prescription therefore defines a functor $\bfA_I$, and it makes sense to set
\begin{equation*}
\textstyle \cala(I) := \lim_{\iota \in \cale_I} \bfA_I(\iota) = \lim_{\iota \in \cale_I} \bfA(\iota(I)).
\end{equation*}
By definition, an element $a\in \cala(I)$ is a collection of operators $\{ a_\iota \in \bfA(\iota(I)) \}_{\iota \in \cale_I}$ subject to 
the conditions that $u_\varphi a_\iota {u_\varphi}^* = a_{\varphi \circ \iota}$ for each embedding $\iota:I\to S^1$ and $\varphi \in \Diff_+(S^1)$,
and that $u_\varphi {a_\iota}^* {u_\varphi}^* = a_{\varphi \circ \iota}$ for each $\iota$ and $\varphi \in \Diff_-(S^1)$.

Given a smooth embedding of abstract intervals $f: I \to J$, there is an induced (anti-)homomorphism $\cala(f):\cala(I)\to \cala(J)$ that sends
$a=\{a_\iota \in \bfA(\iota(I)) \}_{\iota \in \cale_I}$ to the unique element $b=\{b_\jmath \in \bfA(\jmath(J)) \}_{\jmath \in \cale_J}$
that satisfies 
$b_\jmath = u_\varphi a_\iota {u_\varphi}^*$ (or $b_\jmath = u_\varphi a_\iota^* {u_\varphi}^*$) for every $\varphi \in \Diff_+(S^1)$ (respectively $\varphi \in \Diff_-(S^1)$) with $\jmath \circ f = \varphi \circ \iota$.
This defines a functor $\cala \colon \INT \to \VN$
that sends orientation-preserving embeddings to homomorphisms, and orientation-reversing embeddings to anti-homomorphisms.
\end{construction}

\begin{proposition} \label{prop:nets-on-circle-->nets}
Let $(\bfA,H)$ be a conformal net on the circle (Definition~\ref{def:conformal-net-circle}),
then the functor $\cala \colon \INT \to \VN$ given by the above construction is a coordinate-free conformal net (Definition~\ref{def:conformal-net}). 
\end{proposition}

\begin{proof}
The locality, strong additivity, and split property axioms for $(\bfA,H)$ imply the corresponding axioms for $\cala$.
For the remaining requirements, we argue as follows.

\emph{Covariance.}
Let $\varphi\in\Diff_+(I)$ be a diffeomorphism that restricts to the identity on a neighborhood of $\partial I$.
We assume without loss of generality that $I$ is contained in the standard circle, and let $\hat \varphi\in\Diff_+(S^1)$ be the extension of $\varphi$ by the identity map on $I'$.
The map $\cala(\varphi):\cala(I)\to \cala(I)$ is given by 
$\{a_\iota\}_{\iota\in\cale_I} \mapsto \{b_\jmath\}_{\jmath \in \cale_I}$,
where the $b_\jmath$ are determined by the requirement that
$b_\jmath = u_\psi\, a_\iota {u_\psi}^*$
for every $\psi \in \Diff_+(S^1)$ with $\jmath \circ \varphi = \psi \circ \iota$.
Letting $\iota=\jmath=\id_I$ and $\psi=\hat\varphi$, we learn that $b_{\id_I}=u_{\hat\varphi}a_{\id_I}u_{\hat\varphi}^*$.
Under the identification $\{a_\iota\} \mapsto a_{\id_I}$ of $\cala(I)$ with $\bfA(I)$, the map $\cala(\varphi):\cala(I)\to \cala(I)$
therefore corresponds to $\Ad(u_{\hat \varphi}):\bfA(I)\to \bfA(I)$.
Since $u_{\hat \varphi}\in\bfA(I)$, it follows that $\cala(\varphi)$ is an inner automorphism.

\emph{Continuity.}
Let $I$ be an interval, which we take, without loss of generality, to be a subinterval of $S^1$.
We identify $\cala(I)$ with $\bfA(I)$ by $a = \{a_\iota\} \mapsto a_{\id_I}$.
By Lemma~\ref{lem: equiv def of continuity axiom} below, it suffices to check the continuity of the map $\Diff_+(I)\to \Aut(\bfA(I))$.
Letting $\Diff_+(S^1,I) := \{ \varphi \in \Diff_+(S^1) \mid \varphi(I) = I \}$ and
$\PN(\bfA(I)) := \{ U \in \U(H) \mid U \bfA(I) U^* = \bfA(I) \}/S^1$,
we have the following commutative diagram
\begin{equation*}
\tikzmath{
\matrix [matrix of math nodes,column sep=.9cm,row sep=8mm]
{ |(a)| \Diff_+(S^1) \pgfmatrixnextcell |(b)| \Diff_+(S^1,I)  \pgfmatrixnextcell |(c)| \Diff_+(I)\\ 
|(d)| \PU(H) \pgfmatrixnextcell |(e)| \PN(\bfA(I)) \pgfmatrixnextcell |(f)| \Aut(\bfA(I))\\ }; 
\draw[->] (b) -- (a); \draw[->>] (b) -- (c);
\draw[->] (e) -- (d); \draw[->] (e) --node[above]{$\scriptstyle \Ad$} (f);
\draw[->] (a) --node[right]{$\scriptstyle u$} (d);
\draw[->] (b) --node[right]{$\scriptstyle u|_{\Diff_+(S^1,I)}$} (e);
\draw[->] (c) -- (f);
}
\end{equation*}
where the existence of the middle vertical map is guaranteed by the covariance axiom.
The two horizontal maps on the left are subgroup inclusions, equipped with subspace topologies.
The map $\Ad$ is continuous by Lemma~\ref{lem:ad-is-cont}, and the map $u$ is continuous by assumption.
The middle vertical is continuous by restriction.
The vertical map on the right is $\varphi \mapsto\cala(\varphi)$, and we have to show that it is continuous.
The $\calc^\infty$ topology on $\Diff(I)$ coincides with the quotient topology under the map $\Diff_+(S^1,I) \twoheadrightarrow \Diff(I)$.
The continuity of $\varphi \mapsto \cala(\varphi)$ therefore follows from that of $\Ad \circ\hspace{.4mm} u|_{\Diff_+(S^1,I)}$.

\emph{Vacuum sector.}
Let $K \subsetneq I$ contain the boundary point $p \in \dd I$,
and let $K \cup_p \bar{K}$ be equipped with any smooth structure that extends the ones on $K$ and $\bar K$ and for which the orientation-reversing involution $j_K$ that exchanges $K$ and $\bar{K}$ is smooth.
We have to show that the following diagram can be completed:
\[
\tikzmath{
\matrix [matrix of math nodes,column sep=.9cm,row sep=5mm]
{ 
|(a)| \cala(K) \ox_{\alg} \cala(\bar{K}) \pgfmatrixnextcell |[xshift=70](b)| \cala(K) \ox_{\alg} \cala(K)^\op \pgfmatrixnextcell |(b')| \cala(I) \ox_{\alg} \cala(I)^\op\\ 
|(c)| \cala(K \cup_p \bar{K}) \pgfmatrixnextcell\pgfmatrixnextcell |(d)| \bfB(L^2\cala(I)). \\
}; 
\draw[->] (a) --node[above]{$\scriptstyle (\mathrm{id},\cala(j_K))$} (b);
\draw[->] (a) -- (c);
\draw[->] (b) -- (b');
\draw[->] (b') -- (d);
\draw[->,dashed] (c) -- (d);
}
\]

Extend the smooth structure on $K \cup_p \bar{K}$ to one on $S:=I\cup_{\partial I}\bar I$
(make sure that the involution that exchanges $I$ and $\bar{I}$ is smooth),
and pick an orientation-preserving  diffeomorphism $\varphi:S\to S^1$ that intertwines the above involution with the standard involution $j$ on $S^1$ and
that sends $I$ to the upper half $S^1_\top$ of the circle.
We then have isomorphisms 
\begin{alignat*}{3}
\cala(K)&\to\bfA(\varphi(K))
&\cala(\bar K)&\to\bfA(\varphi(\bar K))\\
\cala(K\cup\bar K)&\to\bfA(\varphi(K\cup\bar K))\qquad
&\cala(I)&\to\bfA(S^1_\top)
\end{alignat*}
(given by $a\mapsto a_\iota$ where $\iota$ is the appropriate restriction of $\varphi$)
that make the following diagram commute:
\[
\tikzmath{
\node (a1) at (0,3.15) {$\scriptstyle\cala(K) \ox_{\alg} \cala(\bar K)$};\node (b1) at (4.5,3.15) {$\scriptstyle\cala(K) \ox_{\alg} \cala(K)^\op$};
\node (b1') at (8,3.15) {$\scriptstyle\cala(I) \ox_{\alg} \cala(I)^\op$}; \node (c1) at (-1.5,2.15) {$\scriptstyle\cala(K \cup_p \bar{K})$};
\node (d1) at (9.5,2.15) {$\scriptstyle\bfB(L^2\cala(I))$}; \node (a) at (0,1) {$\scriptstyle\bfA(\varphi(K)) \ox_{\alg} \bfA(\varphi(\bar K))$};
\node (b) at (4.5,1) {$\scriptstyle\bfA(\varphi(K)) \ox_{\alg} \bfA(\varphi(K))^\op$}; \node (b') at (8,1) {$\scriptstyle\bfA(S^1_\top) \ox_{\alg} \bfA(S^1_\top)^\op$}; 
\node (c) at (-1.5,0) {$\scriptstyle\bfA(\varphi(K \cup_p \bar{K}))$}; \node (x) at (7.2,0) {$\scriptstyle\bfB(H)$};\node (d) at (9.5,0) {$\scriptstyle\bfB(L^2\bfA(S^1_\top))$}; 
\draw[->] (a1) --node[above]{$\scriptscriptstyle (\mathrm{id},\,\cala(j_K)$} (b1);\draw[->] (a1) -- (c1);\draw[->] (b1) -- (b1');\draw[->] (b1') -- (d1);
\draw[->] (a) --node[above]{$\scriptscriptstyle (\mathrm{id},\,a\mapsto u_ja^*u_j)$} (b);\draw[->] (a) -- (c);\draw[->] (b) -- (b');\draw[->] (b') -- (d);
\draw[->] (x) --node[above]{$\scriptscriptstyle \cong$} (d);\draw[->] (a1) -- (a);\draw[->] (b1) -- (b);\draw[->] (b1') -- (b');\draw[->] (c1) -- (c);\draw[->] (d1) -- (d);
\draw[->,dashed] (c) -- (x);\draw[white, line width=13] (c1) -- (d1);\draw[->,dashed] (c1) -- (d1);
}
\]
Indeed, if 
$a=\{a_\alpha\}\in \cala(\bar K)$
and
$b=\{b_\beta\}\in \cala(K)$ 
are such that $\cala(j_K)(a)=b$,
then $b_\beta=u_\psi a_\alpha^* u_\psi^*$ for every $\psi\in\Diff_-(S^1)$ with $\beta \circ j_K=\psi\circ\alpha$.
Setting $\alpha=\varphi$, $\beta=\varphi$, and $\psi=j$, we get the commutativity of the rear left square:
\(
\tikzmath{
\node[xshift = 9, name=A] at (0,.8) {$a$};
\node[anchor=west, name=X] at (0,0) {$a_{\varphi} \mapsto u_ja_{\varphi}^*u_j=b_{\varphi}$.};
\node[xshift = -12, name=B] at ($(X.east)+(0,.85)$) {$b$};
\draw[->] (A) -- (B.west |-A);
\draw[line cap=round] (A.east) ++(0,.06) -- ++(0,-.12);
\path (A) -- node[yshift=-1, rotate=-90, scale=.95]{$\mapsto$} (X.north-|A)
(B) -- node[yshift=-1, rotate=-90, scale=.95]{$\mapsto$} (X.north-|B);
}
\)

Let us identify $K$, $\bar K$, and $K\cup\bar K$ with their images under $\varphi$ in order to simplify the notation.
The existence of the top dotted arrow is now equivalent to the existence of an arrow completing the following diagram:
\[
\tikzmath{
\node (a) at (0,1.15) {$\bfA(K) \ox_{\alg} \bfA(\bar K)$};
\node (b) at (5,1.15) {$\bfA(K) \ox_{\alg} \bfA(K)^\op$};
\node (b') at (9,1.15) {$\bfA(S^1_\top) \ox_{\alg} \bfA(S^1_\top)^\op$}; 
\node (c) at (0,0) {$\bfA(K \cup_p \bar{K})$};
\node (x) at (6.7,0) {$\bfB(H)$};
\node (d) at (9,0) {$\bfB(L^2\bfA(S^1_\top))$}; 
\draw[->] (a) --node[above]{$\scriptstyle (\mathrm{id},\,a\mapsto u_ja^*u_j)$} (b);
\draw[->] (a) -- (c);
\draw[->] (b) -- (b');
\draw[->] (b') -- (d.north-|b');
\draw[->,dashed] (c) -- (x);
\draw[->] (x) --node[above]{$\scriptstyle \cong$} (d);
}
\]
We claim that the natural action $\bfA(K \cup_p \bar{K})\to\bfB(H)$ provided by the data of a conformal net on the circle
makes the above diagram commute.

On the subalgebra $\bfA(K)$ of $\bfA(K) \ox_{\alg} \bfA(\bar K)$, 
the commutativity of the above diagram is obvious since, by assumption, the isomorphism $v:H \to L^2(\bfA(S^1_\top))$ is equivariant for the left actions of $\bfA(S^1_\top)$.
On the subalgebra $\bfA(\bar K)$, we argue as follows.
Pick $a\in \bfA(\bar K)$, with image $u_j a^* u_j \in \bfA(K)^\op\subset \bfA(S^1_\top)^\op$.
That element goes to $J(u_j a^* u_j)^*J$ under the right action map $\bfA(S^1_\top)^\op\to \bfB(L^2\bfA(S^1_\top))$, where $J$ is the modular conjugation.
By assumption, the isomorphism $\bfB(H)\to \bfB(L^2\bfA(S^1_\top))$ sends $u_j$ to $J$.
Since $J=J^*$, the above expression then simplifies to $J(J a^* J)^*J=a$, which proves the commutativity of the diagram.
\end{proof}

Recall the definition of the continuity of a net from the beginning of Section \ref{sec: Definition of conformal nets} and Haagerup's $u$-topology from the Appendix.

\begin{lemma}\label{lem: equiv def of continuity axiom}
Given a net $\cala$, if the natural maps $\Diff_+(I)\to \mathrm{Aut}(\cala(I))$ are continuous,
then so are the maps $\mathrm{Hom}_{\INT}(I,J)\to \mathrm{Hom}_{\VN}(\cala(I),\cala(J))$.
\end{lemma}
\begin{proof}
It is enough to show continuity on the subset $\mathrm{Hom}_{\INT}^+(I,J)$ of orientation-preserving embeddings.
Given a generalized sequence $\varphi_i\in \mathrm{Hom}_{\INT}^+(I,J)$, indexed by the poset $\mathcal I$, with limit $\varphi$, and given a vector $\xi\in \cala(J)_*$ in the predual,
we need to show that $\cala(\varphi_i)_*(\xi)$ converges to $\cala(\varphi)_*(\xi)$ in $\cala(I)_*$.

Pick an interval $K$, identify $I$ and $J$ with subintervals of $K$ via some fixed embeddings $I\hookrightarrow K$, $J\hookrightarrow K$ into its interior,
and extend $\varphi$ to a diffeomorphism $\hat\varphi \in\Diff_+(K)$.
For each natural number~$n$, pick an extension $\hat\varphi_{n,i}\in \Diff_+(K)$ of $\varphi_i$ such that
$\|\hat\varphi_{n,i}-\hat \varphi\|_{C^n}<\|\varphi_{i}- \varphi\|_{C^n}$,
where $\|\,\,\|_{C^n}$ is any norm that induces the $C^n$ topology on $\Diff_+(K)$.
It follows that
\[
F\text{-lim}\,\, \hat\varphi_{n,i} \,=\, \hat\varphi,
\]
where $F$ is the filter\footnote{Recall that a filter $F$ on a set $\cali$ is a collection of subsets $S\subset \cali$ such that if $S \in F$ and $T \in F$ then $S \cap T \in F$, and if $S\in F$ and $S\subset S'$ then $S'\in F$.
The expression $F\text{-lim}\,\,x_i=x$ means that for every neighborhood $U$ of $x$, the set $\{i\in \cali\,|\,x_i\in U\}$ is an element of $F$.} on $\mathbb N\times \mathcal I$ generated by the sets $\{(n,i)\in \IN\times \mathcal I\,|\,n\ge n_0, i\ge i_0(n)\}_{n_0 \in \IN, i_0: \IN \rightarrow \mathcal I}$.
Given a lift $\hat\xi\in\cala(K)_*$ of $\xi$ then, by assumption, $F\text{-lim}\,\,\cala(\hat\varphi_{n,i})_*(\hat\xi)=\cala(\hat\varphi)_*(\hat\xi)$.
Composing with the projection $\pi:\cala(K)_*\twoheadrightarrow \cala(I)_*$,
it follows that $\cala(\varphi_i)_*(\xi)=\pi(\cala(\hat\varphi_{n,i})_*(\hat\xi))$ converges to $\pi(\cala(\hat\varphi)_*(\hat\xi))=\cala(\varphi)_*(\xi)$.
\end{proof}

\subsection{Positive-energy nets}    \label{subsec:pos-energy-nets}

The following notion, that we call ``positive-energy net'', corresponds to what most people would call ``conformal net'' in the literature.
The goal of this section is to show that positive-energy nets subject to the extra conditions of strong additivity, the split property, and diffeomorphism covariance,
yield examples of conformal nets in the sense of Definition~\ref{def:conformal-net}.

\begin{definition}
\label{def:positive-energy-nets} 
A \emph{positive-energy net} $(\bfA,H)$ is a Hilbert space $H$, a unit vector $\Omega\in H$ called the \emph{vacuum vector}, a continuous action $\Conf_+(S^1) \to \U(H)$, $\varphi \mapsto u_\varphi$,
and an order preserving map
\begin{equation*}
\INT_{S^1}\to\, \VN_H\,,\,\,\,\,\, I \,\mapsto \bfA(I)
\end{equation*} 
subject to the following conditions.
Here, $I$ and $J$ are subintervals of $S^1$:
\begin{enumerate}
\item\emph{Locality:}
If $I$ and  $J$ have disjoint interiors, then $\bfA(I)$ and $\bfA(J)$ are commuting subalgebras of $\bfB(H)$.
\item \label{def:PEN:cov} \emph{Covariance:}
For $\varphi \in \Conf_+(S^1)$, we have $u_\varphi \bfA(I) u_\varphi^* = \bfA(\varphi(I))$.
\item \label{def:PEN:vacuum} \emph{Vacuum vector:}
The subspace of $\Conf_+(S^1)$-invariant vectors of $H$ is spanned by $\Omega$.
Moreover, $\Omega$ is a cyclic vector for the action of the algebra $\bigvee_{I \in \INT_{S^1}} \bfA(I)$.
\item \label{def:PEN:positive-energy}
\emph{Positive-energy:}
Let $L_0$ be the conformal Hamiltonian, 
defined by the equation $u_{r_\alpha} = e^{i\alpha L_0}$, where $r_\alpha \in \Conf_+(S^1)$ is the anticlockwise rotation by angle $\alpha$.
Then $L_0$ is a positive operator. 
\end{enumerate}
\end{definition}

There are three further conditions that may be imposed on a positive-energy net:

1. A positive-energy net satisfies \emph{strong additivity} if $\bfA(I\cup J) = \bfA(I) \vee \bfA(J)$.
Note that if the interiors of $I$ and $J$ have non-empty intersection, then that condition is automatic \cite{Fredenhagen-Martin(pointlike-localized-fields)}.

2. A positive-energy net satisfies the \emph{split property} if for any pair of disjoint subintervals
$I, J \subseteq S^1$, the closure of the algebraic tensor product $\bfA(I) \ox_{\alg} \bfA(J)\subset \bfB(H)$ is the spatial tensor product $\bfA(I)\,\bar{\ox}\,\bfA(J)$.

3. A positive-energy net is \emph{diffeomorphism covariant} if $u \colon \Conf_+(S^1) \to \U(H)$ extends to a continuous projective action of the orientation-preserving diffeomorphisms 
$\Diff_+(S^1) \to \PU(H)$, $\varphi \mapsto [u_\varphi]$, such that
\begin{enumerate}
\item[({\it i}\hspace{.4mm})] $u_\varphi \bfA(I) u_\varphi^* = \bfA(\varphi(I))$
\item[({\it ii}\hspace{.4mm})] if $\varphi$ has support in $I$ (i.e., is the identity outside $I$), then $u_\varphi \in \bfA(I)$,
\end{enumerate}
for any lift $u_\varphi\in \U(H)$ of $[u_\varphi]$.
If a positive-energy net is diffeomorphism covariant then, by a result of Carpi and Weiner~\cite[Theorem~5.5]{Carpi-Weiner(2005UniqueDiffSymmetry-in-CFT)},
the extension $\Diff_+(S^1) \to \PU(H)$ is uniquely determined by the above two conditions.
We will use this result to further extend the action to orientation-reversing diffeomorphisms.

Recall that $S^1_\top = \{ z \in S^1 : \Im\mathrm{m} (z) \geq 0 \}$ is the upper half of the standard circle,
and that $j:S^1\to S^1$ denotes complex conjugation.

\begin{proposition}
\label{prop:Vacuum-sector-PENS}
Let $(\bfA,H)$ be a positive-energy net.
Then there is an $\bfA(S^1_\top)$-linear unitary isomorphism $v:H \xrightarrow{\scriptscriptstyle \cong} L^2(\bfA(S^1_\top))$ such that,
letting $\bfJ=v^*J\,v$ with $J$ the modular conjugation, we have
\[
\begin{matrix}
\bfJ \bfA( I ) \bfJ \,=\, \bfA ( j (I) ) &  \forall \, I \subset S^1,  \smallskip\\
\bfJ \,u_\varphi\, \bfJ \,=\, u_{j \circ \varphi \circ j}  &  \forall\,\varphi \in \Conf_+(S^1).
\end{matrix}
\]
\end{proposition}

\begin{proof}
By the Reeh-Schlieder theorem~\cite[Thm.~2.8]{Gabbiani-Froehlich(OperatorAlg-CFT)}, $\Omega$ is cyclic and separating for each algebra $\bfA(I)$, and in particular for $\bfA(S^1_\top)$.
Since $\Omega$ is separating, the corresponding state $\omega\in L^1(\bfA(S^1_\top))$, $\omega(a) = \langle a \Omega , \Omega \rangle_H$, is faithful.
The vector $\omega^{1/2}$ is therefore cyclic in $L^2(\bfA(S^1_\top))$, and so
there is a unique $\bfA(S^1_\top)$-linear isometry $L^2(\bfA(S^1_\top))\to H$ that sends $\omega^{1/2}$ to $\Omega$.
That map is then surjective because $\Omega$ is cyclic for the action of $\bfA(S^1_\top)$ on $H$.

The operator $\bfJ$ is the modular conjugation of Tomita-Takesaki theory for the action of $\bfA(S^1_\top)$ on $H$ with respect to the cyclic and separating vector $\Omega$.
Using a result of Borchers~\cite[Thm.~II.9]{Borchers(CPT-thm-in-2d-theories)}, one can then show that
$\bfJ \bfA( I ) \bfJ = \bfA ( j (I) )$ and $\bfJ u_\varphi \bfJ = u_{j \circ \varphi \circ j}$, see~\cite[Thm.~2.19]{Gabbiani-Froehlich(OperatorAlg-CFT)}.
\end{proof}

Let $j:S^1\to S^1$ and $\bfJ:H\to H$ be as in the previous proposition.

\begin{proposition}
\label{prop:PENS-orientation-reversing-diffs-act}
If a positive-energy net $(\bfA,H)$ is diffeomorphism covariant, then the formula
\begin{equation}
\label{eq:U_varphi*j=U_varphi*J}
\qquad u_{\varphi \circ j} = u_\varphi \circ \bfJ,\qquad\varphi\in \Diff_+(S^1)
\end{equation}
defines an extension of the projective action of $\Diff_+(S^1)$ on $H$ to the group $\Diff(S^1)$ of all diffeomorphisms of $S^1$.
\end{proposition}
   
\begin{proof}
In order to show that~\eqref{eq:U_varphi*j=U_varphi*J} defines a representation, we need to verify that
$u_{j \varphi j} = \bfJ\, u_\varphi \bfJ$ holds up to phase for all $\varphi \in \Diff_+(S^1)$.
Consider the homomoprhism $\Diff_+(S^1) \to \PU(H):\varphi\mapsto [\tilde u_\varphi]$ given by $\tilde u_\varphi = \bfJ u_{j \varphi j} \bfJ$.
We have to show $[\tilde u_\varphi] = [u_\varphi]$ for all $\varphi \in \Diff_+(S^1)$.
For $\varphi \in \Conf_+(S^1)$, this equation holds by Proposition~\ref{prop:Vacuum-sector-PENS}.
In particular, both $u$ and $\tilde u$ are extensions of $u|_{\Conf_+(S^1)}$.
By the uniqueness result of Carpi and Weiner~\cite[Theorem~5.5]{Carpi-Weiner(2005UniqueDiffSymmetry-in-CFT)},
it suffices to check that $\tilde u$ satisfies the same two conditions as $u$:
\begin{enumerate}
\item[({\it i}\hspace{.4mm})] $\tilde u_\varphi \bfA(I) \tilde u_\varphi^* = \bfA(\varphi(I))$
\item[({\it ii}\hspace{.4mm})] if $\varphi$ has support in $I$, then $\tilde u_\varphi \in \bfA(I)$,
\end{enumerate}
For the first condition, we check using Proposition~\ref{prop:Vacuum-sector-PENS} that
\[
\tilde u_\varphi \cala(I) \tilde u_\varphi^* 
= \bfJ u_{j \varphi j} \bfJ \cala(I) \bfJ u_{j \varphi j}^* \bfJ
= \bfJ u_{j \varphi j}\cala(j(I))u_{j \varphi j}^* \bfJ
= \bfJ \cala(j (\varphi(I))) \bfJ 
= \cala(\varphi(I)). 
\]
For the second condition, if $\varphi$ has support in $I$, then $j \varphi j$ has support in $j(I)$.
By the definition of diffeomorphism covariance, it follows that $u_{j \varphi j} \in \cala(j(I))$, and so
$\tilde u_\varphi = \bfJ u_{j \varphi j} \bfJ \in \bfJ \cala(j(I)) \bfJ = \cala(I)$ by Proposition~\ref{prop:Vacuum-sector-PENS}. 
\end{proof}

\begin{proposition}
\label{prop:pos-energy-net--->conf-net}
Every positive-energy net (Definition \ref{def:positive-energy-nets}) that
satisfies strong additivity, the split property, and diffeomorphism covariance
extends to a conformal net on the circle (Definition~\ref{def:conformal-net-circle})
and therefore to a conformal net (Definition~\ref{def:conformal-net}).
Moreover, the resulting conformal net is irreducible.
\end{proposition}

\begin{proof}
We first check that a positive-energy net extends to a conformal net on the circle.
The isomorphism $v:H \to L^2(\bfA(S^1_\top))$ is given by Proposition~\ref{prop:Vacuum-sector-PENS}.
The projective action of $\Diff(S^1)$ on $H$ is provided by Proposition~\ref{prop:PENS-orientation-reversing-diffs-act}.
The equation $u_j=\bfJ$ holds by definition, and
the condition $u_\psi \bfA(I) u_\psi^* = \bfA(\psi(I))$ for $\psi=\varphi j\in\Diff_-(S^1)$ follows from Proposition~\ref{prop:Vacuum-sector-PENS}:
\[
u_\psi \bfA(I) u_\psi^* = u_\varphi \bfJ \bfA(I) \bfJ u_\varphi^* = u_\varphi \bfA(j(I)) u_\varphi^*= \bfA(\varphi (j(I)))= \bfA(\psi(I)).
\]

To get a conformal net from a positive-energy net, apply Proposition~\ref{prop:nets-on-circle-->nets} to the conformal net on the circle associated to the positive-energy net.
Irreducibility follows, using~\cite[Proposition~6.2.9]{Longo(Lectures-on-Nets-II)}, from the uniqueness of the vacuum vector.
\end{proof}

We will later need the following result from the literature:
\begin{theorem}
\label{thm:conf-Hamiltonian+split} 
Let $(\bfA,H)$ be a positive-energy net with conformal Hamiltonian~$L_0$.
If $e^{-\beta L_0}$ is of trace class for all $\beta > 0$, then $(\bfA,H)$ satisfies the split property.
\end{theorem}
\begin{proof}
See~\cite[Thm~7.3.3]{Longo(Lectures-on-Nets-II)} 
or~\cite[Lem.~2.12]{Gabbiani-Froehlich(OperatorAlg-CFT)}.
\end{proof}

\subsection{The loop group conformal nets}
\label{subsec:loop-SU(N)-net}
In this section, we describe, following \cite{Gabbiani-Froehlich(OperatorAlg-CFT)}, the construction of positive-energy nets associated to loop groups.  We verify that these positive-energy nets satisfy strong additivity, the split property, and diffeomorphism covariance, and therefore extend to coordinate-free conformal nets.
There is such a net associated to each compact, simple, simply connected Lie group $G$ equipped with a choice of a positive integer $k\in\mathbb N$ called the \emph{level}.

Let $\mathfrak g$ be the Lie algebra of $G$, $\mathfrak g_\IC$ its complexification, and $\mathfrak h\subset \mathfrak g_\IC$ a Cartan subalgebra.
When dealing with loop groups, it is customary \cite{Pressley-Segal(Loop-groups)} to equip $\mathfrak g$ with the negative definite $G$-invariant inner product $\langle\cdot,\cdot\rangle$
such that every short coroot $\theta \in \mathfrak h$ has $\langle \theta, \theta\rangle=2$---this is the so-called basic inner product.
With respect to the dual inner product, a long root $\alpha \in \mathfrak h^*$ then satisfies $\langle \alpha, \alpha\rangle=2$.

Let $L\mathfrak g:=\calc^\infty(S^1,\mathfrak g)$ be the Lie algebra of functions on $S^1$ with values in $\mathfrak g$, under the pointwise Lie bracket operation.
Let $c$ be the $2$-cocycle on $L\mathfrak g$ given by
\begin{equation}\label{eq: cocycle for Lg}
c(f,g)=\frac1{2\pi}\int_{S^1} \langle f,dg\rangle
\end{equation}
and let $\widetilde{L\mathfrak g}$ be the central extension of $L\mathfrak g$ by $\IR$ that corresponds to that cocycle.
Finally, let $\omega$ be the left invariant closed 2-form on the loop group $LG:=\calc^\infty(S^1,G)$ whose value at the origin is given by $c$.
The integral of $\omega$ against a generator of $H_2(LG,\IZ)\cong \IZ$ is $2\pi$,
and so there is a principal bundle with connection $\IR/2\pi\IZ\to P\to LG$ whose curvature is $\omega$.
The group of connection-preserving automorphisms of $P$ that cover left translations is 
the simply connected Lie group that integrates the Lie algebra $\widetilde{L\mathfrak g}$.
It is a central extension
\[
\IS^1\,\to\, \widetilde{LG} \,\to\, LG
\]
of $LG$ by the abelian group $\IS^1:=\IR/2\pi\IZ$, and it is the universal central extension of $LG$ inside the category of Fr\'echet Lie groups.

To construct the appropriate Hilbert space representations of $\widetilde{LG}$, one starts at the Lie algebra level.
The Lie algebra $L\mathfrak g$ has a dense subalgebra $L \mathfrak g_{\mathit{pol}}$
given by Laurent polynomial functions $\IC^\times\to \mathfrak g_\IC$ whose restriction to $S^1\subset \IC^\times$ take values in $\mathfrak g$.
Its complexification $L \mathfrak g_{\mathit{pol}}\otimes_\IR \IC=\mathfrak g_\IC[z,z^{-1}]$ is the algebra of all polynomial functions on $\IC^\times$ with values in $\mathfrak g_\IC$.
We denote by $\hat {\mathfrak g}$ the central extension of $L \mathfrak g_{\mathit{pol}}$ given by the same cocycle \eqref{eq: cocycle for Lg},
and by $\hat {\mathfrak g}_\IC=\hat {\mathfrak g}\otimes_\IR \IC$ the corresponding central extension of $\mathfrak g_\IC[z,z^{-1}]$ by $\IC$.
Finally, we let $\hat {\mathfrak g}_+\subset \hat {\mathfrak g}_\IC$ be the restriction of that last central extension to $\mathfrak g_\IC[z]\subset \mathfrak g_\IC[z,z^{-1}]$.
Note that the cocycle $c$ is trivial on $\mathfrak g_\IC[z]$, and so $\hat {\mathfrak g}_+$
splits as a direct sum of Lie algebras: $\hat {\mathfrak g}_+ \cong \mathfrak g_\IC[z]\oplus \IC$.

Let $\IC_{0,k}$ be the one-dimensional $\hat {\mathfrak g}_+$-module in which the first summand $\mathfrak g_\IC[z]$ acts by zero,
and the second summand $\IC$ acts by $x\mapsto kix$ (the derivative of the $k$th irreducible representation of $\IS^1$).
We then consider the induced $\hat {\mathfrak g}_\IC$-module $W_{0,k}:=U\hat {\mathfrak g}_\IC\otimes_{U\hat {\mathfrak g}_+} \IC_{0,k}$,
and let $L_{0,k}:=W_{0,k}/J$ be the quotient by its unique maximal proper submodule.
The module $L_{0,k}$ can be equipped \cite[Chapt. 11]{Kac(Inf-dim-Lie-alg)} with a positive definite $\hat {\mathfrak g}$-invariant inner product.
We denote by $H_{0,k}$ its Hilbert space completion.
The action of $\hat {\mathfrak g}$ on $L_{0,k}$ extends to an action of $\widetilde{L\mathfrak g}$ on $H_{0,k}$ by unbounded skew-adjoint operators,
and the latter can then be integrated \cite{Goodman-Wallach(Structure+cocycle-rep-loop-groups), Toledano(Integrating-unitary-rep-infinite-dim)} to a continuous unitary representation
\begin{equation}\label{eq:rhoLG}
\pi:\widetilde{LG}\,\,\longrightarrow\,\,\U(H_{0,k}).
\end{equation}
Moreover, by~\cite[Thm.~6.7]{Goodman-Wallach(Structure+cocycle-rep-loop-groups)} or~\cite[Thm.~6.1.2]{Toledano(Integrating-unitary-rep-infinite-dim)}, one can use the Segal-Sugawara formulae to extend the induced map $LG \rightarrow \PU(H_{0,k})$ to a continuous projective representation
\begin{equation}\label{eq:rhoLG--}
LG \rtimes \Diff_+(S^1)\,\longrightarrow\,\PU(H_{0,k}).
\end{equation}
Finally, the projective action of $\Diff_+(S^1)$ on $H_{0,k}$ restricts to an honest action of the conformal group $\Conf_+(S^1)\subset \Diff_+(S^1)$, and so one gets a homomorphism
\begin{equation}\label{eq:rhoLG----}
\widetilde{LG} \rtimes \Conf_+(S^1)\,\longrightarrow\,\U(H_{0,k}).
\end{equation}

Note that the infinitesimal generator $L_0$ of the rotation group $S^1 \subset \Conf_+(S^1)$ has positive spectrum, and satisfies the assumption of Theorem~\ref{thm:conf-Hamiltonian+split}:

\begin{lemma} \label{lem:L_0-pos+trace-class}
Let $L_0$ be the operator on $H_{0,k}$ defined by $u_{r_\alpha} = e^{i\alpha L_0}$ (see Definition~\ref{def:positive-energy-nets}), where $r_\alpha \in S^1$ is the anticlockwise rotation by angle $\alpha$.
Then $L_0$ is a positive self-adjoint operator.
Moreover, for every $\beta > 0$, the operator $e^{-\beta L_0}$ is trace class.
\end{lemma}
  
\begin{proof}
Let $L_{0,k}(n)\subset L_{0,k}$ denote the subspace where $S^1$ acts by its $n$th representation, and let $W_{0,k}(n)$ be the corresponding subspace of $W_{0,k}$.
We have to show that $L_{0,k}(n)=0$ for $n<0$, and that
\[
\tr_{H_{0,k}}(e^{-\beta L_0})=\sum_{n\ge 0} \dim(L_{0,k}(n))e^{-\beta n}<\infty.
\]
Clearly, since $L_{0,k}(n)$ is a quotient of $W_{0,k}(n)$, it is enough to show that $W_{0,k}(n)=0$ for $n<0$, and that $\sum_{n} \dim(W_{0,k}(n))e^{-\beta n}$ is summable.

By the Poincar\'e-Birkhoff-Witt theorem, there is an $S^1$-equivariant isomorphism between $W_{0,k}=U\hat {\mathfrak g}_\IC\otimes_{U\hat {\mathfrak g}_+} \IC_{0,k}$
and the symmetric algebra on $\hat {\mathfrak g}_\IC/\hat {\mathfrak g}_+=z^{-1}{\mathfrak g}_\IC[z^{-1}]$.
The rotation group $S^1$ acts by  its $n$th representation on the span of $z^{-n}$.
We therefore have $W_{0,k}(n)=0$ for $n<0$, and
\[
\sum_{n} \dim(W_{0,k}(n))e^{-\beta n}
= \prod_{n\ge 1}\big(1+e^{-n \beta}+e^{-2n \beta}+e^{-3n \beta}+\ldots\,\big)^{\dim(\mathfrak g)},
\]
which converges.
\end{proof}

For an interval $I \subset S^1$, we denote by $L_IG\subset LG$ the subgroup of loops with support in $I$,
and by $\widetilde{L_IG}$ the preimage of $L_IG$ in $\widetilde{LG}$. 
We also write $L_I\mathfrak g$ and $\widetilde{L_I\mathfrak g}$ for the respective Lie algebras.

Consider $\pi^{(2)}:\widetilde{LG}\times \widetilde{LG}\to \U(H_{0,k})$, $\pi^{(2)}(g,h)=\pi(gh)$, with $\pi$ as in~\eqref{eq:rhoLG}.
The following proposition is proven in~\cite[Chapter IV, Proposition 1.3.2]{Toledano(PhD-thesis)}:

\begin{proposition}\label{prop: 1.3.2 of Toledano's PhD}
Let $I, J\subset S^1$ intersect in one point, and let $K$ be their union.
Then $\pi^{(2)}\big(\widetilde{L_IG}\times \widetilde{L_JG}\big)$ is dense in $\pi\big(\widetilde{L_KG}\big)$,
where those two groups are given the subspace topology from $(\U(H_{0,k}),\text{\rm strong})$.
\end{proposition}

\begin{definition}
The \emph{loop group net for $G$ at level $k$} is the positive-energy net $(\cala_{G,k}, H_{0,k})$ given by
\begin{equation*}
\begin{split}
\cala_{G,k}:\,\, &\INT_{S^1} \to\, \VN_{H_{0,k}}\\
&\,\,\,\,I\mapsto \pi\big(\widetilde{L_IG}\big)''
\end{split}
\end{equation*}
along with the action \eqref{eq:rhoLG----} of $\Conf_+(S^1)$ on $H_{0,k}$,
and the unique (up to scalar) fixed vector $\Omega\in H_{0,k}$ for the rotation group $S^1\subset \Conf_+(S^1)$.
\end{definition}
\noindent The axioms of positive-energy nets are verified as follows.
First note that if $I$ and $J\subset S^1$ have disjoint interiors, then the cocycle $c(f,g)$ vanishes for $f\in L_I\mathfrak g$ and $g\in L_J\mathfrak g$.
As a consequence, the Lie algebras $\widetilde{L_I\mathfrak g}$ and $\widetilde{L_J\mathfrak g}$ commute inside  $\widetilde{L\mathfrak g}$,
the subgroups $\widetilde{L_IG}$ and $\widetilde{L_JG}$ commute in $\widetilde{LG}$,
and the subalgebras $\cala_{G,k}(I)$ and $\cala_{G,k}(J)$ commute in $\bfB(H_{0,k})$.
The covariance axiom holds because the action of $\varphi\in\Conf_+(S^1)$ conjugates $\widetilde{L_IG}$ into $\widetilde{L_{\varphi(I)}G}$.
Positive-energy follows from Lemma~\ref{lem:L_0-pos+trace-class}.
Finally, by the classification of unitary positive-energy representations of $\PSL_2(\IR)$ $\cong \Conf_+(S^1)$,
any vector that is fixed by $S^1\subset \Conf_+(S^1)$ is actually fixed by the whole group $\Conf_+(S^1)$.
In particular, $\Omega$ is a fixed vector for $\Conf_+(S^1)$.

\begin{theorem} \label{thm:loop-group-nets}
The positive-energy net $(\cala_{G,k}, H_{0,k})$ satisfies strong additivity, the split property, and diffeomorphism covariance.

Moreover, if $G=\SU(n)$, then the conformal net associated (by Proposition~\ref{prop:pos-energy-net--->conf-net}) to the positive energy net $(\cala_{\SU(n),k}, H_{0,k})$ has finite index (Definition \ref{def:index-for-nets}).
\end{theorem}

\begin{proof}
Strong additivity follows from Proposition \ref{prop: 1.3.2 of Toledano's PhD} and
the split property follows from Lemma \ref{lem:L_0-pos+trace-class} and Theorem~\ref{thm:conf-Hamiltonian+split}.
We now check diffeomorphism covariance.

Let $p:\widetilde{\Diff}\to \Diff_+(S^1)$ be the central extension by $S^1$ pulled back along \eqref{eq:rhoLG--},
and let $q: \Aut (LG) \to \Aut (\widetilde{LG})$ be the isomorphism given by the functoriality of universal central extensions.
The semidirect product $\widetilde{LG}\rtimes\widetilde{\Diff}$ for the action
\[
\widetilde{\Diff} \,\xrightarrow{p}\, \Diff_+(S^1) \,\to\, \Aut (LG) \,\xrightarrow{q}\, \Aut (\widetilde{LG})
\]
then acts on $H_{0,k}$.
One first observes that the action \eqref{eq:rhoLG--} of a diffeomorphism $\varphi$ conjugates $\widetilde{L_IG}$ into $\widetilde{L_{\varphi(I)}G}$,
and therefore $\cala_{G,k}(I)$ into $\cala_{G,k}(\varphi(I))$.
Indeed, at the Lie algebra level, the action of $\varphi\in\Diff_+(S^1)$ on $\widetilde{L\mathfrak g}=L\mathfrak g\oplus \IR$ is simply given by $\varphi\cdot(f,a)=(f\circ \varphi^{-1},a)$.
Let now $I\subset S^1$ be an interval, and let $\varphi$ be a diffeomorphism with support in $I$.
Denote by $I'$ the closure of $S^1\setminus I$.
By Haag duality for positive-energy nets \cite[Thm.~2.19.(ii)]{Gabbiani-Froehlich(OperatorAlg-CFT)},
in order to show that $u_\varphi\in \cala_{G,k}(I)$, it is enough to argue that it commutes with $\cala_{G,k}(I')$.
Equivalently, we have to show that the chosen lift $\tilde\varphi\in \widetilde{\Diff}$ of $\varphi$ commutes with $\widetilde{L_{I'}G}$ inside the group $\widetilde{LG}\rtimes\widetilde{\Diff}$,
i.e., that the action of $p(\tilde\varphi)=\varphi$ on $\widetilde{L_{I'}G}$ is trivial.
The last statement can be verified at the Lie algebra level.

Finally, building on work of Wassermann \cite{Wassermann(Operator-algebras-and-conformal-field-theory)}, Feng Xu proves in~\cite{Xu(Jones-Wassermann-subfactors)} 
that $(\cala_{\SU(n),k}, H_{0,k})$ has finite index.
\end{proof}

\begin{remark} \label{rem:alternative-xu}
For other compact, simple, simply-connected Lie groups, it is expected that the conformal nets $\cala_{G,k}$ have finite index, as is known to be the case for $G=\SU(n)$.
It is also expected that the category of positive energy representations of $LG$ at level $k$ is equivalent to the category of sectors for the corresponding conformal net.\footnote{We 
have been informed that this will, in fact, be a consequence of ongoing work of Carpi-Weiner; see also~\cite{Weiner-(Conformal-covariance-and-positivity-of-energy)}.
For the case $G=\SU(N)$, this result is a consequence of~\cite[Thm. 2.2]{Xu(Jones-Wassermann-subfactors)}, \cite[Thm. 33]{Kawahigashi-Longo-Mueger(2001multi-interval)},
and \cite[Thm. 3.5]{Xu(Jones-Wassermann-subfactors)}.}
However, as far as we know, those problems are still open.

The main theorem of Wassermann~\cite[p.535]{Wassermann(Operator-algebras-and-conformal-field-theory)}
(combined with~\cite[Theorem 4.1]{Longo(Index-of-subfactors-and-statistics-of-quantum-fields-II)})
shows that the category of positive energy representations of $L\SU(n)$ at level $k$ is a fusion category.
Therefore, assuming that the category of positive energy representations of $L\SU(n)$ at level $k$ is equivalent to the category of sectors for the corresponding conformal net,
it would follow that the category $\Rep(\cala_{\SU(n),k})$ is fusion.
An application of Theorem~\ref{thm: characterization of finite conformal nets} would then yield an alternative proof that $\cala_{\SU(n),k}$ has finite index.
\end{remark}
%=============================================================
%=============================================================

\section*{Appendix}
\label{sec: Preliminaries}
\setcounter{theorem}{0}
\setcounter{section}{1}
\renewcommand{\thesection}{\Alph{section}}

\subsection*{von Neumann algebras}
\addtocontents{toc}{\SkipTocEntry}
    Given a Hilbert space $H$, let $\bfB(H)$ denote its algebra of 
    bounded operators. 
    The ultraweak topology on $\bfB(H)$ is the topology of pointwise 
    convergence with respect to the pairing with its predual, 
    the trace class operators.

  \begin{definition}
    A von Neumann algebra, is a topological *-algebra\footnote{{\it Warning:} there is no compatibility between the topology and the algebra structure; the multiplication map $(\bfB(H),\text{ultraweak})\times (\bfB(H),\text{ultraweak})\to (\bfB(H),\text{ultraweak})$ is not continuous.} that is embeddable 
    as closed subalgebra of $\bfB(H)$ with respect to the ultraweak topology.
  \end{definition}

  The spatial tensor product $A_1\bar\otimes A_2$ of von Neumann algebras 
  $A_i\subset \bfB(H_i)$ is the ultraweak 
  closure in $\bfB(H_1\otimes H_2)$ of 
  their algebraic tensor product $A_1\otimes_\alg A_2$.

  \begin{definition}
    Let $A$ be a von Neumann algebra.
    A left (right) $A$-module is a Hilbert space $H$ equipped with a 
    continuous homomorphism from $A$ (respectively $A^\op$) to $\bfB(H)$.
    We will use the notation ${}_AH$ (respectively $H_A$) to denote the 
    fact that $H$ is a left (right) $A$-module.
  \end{definition}

The main distinguishing feature of the representation theory of von Neumann algebras is expressed in the following lemma.
Here, $\ell^2$ stands for the Hilbert space $\ell^2(\IN)$ if all the spaces in the statement of the lemma are separable.
Otherwise, it stands for $\ell^2(X)$, where $X$ is any set of sufficiently large cardinality.

\begin{lemma}
\label{prop:modules-are-summands}
Let $A$ be a von Neumann algebra and let $H$ and $K$ be two faithful left $A$-modules.
Then $H\otimes \ell^2\cong K\otimes \ell^2$.
In particular, any $A$-module is isomorphic to a direct summand of $H\otimes \ell^2$.
\end{lemma}

\subsection*{The Haagerup $L^2$-space}
\addtocontents{toc}{\SkipTocEntry}
  (See~\cite[\S 2]{BDH(Dualizability+Index-of-subfactors)} for further details.) A faithful left module $H$ for a von Neumann algebra $A$ is called a 
  \emph{standard form} if it comes equipped with 
  an antilinear isometric involution $J$ and a selfdual cone $P\subset H$ 
  subject to the properties
  \begin{enumerate}
	\item $J A J = A'$ on $H$,
	\item $J c J = c^*$ for all $c \in Z(A)$,
	\item $J \xi = \xi$ for all $\xi \in P$,
	\item $a J a J (P) \subseteq P$ for all $a \in A$
  \end{enumerate}
  where $A'$ denotes the commutant of $A$.
  The standard form is unique up to unique unitary 
  isomorphism~\cite{Haagerup(1975standard-form)}.
  It is an $A$-$A$-bimodule, with right action $\xi a := J a^* J \xi$.

  The space of continuous linear functionals $A\to \IC$ forms a Banach 
  space $A_*=L^1(A)$ called the predual of $A$.
  It comes with a positive cone 
  $L^1_+(A):=\{\phi\in A_*\,|\,\phi(x)\ge0\,\, \forall x\in A_+\}$ 
  and two commuting $A$-actions given by $(a\phi b)(x):=\phi(bxa)$.
  Given a von Neumann algebra $A$, its Haagerup $L^2$-space is an 
  $A$-$A$-bimodule that is canonically associated to 
  $A$ \cite{Kosaki(PhD-thesis)}.
  It is the completion of 
  $$\bigoplus_{\phi\in L^1_+(A)} \IC\textstyle\sqrt{\phi}$$
  with respect to some pre-inner product, and is denoted $L^2(A)$.
  The positive cone in $L^2A$ is given by 
  $L^2_+(A):=\{\sqrt\phi\,\,|\,\phi\in L^1_+(A)\}$.
  The space $L^2A$ is also equipped with the modular conjugation $J_A$ 
  that sends $\lambda\sqrt{\phi}$ to $\bar\lambda\sqrt{\phi}$ for 
  $\lambda\in\IC$, and satisfies
  \begin{equation}\label{eq: main property of J}
   J_A(a\xi b)=b^*J_A(\xi)a^*.
  \end{equation}
  All together, the triple $(L^2(A),J_A,L^2_+(A))$ is a standard form for 
  the von Neumann algebra $A$.

  \begin{remark} \label{rem: L^2(A^op)}
    There is an isomorphism $L^2(A) \cong L^2(A^\op)$ under which the 
    left action of $A$ on $L^2A$
    corresponds to the right action of $A^\op$ on $L^2(A^\op)$, 
    and the right action of $A$ on $L^2A$
    corresponds to the left action of $A^\op$ on $L^2(A^\op)$.
  \end{remark}

  \begin{remark} \label{rem:L^2(iso)}
    The assignment $A\mapsto L^2(A)$ defines a functor from the category 
    of factors and isomorphisms,
    to the category of Hilbert spaces and bounded linear maps.
    (This still true for the larger category whose morphisms are 
    finite homomorphisms between 
    factors~\cite[Thm 6.7]{BDH(Dualizability+Index-of-subfactors)}, 
    but in the present paper
    we only need this functor for isomorphisms.)
  \end{remark}

\subsection*{Connes fusion}
\addtocontents{toc}{\SkipTocEntry}
(See~\cite[\S 3]{BDH(Dualizability+Index-of-subfactors)} for further details.)
\begin{definition}
Given two modules $H_A$ and ${}_AK$ over a von Neumann algebra $A$,
their Connes fusion $H\boxtimes_A K$ is the completion 
\cite{Connes(Geometrie-non-commutative), Sauvageot(Sur-le-produit-tensoriel-relatif), Wassermann(Operator-algebras-and-conformal-field-theory)} of
\begin{equation}\label{eq:def of CFus}
\mathrm{Hom}\big(L^2(A)_A,H_A\big)\otimes_A L^2(A)\otimes_A \mathrm{Hom}\big({}_AL^2(A),{}_AK\big)
\end{equation}
with respect to the inner product
$\big\langle\phi_1\otimes \xi_1\otimes \psi_1,\,\phi_2\otimes \xi_2\otimes \psi_2\big\rangle:=
\big\langle(\phi_2^*\phi_1)\xi_1(\psi_1\psi_2^*),\xi_2\big\rangle$.
Here, we have written the action of $\psi_i$ on the right, which means that
$\psi_1\psi_2^*$ stands for the composite $L^2(A)\xrightarrow{\psi_1} K\xrightarrow{\psi_2^*} L^2(A)$.
\end{definition}

The $L^2$ space is a unit for Connes fusion in the sense that there are canonical unitary isomorphisms 
\begin{equation}\label{eq: unitality of CFusion}
{}_A L^2(A) \boxtimes_A H \cong {}_A H \qquad \text{and} \qquad H\boxtimes_A L^2(A)_A \cong H_A.
\end{equation}

\subsection*{Dualizability}
\addtocontents{toc}{\SkipTocEntry}
(See~\cite[\S 4]{BDH(Dualizability+Index-of-subfactors)} for further details.) A von Neumann algebra whose center is $\IC$ is called a \emph{factor}.

\begin{definition}\label{def:dual}
For $A$ and $B$ factors, given an $A$-$B$-bimodule $H$, we say that a $B$-$A$-bimodule $\bar H$ is dual to $H$ if it comes equipped with maps
\begin{equation}\label{eq:duality maps}
R:{}_AL^2(A)_A \rightarrow {}_AH\boxtimes_B \bar H_A\qquad\quad
S:{}_BL^2(B)_B \rightarrow  {}_B\bar H\boxtimes_A H_B
\end{equation}
subject to the duality equations $(R^*\otimes 1)(1\otimes S)=1$, $(S^*\otimes 1)(1\otimes R)=1$, and to
the normalization ${R^*(x\otimes 1)R} = {S^*(1\otimes x)S}$ for all $x\in \mathrm{End}({}_AH_B)$.
A bimodule whose dual module exists is called \emph{dualizable}.
\end{definition}

If ${}_AH_B$ is a dualizable bimodule, then its dual bimodule is well defined up to canonical unitary isomorphism~\cite[Thm 4.22]{BDH(Dualizability+Index-of-subfactors)}.
Moreover, the dual bimodule is canonically isomorphic to the complex conjugate Hilbert space $\overline H$, with the actions $b\bar \xi a:= \overline{a^* \xi b^*}$~\cite[Cor 6.12]{BDH(Dualizability+Index-of-subfactors)}.

\begin{lemma}[{\cite[Lemma 4.6]{BDH(Dualizability+Index-of-subfactors)}}]\label{lem: characterization of duals}
Let ${}_AH_B$ and ${}_BK_A$ be dualizable irreducible bimodules. Then
\[
\mathrm{Hom}_{A,A}\big(H\boxtimes_BK,L^2(A)\big)=\begin{cases}
\IC&\text{if ${}_BK_A\cong{}_B\bar H_A$}\\
0&\text{otherwise}\\
\end{cases}
\]
\end{lemma}

\begin{lemma}[{\cite[Lemma 4.10]{BDH(Dualizability+Index-of-subfactors)}}]\label{lem: endolgebra dim<oo}
If ${}_AH_B$ is a dualizable bimodule, then its algebra of 
$A$-$B$-bilinear endomorphisms is finite-dimensional.
\end{lemma}

\subsection*{Statistical dimension and minimal index}
\addtocontents{toc}{\SkipTocEntry}
(See~\cite[\S 5]{BDH(Dualizability+Index-of-subfactors)} for further details.)
\begin{definition}\label{def: stat dim}
The statistical dimension of a dualizable bimodule ${}_AH_B$ is given by
\begin{equation*}
\qquad\dim ({}_A H_B) := R^*R = S^* S\,\in\,\IR_{\ge 0}
\end{equation*}
where $R$ and $S$ are as in \eqref{eq:duality maps}.
For non-dualizable bimodules, one declares $\dim({}_AH_B)$ to be $\infty$.
\end{definition}

Note that from the definition, it is obvious that 
\begin{equation}\label{eq: dim (A H_B)= dim (B H_A)}
\dim ({}_A H_B)=\dim ({}_B \bar H_A).
\end{equation}
The minimal index $[B:A]$ of an inclusion of factors 
${\iota:A \to B}$ is the square of the statistical dimension of ${}_AL^2B_B$.
If ${}_AH_B$ is a faithful bimodule between factors, 
then we have $[B':A]=[A':B]=\dim({}_AH_B)^2$.

\begin{definition}\label{def: finite homomrphism}
Let $\iota:A\to B$ be an inclusion of factors.
If the minimal index $[B:A]$ is finite, we say that $\iota$ is a \emph{finite} homomorphism.
\end{definition}

As a corollary of Lemma \ref{lem: endolgebra dim<oo}, we have:

\begin{lemma}[{\cite[Lemma 5.15]{BDH(Dualizability+Index-of-subfactors)}}]\label{lem: rel comm is finite dim}
Let $\iota:A\to B$ be a finite homomorphism between factors.
Then the relative commutant of $\iota(A)$ in $B$ is finite-dimensional.
\end{lemma}

\subsection*{Haagerup's  $u$-topology.}
\addtocontents{toc}{\SkipTocEntry}
Given von Neumann algebras $A$ and $B$, the $u$-topology on $\mathrm{Hom}(A,B)$
is defined by declaring that a generalized sequence $\{\varphi_i\}$ converges to $\varphi\in\mathrm{Hom}(A,B)$
if for every $\xi\in L^1(B)$, we have $\lim_i(\xi\circ\varphi_i) = \xi\circ\varphi$ in $L^1(A)$.
Equivalently, it is the topology generated by the semi-norms
\begin{equation*}
\varphi \,\mapsto\, \| \xi \circ \varphi \|_{L^1(A)} = \sup_{a \in A, \|a\| \le 1} | \xi( \varphi(a)) | 
\end{equation*}
for $\xi \in L^1(B)$.

The subgroup $\mathrm N(A) := \{ u \in \U(L^2(A)) \,|\, uAu^* = A \}\subset \U(L^2(A))$
is closed for the strong (=\,weak) topology on $\U(L^2(A))$.

\begin{lemma}\label{lem:ad-is-cont}
The adjoint map $\Ad \colon \mathrm N(A) \to \Aut(A)$, $\Ad(u)(a)=uau^*$, is continuous for the strong topology on $\mathrm N(A)$ and the $u$-topology on $\Aut(A)$.
\end{lemma}

\begin{proof}[Proof.\footnotemark]
\footnotetext{Courtesy of \sf{http://mathoverflow.net/questions/87324/}}
Given $\xi \in L^1_+(A)$, we need to show that
\[
\begin{split}
f_\xi \,\,\colon\,\, \mathrm N(A&) \to\,\,\,\, \IC\\
u \,&\mapsto \sup_{a \in A, \|a\| \le 1} | \xi( u a u^*) |
\end{split}
\]
is continuous for the strong topology on $\mathrm N(A)$.
Let $u_n \to u$ be a convergent sequence in $\mathrm N(A)$.
For every $v\in \mathrm N(A)$, we have 
\[
\textstyle \xi (v a v^*) =\big\langle v a v^* \sqrt{\xi}, \sqrt{\xi}\, \big\rangle_{L^2(A)} = \big\langle a v^* \sqrt{\xi},v^* \sqrt{\xi}\, \big\rangle_{L^2(A)}.
\]
Therefore, given $a$ in the unit ball of $A$, we have
\begin{eqnarray*}
\lefteqn{ \big| \xi (u_n a u_n^*) - \xi(u a u^*) | } & & \\ & = & \textstyle
| \langle a u_n^* \sqrt{\xi},(u_n-u)^* \sqrt{\xi} \rangle_{L^2(A)}  +  \langle a (u_n-u)^* \sqrt{\xi},u^* \sqrt{\xi} \rangle_{L^2(A)} \big|
\\
& \leq & \textstyle 2 \cdot\| \sqrt{\xi} \|_{L^2(A)} 
\cdot \| (u_n-u)^* \sqrt{\xi} \|_{L^2(A)}. 
\end{eqnarray*}
Since $u_n^*\to u^*$ in the strong topology, 
we have $\lim_n \| (u_n-u)^* \sqrt{\xi} \|_{L^2(A)} = 0$,
and so $\xi (u_n a u_n^*)$ converges to $\xi (u a u^*)$ uniformly in the unit ball of $A$.
\end{proof}

The functoriality of $L^2$ yields a map $\Aut(A) \to \mathrm N(A)$, $\psi \mapsto L^2(\psi)$.
Haagerup calls this the canonical implementation.
He also shows~\cite[Prop.~3.5]{Haagerup(1975standard-form)} that it exhibits $\Aut(A)$ with the $u$-topology as a closed subgroup of $\mathrm N(A)$ with the strong topology.

\begin{proposition}
\label{prop:subspace-topology-on-PU(A_0)}
Let $A_0 \subseteq A$ be a subfactor.
Assume that the action of the algebraic tensor product $A_0 \ox_\alg A'$ on $L^2(A)$ extends to the spatial tensor product $A_0 \,\bar{\otimes}\, A'$.
Then $\Ad \colon \U(A_0) \to \Aut(A)$ induces a homeomorphism from $\PU(A_0)$ with the quotient strong topology onto its image in $\Aut(A)$ with the $u$-topology.
\end{proposition}

\begin{proof} 
Recall that the canonical implementation $\psi \mapsto L^2(\psi)$ identifies $\Aut(A)$ with a closed subgroup of $\U(L^2(A))$. 
For $u \in \U(A)$, the canonical implementation of $\Ad(u) \in \Aut(A)$ is $L^2(\Ad(u)) = u J u^* J$, where $J$ is the modular conjugation on $L^2(A)$.
Thus it suffices to show that
\[
\begin{split}
f \colon \U(A_0) \,&\to \U(L^2(A)),\\
u \,&\mapsto\, u J u^* J
\end{split}
\]
descends to a homeomorphism $\bar f \colon \PU(A_0)\to f(\U(A_0))$.

The map $\bar{f}$ is well-defined, bijective, and continuous by Lemma~\ref{lem:ad-is-cont}. 
It remains to show that $\bar{f}^{-1}$ is continuous.
Let $u_n$ and $u$ be elements of $\U(A_0)$ so that $\lim_n f(u_n) = f(u)$ in $\U(L^2(A))$.
We would like to know that $\lim_n [u_n] = [u]$ in $\PU(A_0)$.
For that, it is enough to show that there exist $\lambda_n \in S^1$ such that $\lim_n\lambda_n u_n = u$ in $\U(A_0)$.

Since the algebra generated by $A_0$ and $A'$ on $L^2(A)$ is their spatial tensor product,
we can identify $f(u_n)$ and $f(u)$ with the elements $u_n \otimes J u_n^* J$ and $u \otimes J u^* J$ of $\U(A_0 \bar{\ox} A)$.
The existence of the $\lambda_n$ then follows from Lemma~\ref{lem:u_n-ox-v_n--->u-ox-v}.
\end{proof}

\begin{lemma} \label{lem:u_n-ox-v_n--->u-ox-v}
Let $A$ and $B$ be factors.
Let $\{u_n\}$ be a sequence in $\U(A)$, and $\{v_n\}$ a sequence in $\U(B)$ so that
\[
\qquad\qquad\qquad\lim_n (u_n \otimes v_n) = u \otimes v\qquad\quad\text{in $\,\U(A \,\bar{\otimes}\, B)$}
\]
for given $u \in \U(A)$ and $v \in \U(B)$.
Then there exist $\lambda_n \in S^1$ so that $\lim_n\lambda_n u_n = u$ and $\lim_n \lambda_n^{-1} v_n = v$.
\end{lemma}

\begin{proof}
The identity $\lim_n (u_n \otimes v_n) = u \otimes v$ is equivalent to $\lim_n (u^*u_n \otimes v^*v_n) = 1 \otimes 1$,
so we may assume that $u = 1$ and $v = 1$.
Pick a faithful representation $H$ of $A$, and a unit vector $\xi \in H$.
Replacing $u_n$ by $\lambda_n u_n$ and $v_n$ by $\lambda_n^{-1} v_n$ for appropriate $\lambda_n \in S^1$,
we may also assume that $\langle u_n \xi \,|\, \xi \rangle \geq 0$.

Denote by $A_1$ and $B_1$ the unit balls in $A$ and $B$. 
These are compact in the weak (= ultraweak) topology, and so it is enough to show that the limit of any weakly convergent subsequence of $\{u_n\}$ is equal to $1$, and the same for $\{v_n\}$.
We can therefore assume that $\hat{u}:=\lim_n u_n$ and and $\hat{v}:=\lim_n v_n$ exist.
The product map $A_1 \x B_1 \to (A \,\bar{\ox}\, B)_1$ is continuous for the weak topology.\footnote{This fails if one replaces $A \,\bar{\ox}\, B$ by some other completion of $A \otimes_\alg B$.}
It follows that $\hat{u} \otimes \hat{v} = 1 \otimes 1$, and so $\hat{u} = \lambda$ and $\hat{v} = \lambda^{-1}$ for some $\lambda \in S^1$.
Finally, $\lim_n \langle u_n \xi \,|\, \xi \rangle = \langle \hat{u} \xi \,|\, \xi \rangle = \lambda$ is positive, and so $\lambda = 1$.
\end{proof}

Given two von Neumann algebras $A$ and $B$, recall that $\mathrm{Hom}_{\VN}(A,B)$ denotes the space
of homomorphisms and antihomomorphisms from $A$ to $B$.
That set is topologized as the disjoint union of $\mathrm{Hom}(A,B)$ and $\mathrm{Hom}(A,B^\op)$,
where both of those Hom sets are given the Haagerup $u$-topology.

%==========================================================================
%==========================================================================

\bibliographystyle{abbrv}
\bibliography{../Files/db-cn3}

\end{document}